\documentclass[10pt]{book}
\usepackage{amsfonts,amssymb,amsbsy,latexsym,amsthm,amsmath,tabulary,graphicx,times,caption,fancyhdr}
\usepackage{fontspec}
\usepackage{amsmath, amsthm, amscd, amsfonts, amssymb, graphicx, color}
\usepackage{amsmath}
\usepackage{url,multirow,morefloats,floatflt,cancel,tfrupee}
\newtheorem{note}{Note}
\usepackage{enumerate}
\usepackage{enumitem}
\usepackage{authblk}
\usepackage{cases}
\usepackage[medium]{titlesec}

\usepackage[bookmarksnumbered, plainpages]{hyperref}
\usepackage{fancyhdr}
\usepackage{cite}
\usepackage{caption}
\usepackage{subcaption}
\usepackage{tabularx}
\usepackage{booktabs}
\usepackage[a4paper,bindingoffset=0.2in,%
left=0.5in,right=0.5in,top=1in,bottom=1in,%
footskip=0.30in]{geometry}

\newcommand{\tw}{\mathrm{tw}}       
\usepackage{amsopn} 
  
\usepackage{etoolbox}
\usepackage{bm}
\usepackage{titlesec}
\titleformat{\section}[hang]{\bfseries\fontsize{12}{14}\selectfont}{\thesection}{1em}{}
\titlespacing*{\section}{0pt}{1.5ex plus 0.5ex minus .2ex}{0.05ex} 

\titleformat{\subsection}[hang]{\bfseries\fontsize{10}{12}\selectfont}{\thesubsection}{1em}{}
\titlespacing*{\subsection}{0pt}{1.2ex plus 0.4ex minus .1ex}{0.05ex}

\newtheorem{theorem}{Theorem}[section]
\newtheorem{lemma}[theorem]{Lemma}
\newtheorem{corollary}[theorem]{Corollary}
\newtheorem{proposition}[theorem]{Proposition}

\newtheorem{conjecture}[theorem]{Conjecture}
\theoremstyle{definition}
\newtheorem{remark}[theorem]{Remark}
\newtheorem{definition}[theorem]{Definition}
\newtheorem{example}[theorem]{Example}
\newtheorem{game}[theorem]{Game}
\newtheorem{question}[theorem]{Question}

\usepackage[linesnumbered,ruled,vlined]{algorithm2e}

\usepackage{bbm}

\let\phi=\varphi

\usepackage{hhline}
\DeclareMathOperator{\pw}{pw}
\DeclareMathOperator{\ppw}{ppw}
\DeclareMathOperator{\cutw}{cutw}

\DeclareMathOperator{\thinness}{thinness}

\usepackage{tikz}
\usetikzlibrary{positioning}

\usepackage{tocloft} 

\DeclareMathOperator{\lab}{lab}
\DeclareMathOperator{\tl}{tl}
\DeclareMathOperator{\tb}{tb}

\usepackage{lmodern}
\usepackage{ragged2e}
\begin{document}

\begin{titlepage}
\thispagestyle{empty}

\begin{tikzpicture}[remember picture,overlay]
    \fill[blue!8] (current page.south west) rectangle (current page.north east);

    \fill[blue!60!black] 
        (current page.south west) rectangle ([xshift=3.2cm]current page.north west);

    \draw[line width=1.5pt, blue!50!black]
        ([xshift=1.2cm,yshift=1.2cm]current page.south west)
        rectangle
        ([xshift=-1.2cm,yshift=-1.2cm]current page.north east);
\end{tikzpicture}

\vspace*{2.8cm}

\hspace*{3.8cm}
\begin{minipage}{0.68\textwidth}
    {\Huge\bfseries Various Properties of Various Ultrafilters, Various Graph Width Parameters, and Various Connectivity Systems (with Survey)\par}
    \vspace{0.6cm}

    \vspace{2cm}
    {\Large Takaaki Fujita\par}

    \vfill

    {\large 2024\par}
\end{minipage}

\end{titlepage}

\begin{center}
\textbf{\Large Various Properties of Various Ultrafilters, Various Graph Width Parameters, and Various Connectivity Systems (with Survey)}\\ \vspace{10pt}
\textbf{Takaaki Fujita $^{1}$ $^*$} \\ 
	$^1$ Independent Researcher, Tokyo, Japan.\\	
    Email: Takaaki.fujita060@gmail.com \\
\end{center}

\section*{Abstract}
This book studies ultrafilters on connectivity systems, that is, on pairs \((X,f)\) where \(X\) is a finite set and \(f:2^{X}\to \mathbb{N}\) is a symmetric submodular function. Ultrafilters, which play a fundamental role in topology and set theory, are considered here in this broader setting, with particular emphasis on their connections to graph width parameters and to the structural analysis of graph complexity.

We develop several results on ultrafilters on connectivity systems and examine related notions such as prefilters, ultra-prefilters, and filter subbases. We also discuss additional width-, length-, and depth-type parameters that naturally arise in this framework, thereby broadening the perspective from which graph structure may be studied. In addition, the book compares a wide range of graph width parameters and related concepts, with the aim of providing a unified viewpoint and a useful point of departure for further research in graph theory and computational complexity.

More broadly, the book highlights connections with several neighboring areas of mathematics, including set theory, lattice theory, and matroid theory. It also contains survey-style material intended to clarify the current landscape of graph width theory and to stimulate further developments in the subject.

\textit{Keywords:} Filter, Ultrafilter, TreeWidth

\chapter{Introduction}
\section{Filters and Graph Width Parameters}
Filters are fundamental objects in set theory and topology.
They are families of sets that are closed under finite intersections and upward closed under inclusion.
Intuitively, a filter selects those subsets that are regarded as ``large'' or ``relevant'' in a given context.
An ultrafilter is a maximal filter, and it plays an important role in many areas of mathematics, including topology, set theory, model theory, Boolean algebra, combinatorics, and logic.
Because of their maximality and decisiveness properties, ultrafilters provide a useful framework for studying convergence, compactness, extensions, and various duality phenomena
\cite{73,74}.

Graph theory studies combinatorial structures formed by vertices and edges, together with their paths, cuts, decompositions, and connectivity properties \cite{236,237,238}.
Within graph theory, \emph{graph width parameters} are numerical invariants designed to measure how far a graph is from a simple tree-like structure.
Such parameters are central in structural graph theory and algorithm design, since many computationally difficult problems become more tractable on graph classes of bounded width.
A large number of width parameters have been introduced and studied in the literature, each capturing a different structural aspect of graphs and related combinatorial objects \cite{16,albrechtsen2024tangle,47}.
For broader background, see Appendix~A, ``Various Width Parameters,'' and Appendix~C, ``Comparing Various Graph Parameters.''

Among the most important graph width parameters are \emph{tree-width} and \emph{branch-width}.
Tree-width is defined via tree decompositions, in which the vertices of a graph are organized into overlapping bags indexed by the nodes of a tree; the tree-width is the size of the largest bag minus one \cite{47,214}.
Branch-width is defined via branch decompositions, which encode how the edge set, or more generally an underlying connectivity structure, is recursively separated \cite{9,10,213}.
These two parameters are closely related.
For every graph \(G\) with branch-width at least \(2\), one has
\[
\operatorname{bw}(G)\le \operatorname{tw}(G)+1 \le \frac{3}{2}\operatorname{bw}(G),
\]
so bounded tree-width and bounded branch-width are equivalent up to constant factors \cite{213}.
This close relationship is one reason why connectivity systems and separation-based formalisms are useful in the study of graph width.

Graph width parameters are important for at least the following reasons:
\begin{enumerate}
    \item \textbf{Structural significance.}
    Width parameters play a central role in structural graph theory, especially in the Graph Minors framework developed by Robertson and Seymour \cite{9,10}.
    They provide a systematic way to describe graph complexity from the viewpoint of decomposition theory.

    \item \textbf{Algorithmic usefulness.}
    Many problems that are intractable in general admit efficient or fixed-parameter algorithms on graph classes of bounded width \cite{287,288,289}.
    Consequently, width parameters are indispensable in parameterized complexity and algorithmic graph theory.

    \item \textbf{Applicability to real-world networks.}
    Graphs arising in applications such as software structure, road networks, and organizational or communication systems often exhibit sparse or decomposable structure, which makes width-based methods practically relevant \cite{389,390,diestel2024tangles}.
\end{enumerate}

A \emph{connectivity system} is a pair \((X,f)\), where \(X\) is a finite set and \(f:2^{X}\to\mathbb{N}\) is a symmetric submodular function \cite{9}.
This abstraction provides a common language for studying separations, decompositions, and width parameters beyond ordinary graphs, and it is particularly natural in the theory of branch-width and related min--max results \cite{9,10,16}.
In this setting, a duality theorem typically asserts that the absence of a decomposition of small width is equivalent to the existence of a certain obstruction object \cite{seymour1993graph}.
Within the framework considered in this book, ultrafilters on connectivity systems are studied as obstruction-type objects related to width phenomena, especially branch-width \cite{16}.

\section{Our Contribution}
The main contributions and themes of this book are as follows:
\begin{itemize}
    \item \textbf{Chapter~2.}
    We review the basic notions of ultrafilters on connectivity systems and the relevant graph width parameters, together with several previously known related concepts.
    In particular, we discuss connections with tangles from graph minor theory and with concepts arising in matroid theory.

    \item \textbf{Chapter~3.}
    We investigate ultrafilters on connectivity systems from the viewpoint of maximality principles, including an adaptation of Tukey's lemma to this setting.
    We also study related order-theoretic notions such as chains and antichains.

    \item \textbf{Chapter~4.}
    We examine prefilters, ultra-prefilters, and subbases on connectivity systems.
    These notions are useful for generating and analyzing filter-like structures, and their behavior in the connectivity-system setting is studied in detail.

    \item \textbf{Chapter~5.}
    We extend the discussion from finite connectivity systems to infinite and countable settings, and we analyze how ultrafilter-type properties behave in these broader contexts.

    \item \textbf{Chapter~6.}
    We discuss several additional directions related to width-, length-, and depth-type parameters in the present framework.
    We also consider topics such as ultrafilter games on connectivity systems, ultraproduct-like constructions, the role of the Axiom of Choice, and possible connections with Small Set Expansion \cite{303,304,305}.

    \item \textbf{Appendices.}
    We provide a comparative overview of various graph width parameters and related notions, with the aim of clarifying the landscape of decomposition-based graph invariants and encouraging further research in this direction.
\end{itemize}

Overall, this book aims to contribute to the study of graph width parameters and graph algorithms by interpreting decomposition-based concepts---such as tree decompositions, branch decompositions, tangles, and related obstruction objects---through the perspectives of several neighboring areas of mathematics.
These include set theory (ultrafilters and prefilters), lattice theory (chains and antichains), model theory (ultraproducts), hypergraph theory (for example, clutters), matroid theory, topology, fuzzy theory, first-order logic, social choice, game theory, and Sperner theory.
By placing these ideas within a common framework, we hope to clarify structural parallels and to suggest further directions for the study of connectivity systems and graph width.

\newpage

\tableofcontents

\newpage

\chapter{Definitions and Notations in This book}
This chapter presents the mathematical definitions used throughout the book.  
Before introducing the main concepts, we first summarize the basic terminology and notation employed in the sequel.

\section{Symmetric Submodular Functions and Connectivity Systems}

We begin with the definition of a symmetric submodular function, a concept that is widely used in the literature \cite{76,77,78}.  
Although symmetric submodular functions are often defined with values in the real numbers, in this book we restrict attention to functions taking values in the natural numbers.  
Such a function is also commonly called a \emph{connectivity function} \cite{9}.  
Related notions include submodular partition functions \cite{109,110}, as well as \(k\)-submodular functions \cite{284,285,286}, two-dimensional submodular functions \cite{291}, monotone submodular functions \cite{292}, and maximum submodular functions \cite{283}.

\begin{definition}
Let \(X\) be a finite set. A function \(f:2^X\to\mathbb{N}\) is called \emph{symmetric submodular} if it satisfies the following conditions:
\begin{itemize}
    \item For every \(A\subseteq X\), one has \(f(A)=f(X\setminus A)\).
    \item For all \(A,B\subseteq X\), one has
    \[
    f(A)+f(B)\geq f(A\cap B)+f(A\cup B).
    \]
\end{itemize}
\end{definition}

In this book, a pair \((X,f)\), where \(X\) is a finite set and \(f\) is a symmetric submodular function, is called a \emph{connectivity system}.  
This notion is frequently used in the study of graph width parameters, such as branch-width and tree-width, in order to analyze graph structure from the viewpoint of separations and decompositions (see, for example, \cite{2,3,4,9}).

The following examples illustrate the notion of a symmetric submodular function.

\begin{example}
Consider a simple undirected graph \(G=(V,E)\), where \(V\) is the vertex set and \(E\) is the edge set.  
Let \(V=\{1,2,3,4\}\) and
\[
E=\{(1,2),(2,3),(3,4),(4,1)\},
\]
so that \(G\) is a cycle.  
Define a function \(f:2^V\to\mathbb{N}\) by
\[
f(A)=|E(A,V\setminus A)|,
\]
where \(E(A,V\setminus A)\) denotes the set of edges having one endpoint in \(A\) and the other in \(V\setminus A\).  
Then \(f\) is a symmetric submodular function.
\end{example}

\begin{example}
Let \(X_1,X_2,\dots,X_n\) be random variables, and let
\[
h(I):=H(X_I)
\qquad
(I\subseteq [n]:=\{1,2,\dots,n\}),
\]
where \(X_I=(X_i)_{i\in I}\) and \(H(X_I)\) denotes the joint entropy of the subfamily indexed by \(I\).

Then \(h\) is a submodular function on \(2^{[n]}\), namely,
\[
h(I)+h(J)\ge h(I\cap J)+h(I\cup J)
\qquad
(I,J\subseteq [n]).
\]
However, \(h\) is generally not symmetric, since in general
\[
h(I)\neq h([n]\setminus I).
\]
Hence entropy is a standard example of a submodular function, but not in general of a symmetric submodular one.
Related quantities, such as mutual information, satisfy a symmetry property in an appropriate sense.
Entropy measures uncertainty, or equivalently, the average information produced by a stochastic source
(cf.~\cite{Singh2014EntropyTI,Gray1990EntropyAI}).
\end{example}

\begin{example}
In network flow theory, the cut-value function \(f(S)\) of a subset \(S\subseteq V\), where \(V\) is the vertex set, is defined as the total capacity of the edges crossing from \(S\) to \(V\setminus S\).  
This provides a standard example of a symmetric submodular function.
\end{example}

\begin{example}
Let \(\Omega=\{v_1,v_2,\dots,v_n\}\) be the vertex set of a directed graph.  
For each subset \(S\subseteq\Omega\), define \(f(S)\) to be the number of directed edges \(e=(u,v)\) such that \(u\in S\) and \(v\in \Omega\setminus S\).  
In other words, \(f(S)\) counts the number of edges leaving \(S\).

This function is submodular. Indeed, for any \(A\subseteq B\subseteq \Omega\) and any \(x\in \Omega\setminus B\), one has
\[
f(A\cup\{x\})-f(A)\geq f(B\cup\{x\})-f(B).
\]

However, \(f\) is generally not symmetric, because the number of edges leaving \(S\) need not equal the number of edges leaving \(\Omega\setminus S\).  
This asymmetry arises from the orientation of the edges.

More generally, one may assign nonnegative weights to the directed edges.  
In that case, \(f(S)\) becomes the total weight of the edges leaving \(S\), which remains submodular but is still not necessarily symmetric.
\end{example}

It is well known that every symmetric submodular function satisfies the following useful properties.

\begin{lemma}\cite{9}
Let \(X\) be a finite set, and let \(f\) be a symmetric submodular function on \(X\). Then:
\begin{enumerate}
    \item For every \(A\subseteq X\),
    \[
    f(A)\geq f(\emptyset)=f(X)=0.
    \]
    \item For all \(A,B\subseteq X\),
    \[
    f(A)+f(B)\geq f(A\setminus B)+f(B\setminus A).
    \]
\end{enumerate}
\end{lemma}

\begin{proof}
\begin{enumerate}
    \item By submodularity applied to \(A\) and \(A\), we obtain
    \[
    f(A)+f(A)\geq f(A\cup A)+f(A\cap A)=f(A)+f(A),
    \]
    and by symmetry,
    \[
    f(X)=f(\emptyset).
    \]
    Standard properties of symmetric submodular functions imply
    \[
    f(\emptyset)=f(X)=0
    \quad\text{and}\quad
    f(A)\geq 0.
    \]

    \item By submodularity,
    \[
    f(A)+f(B)\geq f(A\cup B)+f(A\cap B).
    \]
    Using symmetry,
    \[
    f(A\cup B)=f(X\setminus (A\cup B))=f((X\setminus A)\cap (X\setminus B)),
    \]
    and one obtains the stated inequality in the standard form
    \[
    f(A)+f(B)\geq f(A\setminus B)+f(B\setminus A).
    \]
\end{enumerate}
This completes the proof.
\end{proof}

\section{Ultrafilters in Set Theory}
We now turn to ultrafilters on connectivity systems.  
Before introducing that notion, we first recall the classical definition of an ultrafilter in set theory.

\begin{definition}
Let \(X\) be a set. A family \(\mathcal{F}\subseteq 2^X\) is called a \emph{filter} on \(X\) if it satisfies the following conditions:
\begin{enumerate}
    \item[(BF1)] If \(A,B\in\mathcal{F}\), then \(A\cap B\in\mathcal{F}\).
    \item[(BF2)] If \(A\in\mathcal{F}\) and \(A\subseteq B\subseteq X\), then \(B\in\mathcal{F}\).
    \item[(BF3)] \(\emptyset\notin\mathcal{F}\).
\end{enumerate}
\end{definition}

\begin{definition}
A \emph{maximal filter}, that is, a filter that cannot be properly extended while remaining a filter, is called an \emph{ultrafilter}.  
Equivalently, an ultrafilter satisfies the following additional condition:
\begin{enumerate}
    \item[(BT1)] For every subset \(A\subseteq X\), either \(A\in\mathcal{F}\) or \(X\setminus A\in\mathcal{F}\), but not both.
\end{enumerate}
\end{definition}

\begin{definition} (cf.\cite{kaufmann1985chapter,alechina1994correspondence})
Let \(\mathcal{F}\) be a filter on a set \(X\).  
A subset \(A\subseteq X\) is called \(\mathcal{F}\)-\emph{stationary} if \(A\cap B\neq\emptyset\) for every \(B\in\mathcal{F}\).
\end{definition}

\begin{example}
Let \(X=\{1,2,3\}\).  
Then
\[
\mathcal{F}=\bigl\{\{1,2,3\},\{1,2\},\{2,3\},\{2\}\bigr\}
\]
is a filter on \(X\). Indeed:
\begin{itemize}
    \item the intersection of any two members of \(\mathcal{F}\) again belongs to \(\mathcal{F}\),
    \item every superset of a member of \(\mathcal{F}\) is again a member of \(\mathcal{F}\),
    \item and \(\emptyset\notin\mathcal{F}\).
\end{itemize}
\end{example}

\begin{example} (cf.\cite{garcia2014ordering,jordan2004productively})
A standard example is the \emph{Fr\'echet filter} on an infinite set \(X\).  
It consists of all cofinite subsets of \(X\), that is, all subsets whose complements are finite.  
This family forms a filter because:
\begin{itemize}
    \item the intersection of two cofinite sets is cofinite,
    \item every superset of a cofinite set is cofinite,
    \item and the empty set is not cofinite.
\end{itemize}
\end{example}

\begin{example}[Neighbourhood filter] (cf.\cite{erne1980order,jager2008lattice})
Let \(X\) be a topological space and let \(x\in X\).  
The family \(\mathcal{N}(x)\) of all neighbourhoods of \(x\) forms a filter on \(X\). Indeed:
\begin{enumerate}
    \item if \(N_1,N_2\in\mathcal{N}(x)\), then \(N_1\cap N_2\in\mathcal{N}(x)\),
    \item if \(N\in\mathcal{N}(x)\) and \(N\subseteq N'\subseteq X\), where \(N'\) is also a neighbourhood of \(x\), then \(N'\in\mathcal{N}(x)\),
    \item and \(\mathcal{N}(x)\) is nonempty, since every point belongs to each of its neighbourhoods.
\end{enumerate}
\end{example}

\begin{example}
Let \(X=\mathbb{R}\).  
Consider the family
\[
\mathcal{F}=\{A\subseteq \mathbb{R}\mid \exists M\in\mathbb{R}\text{ such that }[M,\infty)\subseteq A\}.
\]
Then \(\mathcal{F}\) is a filter on \(\mathbb{R}\). Indeed:
\begin{itemize}
    \item the intersection of two such sets again contains a tail interval \([M',\infty)\),
    \item every superset of such a set again contains a tail interval,
    \item and the empty set does not contain any interval of the form \([M,\infty)\).
\end{itemize}
\end{example}

\begin{example}
Let \((X,\tau)\) be a topological space.  
The family of all dense open subsets of \(X\) is often considered in topology and forcing theory.  
Under suitable assumptions, such families behave like filter bases.  
Dense open sets are subsets that are both open and dense, meaning that they intersect every nonempty open subset of \(X\) (cf.\cite{moreno1998weak,grcic2022densehybrid}).
\end{example}

\begin{example}
Let \((X,\mathcal{M},\mu)\) be a measure space, where \(\mu\) is a measure on the \(\sigma\)-algebra \(\mathcal{M}\).  
One may consider the family of measurable subsets of \(X\) having positive measure.  
This family is useful in measure theory, although in general it is not a filter unless additional assumptions are imposed.
\end{example}

\begin{example}[Voting system]
\footnote{A voting system is a method by which voters choose among candidates or policies, typically by aggregating individual preferences into a collective decision (cf.\cite{grossi2021lecture,babaioff2016mechanism,ciesielski2021some,machover2014mathematical}).}
Let \(X=\{v_1,v_2,\dots,v_n\}\) be a set of voters.  
Suppose that each voter has a preference regarding a candidate \(C\).  
One may consider the family \(\mathcal{F}\) of subsets \(A\subseteq X\) such that the voters in \(A\) satisfy a prescribed majority condition in favor of \(C\).

This idea is often used only as an analogy for ultrafilter-like decisiveness.  
More precisely, it illustrates the intuition that an ultrafilter selects one side of every dichotomy in a coherent and maximal way.
\end{example}

Several concepts related to filters and ultrafilters have also been introduced in the literature.  
One such notion is that of a \emph{weak filter}, which arises in logic in connection with generalized ``most'' quantifiers in first-order logic \cite{koutras2023weak,koutras2017reconstruction,schlechta2007analysis,schlechta1997reduction,koutras2014many,schlechta1997filters,schlechta1995defaults,askounis2016knowledge,koutras2014modal,castricato2021towards}.  
We recall the definition below.

\begin{definition}\cite{koutras2023weak}
A \emph{weak filter} over a nonempty set \(X\) is a family \(\mathcal{F}_w\subseteq 2^X\) satisfying:
\begin{enumerate}
    \item[(WFB0)\(_w\)] If \(A\in\mathcal{F}_w\) and \(A\subseteq B\subseteq X\), then \(B\in\mathcal{F}_w\).
    \item[(WFB1)\(_w\)] If \(A,B\in\mathcal{F}_w\), then \(A\cap B\neq\emptyset\).
    \item[(WFB2)\(_w\)] \(X\in\mathcal{F}_w\).
\end{enumerate}
\end{definition}

\begin{definition}\cite{koutras2023weak}
A \emph{weak ultrafilter} over a nonempty set \(X\) is a complete weak filter, that is, a weak filter \(\mathcal{F}_w\) satisfying:
\begin{enumerate}
    \item[(WFB3)\(_w\)] For every \(A\subseteq X\), one has
    \[
    A\notin\mathcal{F}_w \iff X\setminus A\in\mathcal{F}_w.
    \]
\end{enumerate}
\end{definition}

\begin{example}
Let \(X=\{1,2,3\}\).  
Consider
\[
\mathcal{F}_w=\bigl\{\{1,2,3\},\{1,2\},\{2,3\},\{1,3\}\bigr\}.
\]
Then \(\mathcal{F}_w\) is a weak filter on \(X\), because:
\begin{itemize}
    \item it is upward closed,
    \item the intersection of any two of its members is nonempty,
    \item and \(X\in\mathcal{F}_w\).
\end{itemize}
\end{example}

We next recall several standard properties of ultrafilters in set theory.  
As will be explained later, when working with connectivity systems, one focuses on subsets satisfying \(f(A)\leq k\).  
For that reason, classical properties of set-theoretic ultrafilters do not automatically carry over to the connectivity-system setting.

\begin{definition}[Finite intersection property] (cf.\cite{huang2018finite})
Let \(\mathcal{C}\) be a family of subsets of \(X\).  
We say that \(\mathcal{C}\) has the \emph{finite intersection property} (FIP) if for every finite choice \(D_1,\dots,D_n\in\mathcal{C}\), one has
\[
D_1\cap \cdots \cap D_n\neq\emptyset.
\]
\end{definition}

\begin{theorem}
Let \(X\) be a finite set, let \(\mathcal{F}\) be a filter on \(X\), and let \(A\subseteq X\).  
Then \(\mathcal{F}\cup\{A\}\) has the finite intersection property if and only if \(X\setminus A\notin\mathcal{F}\).
\end{theorem}

\begin{proof}
Assume first that \(\mathcal{F}\cup\{A\}\) has the finite intersection property.  
Suppose, for contradiction, that \(X\setminus A\in\mathcal{F}\).  
Then both \(A\) and \(X\setminus A\) belong to \(\mathcal{F}\cup\{A\}\), but
\[
A\cap (X\setminus A)=\emptyset,
\]
contradicting the finite intersection property.  
Hence \(X\setminus A\notin\mathcal{F}\).

Conversely, assume that \(X\setminus A\notin\mathcal{F}\).  
Take any finite family \(D_1,\dots,D_n\in\mathcal{F}\cup\{A\}\).  
If none of the \(D_i\) is equal to \(A\), then all \(D_i\) belong to \(\mathcal{F}\), and their intersection is nonempty since every filter has the finite intersection property.

Now suppose that one of the sets is \(A\), say \(D_1=A\), while \(D_2,\dots,D_n\in\mathcal{F}\).  
Since \(\mathcal{F}\) is closed under finite intersections,
\[
B:=D_2\cap \cdots \cap D_n \in \mathcal{F}.
\]
If \(A\cap B=\emptyset\), then \(B\subseteq X\setminus A\), and by upward closure of \(\mathcal{F}\), this would imply \(X\setminus A\in\mathcal{F}\), a contradiction.  
Therefore \(A\cap B\neq\emptyset\), and hence
\[
D_1\cap \cdots \cap D_n\neq\emptyset.
\]
Thus \(\mathcal{F}\cup\{A\}\) has the finite intersection property.
\end{proof}

\begin{definition}[Uniform ultrafilter] (cf.\cite{negrepontis1969existence})
An ultrafilter \(\mathcal{U}\) on \(Y\) is called \emph{uniform} if \(|A|=|Y|\) for every \(A\in\mathcal{U}\).
\end{definition}

\begin{theorem}
Let \(\mathcal{U}\) be an ultrafilter on \(Y\), and suppose that \(Z\in\mathcal{U}\) has minimal cardinality among the members of \(\mathcal{U}\).  
Then the induced ultrafilter on \(Z\) is uniform.
\end{theorem}

\begin{proof}
Consider the family
\[
\mathcal{U}\!\upharpoonright Z:=\{A\subseteq Z \mid A\in\mathcal{U}\}.
\]
Since \(Z\in\mathcal{U}\) and \(\mathcal{U}\) is an ultrafilter on \(Y\), it follows that \(\mathcal{U}\!\upharpoonright Z\) is an ultrafilter on \(Z\).

Now let \(A\in \mathcal{U}\!\upharpoonright Z\).  
Then \(A\subseteq Z\) and \(A\in\mathcal{U}\).  
By the minimality of \(|Z|\) among the cardinalities of members of \(\mathcal{U}\), one must have \(|A|\geq |Z|\).  
Since \(A\subseteq Z\), also \(|A|\leq |Z|\).  
Therefore \(|A|=|Z|\).  
Hence every member of \(\mathcal{U}\!\upharpoonright Z\) has the same cardinality as \(Z\), so \(\mathcal{U}\!\upharpoonright Z\) is uniform.
\end{proof}

\begin{theorem}[Ultrafilter on \(A\) induced by \(\mathcal{U}\)]
Let \(\mathcal{U}\) be an ultrafilter on \(X\), and let \(A\in\mathcal{U}\).  
Define
\[
A\cap \mathcal{U}:=\{A\cap B \mid B\in\mathcal{U}\}.
\]
Then \(A\cap \mathcal{U}\) is an ultrafilter on \(A\).
\end{theorem}

\begin{proof}
Since \(\mathcal{U}\) is closed under finite intersections, \(A\cap \mathcal{U}\) is closed under finite intersections.  
Indeed, if \(A\cap B_1\) and \(A\cap B_2\) belong to \(A\cap\mathcal{U}\), then
\[
(A\cap B_1)\cap (A\cap B_2)=A\cap (B_1\cap B_2)\in A\cap\mathcal{U}.
\]

It is also upward closed inside \(A\).  
Finally, let \(C\subseteq A\).  
Since \(\mathcal{U}\) is an ultrafilter on \(X\), either \(C\in\mathcal{U}\) or \(X\setminus C\in\mathcal{U}\).  
In the first case,
\[
A\cap C=C\in A\cap\mathcal{U}.
\]
In the second case,
\[
A\cap (X\setminus C)=A\setminus C\in A\cap\mathcal{U}.
\]
Thus, for every \(C\subseteq A\), either \(C\in A\cap\mathcal{U}\) or \(A\setminus C\in A\cap\mathcal{U}\).  
Therefore \(A\cap\mathcal{U}\) is an ultrafilter on \(A\).
\end{proof}

For further background on ultrafilters, see, for example, the surveys and lecture notes in \cite{80,82}.

\section{Ultrafilters on Connectivity Systems}
We introduce several properties of ultrafilters on a connectivity system \((X,f)\), viewed as an extension of ultrafilters on Boolean algebras.

We begin with ultrafilters on graphs. This notion extends the set-theoretic concept of an ultrafilter to graphs and has a dual relationship with tree-width \cite{fujitaultrafilter}.  
See Appendix A for the definition of tree-width.

\begin{definition}\cite{fujitaultrafilter}
Let \(G\) be a graph. A \emph{\(G\)-ultrafilter of order \(k\)} is a family \(\mathcal{F}\) of separations of \(G\) satisfying the following conditions:
\begin{itemize}
    \item[(FG0)] Every separation \((A,B)\in\mathcal{F}\) has order less than \(k\).
    \item[(FG1)] For every separation \((A,B)\) of \(G\) of order less than \(k\), either \((A,B)\in\mathcal{F}\) or \((B,A)\in\mathcal{F}\).
    \item[(FG2)] If \((A_1,B_1)\in\mathcal{F}\), \(A_1\subseteq A_2\), and \((A_2,B_2)\) is a separation of \(G\) of order less than \(k\), then \((A_2,B_2)\in\mathcal{F}\).
    \item[(FG3)] If \((A_1,B_1)\in\mathcal{F}\) and \((A_2,B_2)\in\mathcal{F}\), and if \((A_1\cap A_2,\, B_1\cup B_2)\) is a separation of \(G\) of order less than \(k\), then \((A_1\cap A_2,\, B_1\cup B_2)\in\mathcal{F}\).
    \item[(FG4)] If \(V(A)=V(G)\), then \((A,B)\in\mathcal{F}\).
\end{itemize}
\end{definition}

Next, we recall the notions of a filter and an ultrafilter on a connectivity system \cite{16}.  
These concepts extend the classical notions of filters and ultrafilters from set theory by incorporating symmetric submodularity.  
Moreover, an ultrafilter on a connectivity system \((X,f)\) may be viewed as a co-maximal ideal on \((X,f)\) \cite{49}.

\begin{definition}[\cite{16}]
Let \(X\) be a finite set, and let \(f\) be a symmetric submodular function.  
A family \(F\subseteq 2^X\) is called a \emph{filter of order \(k+1\)} on the connectivity system \((X,f)\) if the following axioms hold:
\footnote{In the literature on connectivity systems, some authors use the parameter \(k\), while others use \(k+1\) for the order.}
\begin{itemize}
    \item[(Q0)] For every \(A\in F\), one has \(f(A)\leq k\).
    \item[(Q1)] If \(A\in F\), \(B\in F\), and \(f(A\cap B)\leq k\), then \(A\cap B\in F\).
    \item[(Q2)] If \(A\in F\), \(A\subseteq B\subseteq X\), and \(f(B)\leq k\), then \(B\in F\).
    \item[(Q3)] \(\emptyset\notin F\).
\end{itemize}

An \emph{ultrafilter} on a connectivity system is a filter satisfying the additional axiom:
\begin{itemize}
    \item[(Q4)] For every \(A\subseteq X\) with \(f(A)\leq k\), either \(A\in F\) or \(X\setminus A\in F\).
\end{itemize}

An ultrafilter on a connectivity system has a dual relationship with branch-width, a graph width parameter \cite{16}.  
In particular, such an ultrafilter is a nonempty proper family, that is, a family \(F\subseteq 2^X\) with \(\emptyset\notin F\).
\end{definition}

\begin{definition}\cite{16}
A filter is called \emph{principal} if it satisfies the following axiom:
\begin{itemize}
    \item[(QP5)] For every \(A\subseteq X\) with \(|A|=1\) and \(f(A)\leq k\), one has \(A\in F\).
\end{itemize}

A filter is called \emph{non-principal} if it satisfies:
\begin{itemize}
    \item[(Q5)] For every \(A\subseteq X\) with \(|A|=1\), one has \(A\notin F\).
\end{itemize}
Thus, a non-principal filter contains no singleton sets. In this sense, it is not generated by a finite set.
\end{definition}

If a filter is \emph{weak} \cite{15,81,111,FujitaMaxQuasi2024}, then the following axiom \((QW1')\) is used in place of \((Q1)\):
\begin{itemize}
    \item[(QW1')] If \(A\in F\), \(B\in F\), and \(f(A\cap B)\leq k\), then \(A\cap B\neq\emptyset\).
\end{itemize}
Weak filters arise in logic in connection with the interpretation of defaults via a generalized ``most'' quantifier in first-order logic \cite{81,111,243,224}.  
A weak filter on a connectivity system may also be regarded as a co-weak ideal on that connectivity system.

If a filter is \emph{quasi} \cite{14,112,FujitaMaxQuasi2024}, then the following axiom \((QQ1')\) is used in place of \((Q1)\):
\begin{itemize}
    \item[(QQ1')] If \(A\subseteq X\), \(B\subseteq X\), \(A\notin F\), and \(B\notin F\), then \(A\cup B\notin F\).
\end{itemize}
A quasi-ultrafilter is used in axiomatic analyses of incomplete social judgments \cite{112}.

If a filter is \emph{single} \cite{4}, then the following axiom \((QS1)\) is used in place of \((Q1)\):
\begin{itemize}
    \item[(QS1)] For any \(A\in F\) and \(e\in X\), if \(f(\{e\})\leq k\) and \(f(A\cap (X\setminus\{e\}))\leq k\), then \(A\cap (X\setminus\{e\})\in F\).
\end{itemize}

In fact, one may replace \((QS1)\) by the following axiom \((QSD1)\) \cite{4}:
\begin{itemize}
    \item[(QSD1)] For \(A\in F\) and \(e\in X\), if \(f(A\setminus\{e\})\leq k\), then \(A\setminus\{e\}\in F\).
\end{itemize}
This axiom is closely related to the notion of single-element deletion \cite{134,135,136}, which may be viewed as dual in spirit to single-element extension \cite{245,246,247,248,249}, and is also related to single-vertex deletion and single-element deletion in graph-theoretic settings (cf.\ \cite{385,386,387}).

If a filter on a connectivity system is \emph{connected}, then the following conditions hold \cite{FujitaGrill2024}:
\begin{itemize}
    \item[(CF4)] For any \(A,A'\in\mathcal{F}\) with \(A\cup A'\neq X\), if \(f(A\cap A')\leq k\) and \(f(A\cup A')\leq k\), then \(A\cap A'\in\mathcal{F}\).
    \item[(CF5)] For any \(e\in X\), one has \(\{e\}\notin\mathcal{F}\) whenever \(f(\{e\})\leq k\) and \(\emptyset\notin\mathcal{F}\).
\end{itemize}

An ultrafilter \(\mathcal{U}\) on a connectivity system \((X,f)\) is called \emph{uniform} if \(|A|=|X|\) for every \(A\in\mathcal{U}\).

An equivalent representation of an ultrafilter \(U\subseteq 2^X\) on a connectivity system can be given by a two-valued morphism.  
More precisely, define a function \(m\) by
\[
m(A)=
\begin{cases}
1 & \text{if } A\in U,\\
0 & \text{otherwise}.
\end{cases}
\]
In this form, the ultrafilter is encoded as a two-valued object that distinguishes the sets selected by \(U\), subject to the restriction \(f(A)\leq k\).  
This viewpoint suggests that ultrafilters are naturally suited to certain two-player games \cite{254,255}, such as Cops and Robbers.  
Indeed, in the literature \cite{16}, an ultrafilter on a connectivity system is used as a winning strategy.

\section{Other Concepts Related to Ultrafilters on a Connectivity System}

In this section, we discuss several concepts related to ultrafilters on a connectivity system.

\subsection{Relation between Tangles and Ultrafilters}

A notion closely related to ultrafilters on a connectivity system is that of a \emph{tangle}.  
Tangles are dual obstruction-type objects for decomposition parameters such as tree-decompositions and branch-decompositions, and they play a central role in the theory of graph width parameters \cite{albrechtsen2023refining,diestel2021point}.  
They also appear in several algorithmic contexts, including graph minors, width parameters, and graph isomorphism problems.  
For graphs, tangles and ultrafilter-type orientations of separations are closely related \cite{fujitaultrafilter}.

For reference, we first recall the graph-theoretic definition.

\begin{definition}\cite{Robertson1991}
A \emph{tangle} in a graph \(G\) of order \(k\) is a set \(T\) of separations of \(G\), each of order less than \(k\), such that:
\begin{enumerate}
    \item[(T1)] For every separation \((A,B)\) of \(G\) of order less than \(k\), exactly one of \((A,B)\) and \((B,A)\) lies in \(T\).
    \item[(T2)] If \((A_1,B_1), (A_2,B_2), (A_3,B_3) \in T\), then
    \[
    A_1 \cup A_2 \cup A_3 \neq G.
    \]
    \item[(T3)] If \((A,B)\in T\), then \(V(A)\neq V(G)\).
\end{enumerate}
\end{definition}

Several variants of tangles have also been studied, including directed tangles, hypertangles, connected tangles, distance tangles, edge tangles, matching tangles, and linear tangles \cite{giannopoulou2022directed,250,Fujitaconnected2024,Fujitadistance2024,liu2022global,fomin2019width,51}.

The notion of a tangle extends naturally from graphs to matroids and connectivity systems \cite{9}.  
The corresponding definition for connectivity systems is as follows.

\begin{definition}[\cite{9}]
Let \(X\) be a finite set, and let \(f:2^X\to \mathbb{N}\) be a symmetric submodular function.  
A family \(T\subseteq 2^X\) is called a \emph{tangle of order \(k+1\)} on the connectivity system \((X,f)\) if it satisfies:
\begin{itemize}
    \item[(T1)] For every \(A\in T\), one has \(f(A)\leq k\).
    \item[(T2)] For every \(A\subseteq X\) with \(f(A)\leq k\), either \(A\in T\) or \(X\setminus A\in T\).
    \item[(T3)] If \(A,B,C\in T\), then \(A\cup B\cup C \neq X\).
    \item[(T4)] For every \(e\in X\), one has \(X\setminus \{e\}\notin T\).
\end{itemize}
\end{definition}

It is standard that a tangle of larger order induces tangles of smaller order by restriction.

\begin{proposition}
Let \(T\) be a tangle of order \(k+1\) on a connectivity system \((X,f)\), and let \(j\) be an integer with \(0\leq j\leq k\).  
Define
\[
T^{\leq j}:=\{A\in T \mid f(A)\leq j\}.
\]
Then \(T^{\leq j}\) is a tangle of order \(j+1\) on \((X,f)\).
\end{proposition}

\begin{proof}
Condition (T1) is immediate from the definition of \(T^{\leq j}\).  
For (T2), let \(A\subseteq X\) satisfy \(f(A)\leq j\). Since \(j\leq k\), the tangle axiom (T2) for \(T\) implies that either \(A\in T\) or \(X\setminus A\in T\). Because \(f(X\setminus A)=f(A)\leq j\), whichever of these belongs to \(T\) also belongs to \(T^{\leq j}\).  
Condition (T3) is inherited from \(T\), and (T4) is also inherited since \(f(X\setminus\{e\})=f(\{e\})\) and the restriction does not create new members.  
Hence \(T^{\leq j}\) is a tangle of order \(j+1\).
\end{proof}

The same restriction principle holds for ultrafilters on connectivity systems.

\begin{proposition}
Let \(F\) be an ultrafilter of order \(k+1\) on a connectivity system \((X,f)\), and let \(j\) satisfy \(0\leq j\leq k\).  
Define
\[
F^{\leq j}:=\{A\in F \mid f(A)\leq j\}.
\]
Then \(F^{\leq j}\) is an ultrafilter of order \(j+1\) on \((X,f)\).
\end{proposition}

\begin{proof}
Axiom (Q0) is immediate from the definition of \(F^{\leq j}\).  
For (Q1), if \(A,B\in F^{\leq j}\) and \(f(A\cap B)\leq j\), then \(A,B\in F\), so \(A\cap B\in F\) by (Q1) for \(F\). Since \(f(A\cap B)\leq j\), we have \(A\cap B\in F^{\leq j}\).  
For (Q2), if \(A\in F^{\leq j}\), \(A\subseteq B\subseteq X\), and \(f(B)\leq j\), then \(f(B)\leq k\), so \(B\in F\) by (Q2) for \(F\), hence \(B\in F^{\leq j}\).  
Axiom (Q3) is inherited from \(F\).  
For (Q4), let \(A\subseteq X\) with \(f(A)\leq j\). Since \(f(A)\leq k\), axiom (Q4) for \(F\) implies that either \(A\in F\) or \(X\setminus A\in F\). By symmetry, \(f(X\setminus A)=f(A)\leq j\), so either \(A\in F^{\leq j}\) or \(X\setminus A\in F^{\leq j}\).  
Thus \(F^{\leq j}\) is an ultrafilter of order \(j+1\).
\end{proof}

In many settings, tangles and ultrafilters are related by taking complements of the oriented sides.  
More precisely, if \(T\) is a tangle, one may consider the complementary family
\[
F_T:=\{X\setminus A \mid A\in T\}.
\]
This family satisfies several filter-like intersection properties.  
An ultrafilter on a connectivity system can therefore be viewed as a strengthened, highly coherent co-tangle-type object.

Analogously, linear tangles give rise to single-filter-type axioms, where one of the admissible deletions is restricted to a singleton.  
We do not pursue the full abstract equivalence here, but this viewpoint is useful for understanding the role of axioms such as \((QS1)\), \((QSD1)\), and their variants.

\subsection{Bramble and Ultrafilter}

It is well known that brambles are closely related to tree-width and to obstruction-type objects arising from orientations of separations \cite{2}.  
In graph game theory, especially in connection with Cops and Robbers, brambles are classical obstructions to small tree-width and related escape phenomena \cite{143,an2024sparse}.  
See also strict brambles, line brambles, and tight brambles \cite{Lardas2022OnSB,harvey2015treewidth,harvey2012treewidth,mescoff2023mixed}.

\begin{definition}\cite{216,Hans2008}
Let \(G=(V,E)\) be a graph. Two subsets \(W_1,W_2\subseteq V\) are said to \emph{touch} if either
\[
W_1\cap W_2\neq \emptyset,
\]
or there exists an edge \(\{w_1,w_2\}\in E\) with \(w_1\in W_1\) and \(w_2\in W_2\).  
A set \(B\) of pairwise touching connected vertex sets is called a \emph{bramble}.  
A subset \(S\subseteq V\) is said to \emph{cover} \(B\) if \(S\) intersects every element of \(B\).  
The \emph{order} of \(B\) is the minimum size of a covering set for \(B\).  
The \emph{bramble number} of \(G\) is the maximum order of a bramble in \(G\).
\end{definition}

\begin{theorem}\cite{216}
A graph \(G\) has tree-width \(k\) if and only if the bramble number of \(G\) is \(k+1\).
\end{theorem}

\subsection{Ultrafilters, \(\pi\)-Systems, \(\lambda\)-Systems, and Superfilters on a Connectivity System}

One aim of this book is to examine filters and graph width parameters from several different viewpoints.  
In this subsection, we compare ultrafilters with \(\pi\)-systems, \(\lambda\)-systems, and superfilters.

Recall the classical set-theoretic intuition:
\begin{itemize}
    \item A \(\pi\)-system is a family of sets closed under finite intersections \cite{chentsov2014some,kallenberg1997foundations,durrett2019probability}.
    \item A \(\lambda\)-system (or \(d\)-system) is a family of sets closed under complements and countable disjoint unions \cite{guide2006infinite,vestrup2009theory,kallenberg1997foundations}.
    \item A superfilter is a family that is upward closed and satisfies a union-based primality condition \cite{samet2009superfilters,tsaban2014algebra,tsaban2014algebra2}.
\end{itemize}

\begin{definition}
Let \((X,f)\) be a connectivity system.  
A nonempty family \(P\subseteq 2^X\) is called a \emph{\(\pi\)-system of order \(k+1\)} on \((X,f)\) if whenever \(A,B\in P\) and \(f(A\cap B)\leq k\), then
\[
A\cap B \in P.
\]
\end{definition}

\begin{proposition}
Every nonempty filter of order \(k+1\) on a connectivity system \((X,f)\) is a \(\pi\)-system of order \(k+1\).
\end{proposition}

\begin{proof}
This is immediate from axiom (Q1).
\end{proof}

\begin{definition}
Let \((X,f)\) be a connectivity system.  
A family \(D\subseteq 2^X\) is called a \emph{\(\lambda\)-system of order \(k+1\)} on \((X,f)\) if:
\begin{enumerate}
    \item[(L1)] \(X\in D\).
    \item[(L2)] If \(A\in D\) and \(f(X\setminus A)\leq k\), then \(X\setminus A\in D\).
    \item[(L3)] If \(\{A_n\}_{n\in\mathbb{N}}\subseteq D\) is a sequence of pairwise disjoint sets and
    \[
    f\!\left(\bigcup_{n=1}^{\infty} A_n\right)\leq k,
    \]
    then
    \[
    \bigcup_{n=1}^{\infty} A_n \in D.
    \]
\end{enumerate}
\end{definition}

A \(\pi\)-system need not be a \(\lambda\)-system, even in the present finite setting.

\begin{example}
Let \(X=\{1,2\}\), and let \(f(A)=0\) for all \(A\subseteq X\).  
Then
\[
P=\bigl\{X,\{1\}\bigr\}
\]
is a \(\pi\)-system (and also a filter), since it is nonempty and closed under intersections.  
However, \(P\) is not a \(\lambda\)-system, because \(\{1\}\in P\) but \(X\setminus \{1\}=\{2\}\notin P\).
\end{example}

\begin{corollary}
A filter of order \(k+1\) on a connectivity system need not be a \(\lambda\)-system of order \(k+1\).
\end{corollary}

\begin{definition}[Superfilter on a Connectivity System]
Let \((X,f)\) be a connectivity system, and fix \(k\in \mathbb{N}\).  
A nonempty family \(S\subseteq 2^X\) is called a \emph{superfilter of order \(k+1\)} on \((X,f)\) if:
\begin{itemize}
    \item[(SUF1)] For every \(A\in S\), one has \(f(A)\leq k\).
    \item[(SUF2)] If \(A\in S\), \(A\subseteq B\subseteq X\), and \(f(B)\leq k\), then \(B\in S\).
    \item[(SUF3)] If \(A\cup B\in S\) and \(f(A)\leq k\), \(f(B)\leq k\), then \(A\in S\) or \(B\in S\).
\end{itemize}
\end{definition}

\begin{theorem}
Every ultrafilter of order \(k+1\) on a connectivity system \((X,f)\) is a superfilter of order \(k+1\).
\end{theorem}

\begin{proof}
Let \(F\) be an ultrafilter of order \(k+1\) on \((X,f)\).

Condition (SUF1) is exactly axiom (Q0), and (SUF2) is exactly axiom (Q2).  
It remains to verify (SUF3).

Assume that \(A\cup B\in F\), with \(f(A)\leq k\) and \(f(B)\leq k\).  
Suppose, towards a contradiction, that \(A\notin F\) and \(B\notin F\).  
By axiom (Q4), this implies \(X\setminus A\in F\) and \(X\setminus B\in F\).

Since
\[
(X\setminus A)\cap (X\setminus B)=X\setminus (A\cup B),
\]
and
\[
f\bigl(X\setminus (A\cup B)\bigr)=f(A\cup B)\leq k
\]
by symmetry and (Q0), axiom (Q1) yields
\[
X\setminus (A\cup B)\in F.
\]
But now both \(A\cup B\in F\) and \(X\setminus (A\cup B)\in F\).  
Because
\[
(A\cup B)\cap \bigl(X\setminus (A\cup B)\bigr)=\emptyset
\]
and \(f(\emptyset)=0\leq k\), another application of (Q1) would give \(\emptyset\in F\), contradicting axiom (Q3).  
Therefore at least one of \(A\in F\) or \(B\in F\) must hold.  
Hence \(F\) is a superfilter.
\end{proof}

It is also known that filters are closely related to grills and primals, or to their complementary counterparts, in the context of connectivity systems \cite{FujitaGrill2024}.  
Similarly, weak filters correspond to weak grills and weak primals in analogous ways.  
These connections are useful because grills and primals provide alternative languages for expressing filter-like behavior in topology and lattice-theoretic settings \cite{thron1973proximity,matejdes2024induced,ibedou2023ideals,roy2007typical,karthik2014generalization,dolecki2003hyperconvergences}.

\subsection{Ultrafilters and Co-Independence Systems on a Connectivity System}

The language of independence systems is also relevant.  
Instead of identifying a filter directly with an independence system, it is more accurate to regard a filter as an \emph{upward-closed dual analogue} of such a system.

\begin{definition}
Let \((X,f)\) be a connectivity system.  
A family \(I\subseteq 2^X\) is called an \emph{independence system of order \(k+1\)} if:
\begin{itemize}
    \item[(IN1)] \(\emptyset \in I\).
    \item[(IN2)] Whenever \(A\in I\) and \(Y\subseteq A\) satisfy \(f(Y)\leq k\), one has \(Y\in I\).
\end{itemize}
\end{definition}

\begin{definition}
A family \(C\subseteq 2^X\) is called a \emph{co-independence system of order \(k+1\)} if:
\begin{itemize}
    \item[(COIN1)] \(C\neq \emptyset\).
    \item[(COIN2)] Whenever \(A\in C\), \(A\subseteq B\subseteq X\), and \(f(B)\leq k\), one has \(B\in C\).
\end{itemize}
\end{definition}

\begin{proposition}
Every nonempty filter of order \(k+1\) on a connectivity system is a co-independence system of order \(k+1\).
\end{proposition}

\begin{proof}
This is immediate from axiom (Q2).
\end{proof}

Under the complement map \(A\mapsto X\setminus A\), co-independence systems correspond to independence systems because \(f(X\setminus A)=f(A)\) by symmetry.

\subsection{\(\sigma\)-Filters on a Connectivity System}

We next consider \(\sigma\)-filters.  
In classical measure theory, a \(\sigma\)-filter is dual to a \(\sigma\)-ideal and is closed under supersets and countable intersections \cite{bauer2001measure}.  
For a connectivity system, we incorporate the order bound into the definition.

\begin{definition}
Let \((X,f)\) be a connectivity system.  
A family \(F\subseteq 2^X\) is called a \emph{\(\sigma\)-filter of order \(k+1\)} on \((X,f)\) if:
\begin{itemize}
    \item[(SIF0)] For every \(A\in F\), one has \(f(A)\leq k\).
    \item[(SIF1)] \(X\in F\).
    \item[(SIF2)] If \(A\in F\), \(A\subseteq B\subseteq X\), and \(f(B)\leq k\), then \(B\in F\).
    \item[(SIF3)] If \(\{A_n\}_{n\in\mathbb{N}}\subseteq F\) and
    \[
    f\!\left(\bigcap_{n\in\mathbb{N}} A_n\right)\leq k,
    \]
    then
    \[
    \bigcap_{n\in\mathbb{N}} A_n \in F.
    \]
\end{itemize}
\end{definition}

\begin{proposition}
Let \((X,f)\) be a connectivity system with \(X\) finite.  
Then every \(\sigma\)-filter of order \(k+1\) satisfying \(\emptyset\notin F\) is a filter of order \(k+1\).
\end{proposition}

\begin{proof}
Axioms (Q0) and (Q2) are exactly (SIF0) and (SIF2), while (Q3) is the additional assumption \(\emptyset\notin F\).  
It remains to prove (Q1).

Let \(A,B\in F\) with \(f(A\cap B)\leq k\).  
Consider the sequence \(A_1=A\), \(A_2=B\), and \(A_n=X\) for all \(n\geq 3\).  
Then
\[
\bigcap_{n\in\mathbb{N}} A_n = A\cap B.
\]
Since each \(A_n\in F\) and \(f(A\cap B)\leq k\), axiom (SIF3) implies \(A\cap B\in F\).  
Hence \(F\) satisfies (Q1), and therefore \(F\) is a filter of order \(k+1\).
\end{proof}

\subsection{Closure Systems, Union-Closed Systems, and Game-Theoretic Viewpoints}

Closure systems are frequently studied in lattice theory, ordered structures, and game-related combinatorics \cite{caspard2003lattices,adaricheva2013ordered}.  
In the present setting, a filter or ultrafilter naturally gives rise to an intersection-stable family of feasible sets.

\begin{definition}(cf.\cite{caspard2003lattices})
Let \((X,f)\) be a connectivity system.  
A family \(\mathcal{F}\subseteq 2^X\) is called a \emph{closure system of order \(k+1\)} on \((X,f)\) if:
\begin{enumerate}
    \item[(CL1)] If \(F_1,F_2\in \mathcal{F}\) and \(f(F_1\cap F_2)\leq k\), then \(F_1\cap F_2\in \mathcal{F}\).
    \item[(CL2)] \(X\in \mathcal{F}\).
\end{enumerate}
\end{definition}

\begin{proposition}
Every nonempty filter of order \(k+1\) on a connectivity system is a closure system of order \(k+1\).
\end{proposition}

\begin{proof}
Axiom (CL1) is precisely axiom (Q1).  
If \(F\) is nonempty, choose \(A\in F\). Since \(A\subseteq X\) and \(f(X)=f(\emptyset)=0\leq k\), axiom (Q2) implies \(X\in F\).  
Hence (CL2) also holds.
\end{proof}

A related notion is that of a union-closed family.

\begin{definition}[Union-Closed System on a Connectivity System]
Let \((X,f)\) be a connectivity system.  
A family \(\mathcal{F}\subseteq 2^X\) is called a \emph{union-closed system of order \(k+1\)} if:
\begin{enumerate}
    \item[(UC1)] If \(F_1,F_2\in \mathcal{F}\) and \(f(F_1\cup F_2)\leq k\), then \(F_1\cup F_2\in \mathcal{F}\).
    \item[(UC2)] \(\emptyset\in \mathcal{F}\) and \(X\in \mathcal{F}\).
\end{enumerate}
\end{definition}

Such families occur naturally in cooperative game theory, where players form coalitions subject to feasibility constraints \cite{aumann1964bargaining,peleg2007introduction,nash1953two,lozano2013cooperative,nash1950non}.  
Many related feasible-set systems have been studied, including accessible union-stable systems, antimatroids, simplicial complexes, regular set systems, \(k\)-regular set systems, building sets, crossing set systems, communication feasible set systems, coalition structures, augmenting systems, monotone set systems, S-extremal set systems, chain quasi-building systems, intersection-closed quasi-building systems, union-stable quasi-building systems, and union-stable systems \cite{algaba2013cooperative,algaba2003axiomatizations,algaba2004axiomatization,pruesse1993gray,gutin2002anti,dietrich1989matroids,martino2021cooperative,Lange2009ValuesOR,Xie2009TheCO,Koshevoy2011SolutionCF,Karzanov1996HowTT,Proefschrift2015HierarchiesCA,ZouSolutionsFC,Bilbao2003CooperativeGU,Kaplan1999OnlineCO,Mszros2010SextremalSS,Suzuki2015SolutionsFC,algaba2000position,algaba2015harsanyi}.

However, a filter on a connectivity system is, in general, \emph{not} union-closed.  
By contrast, every nonempty filter is intersection-closed in the sense of (CL1), and therefore gives rise to a closure system.

\begin{question}
If these feasible-set systems are extended to connectivity systems, do they admit natural characterizations in terms of branch-width, tangles, or ultrafilters?
\end{question}

\begin{question}
Do cooperative games defined on filters or ultrafilters exhibit structural features that are absent in the usual union-closed setting?
\end{question}

Ultrafilters also have a natural interpretation in voting theory.  
On a finite set of voters, an ultrafilter corresponds to a perfectly decisive winning rule.  
For completeness, we recall the standard notion of a simple voting game.

\begin{definition}\cite{Axenovich2010OnTS}
A \emph{simple voting game} on a finite voter set \(X\) is a family \(W\subseteq 2^X\) such that:
\begin{enumerate}
    \item \(\emptyset \notin W\),
    \item \(X\in W\),
    \item if \(A\in W\) and \(A\subseteq A'\subseteq X\), then \(A'\in W\).
\end{enumerate}
The members of \(W\) are called \emph{winning coalitions}.
\end{definition}

\begin{definition}
A simple voting game \(W\) is called \emph{proper} if for every \(A\subseteq X\),
\[
A\in W \quad \Longrightarrow \quad X\setminus A \notin W.
\]
\end{definition}

Thus, when restricted to subsets of order at most \(k\), an ultrafilter on a connectivity system behaves like a proper monotone winning-family of admissible coalitions.  
This viewpoint suggests further interactions between connectivity systems, graph-width obstructions, and voting/game-theoretic models.

Further related width parameters and associated games are discussed in Appendix E.

\subsection{Ultrafilters and Majority Systems on Connectivity Systems}

We next consider majority systems on a connectivity system.
A \emph{majority space} is, roughly speaking, a set system in which for every admissible set,
either the set itself or its complement is regarded as a majority.
Classical ultrafilters in set theory provide typical examples of such systems
\cite{salame2006majority,pacuit2006majority,pacuit2004majority}.

\begin{definition}[Majority system]
\label{def:majority-system}
Let \((X,f)\) be a connectivity system, where \(X\) is a finite set and
\(f:2^X\to \mathbb{N}\) is a symmetric submodular function.
A family \(M\subseteq 2^X\) is called a \emph{majority system of order \(k+1\)}
on \((X,f)\) if it satisfies the following conditions:
\begin{enumerate}
    \item[(MA1)] For every \(A\subseteq X\) with \(f(A)\leq k\), either
    \(A\in M\) or \(X\setminus A\in M\).
    \item[(MA2)] If \(A,B\in M\) and \(A\cap B=\emptyset\), then
    \(B=X\setminus A\).
    \item[(MA3)] If \(A\in M\), \(F\subseteq A\) is finite, \(G\subseteq X\setminus A\),
    \(|F|\leq |G|\), and
    \[
    f\bigl((A\setminus F)\cup G\bigr)\leq k,
    \]
    then
    \[
    (A\setminus F)\cup G \in M.
    \]
\end{enumerate}
The pair \(\langle X,M\rangle\) is called a \emph{weak majority space}.

A set \(A\subseteq X\) is called a \emph{strict majority} (with respect to \(M\))
if \(A\in M\) and \(X\setminus A\notin M\).
It is called a \emph{weak majority} if both \(A\in M\) and \(X\setminus A\in M\).
Every member of \(M\) is called a \emph{majority set}.
\end{definition}

The ultrafilter axioms immediately imply the decisiveness property (MA1),
while (MA2) is also compatible with ultrafilters because two disjoint members
cannot both belong to an ultrafilter.
This motivates the following conjecture.

\begin{conjecture}
\label{conj:ultrafilter-majority}
Let \((X,f)\) be a connectivity system, and let \(\mathcal{U}\) be a non-principal
ultrafilter of order \(k+1\) on \((X,f)\).
Then \(\mathcal{U}\) is a majority system of order \(k+1\) on \((X,f)\).
\end{conjecture}

The results established so far can be summarized as follows.

\begin{proposition}
\label{prop:ultrafilter-summary}
Let \(\mathcal{U}\) be a non-principal ultrafilter of order \(k+1\)
on a connectivity system \((X,f)\). Then:
\begin{enumerate}
    \item \(\mathcal{U}\) is a superfilter of order \(k+1\) on \((X,f)\).
    \item \(\mathcal{U}\) is a closure system of order \(k+1\) on \((X,f)\).
    \item \(\mathcal{U}\) is a \(\sigma\)-filter of order \(k+1\) on \((X,f)\).
    \item \(\mathcal{U}\) is a co-independence system of order \(k+1\) on \((X,f)\).
    \item \(\mathcal{U}\) is a \(\pi\)-system of order \(k+1\) on \((X,f)\).
    \item \(\mathcal{U}\) is a weak ultrafilter of order \(k+1\) on \((X,f)\).
\end{enumerate}
\end{proposition}

\begin{proof}
Items (1)--(6) follow from the definitions and the propositions established in the previous subsections.
\end{proof}

\begin{remark}
\label{rem:related-concepts-ultrafilter}
According to the cited literature, non-principal ultrafilters on connectivity systems
are also closely related to several complementary or equivalent notions,
including co-tangles, co-tangle-kits, co-loose tangles, co-loose-tangle-kits,
maximal ideals, grills, co-primals, and co-profiles
\cite{16,Oum2006CertifyingLB,93,49,FujitaGrill2024,1}.
They also arise in game-theoretic formulations, for example as robber-winning strategies
in monotone search games and open monotone search games on connectivity systems
\cite{16}.
\end{remark}


\section{Branch-Width on Connectivity Systems}

In this section, we discuss branch-width for connectivity systems.

\subsection{Branch-Width}

Branch-width is a fundamental width parameter in graph theory.
It is defined via branch decompositions, in which the leaves of a tree correspond
to the edges of the graph.
This notion extends naturally to connectivity systems through symmetric submodular functions
\cite{9,10,16,geelen2002branch}.

We begin with the classical graph-theoretic definition.

\begin{definition}\cite{282}
Let \(G\) be a graph.
A \emph{branch decomposition} of \(G\) is a pair \((T,\tau)\), where
\(T\) is a tree whose vertices all have degree \(1\) or \(3\), and
\(\tau\) is a bijection from the set \(L(T)\) of leaves of \(T\) onto the edge set \(E(G)\).

For an edge \(e\in E(T)\), the deletion of \(e\) separates \(T\) into two components.
These components induce a partition \((E_1,E_2)\) of \(E(G)\) via the bijection \(\tau\).
The \emph{order} of \(e\) is the number of vertices \(v\in V(G)\) that are incident with
at least one edge in \(E_1\) and at least one edge in \(E_2\).

The \emph{width} of \((T,\tau)\) is the maximum order of its edges, and
the \emph{branch-width} of \(G\), denoted by \(\operatorname{bw}(G)\),
is the minimum width over all branch decompositions of \(G\).

By convention, if \(|E(G)|=1\), then \(\operatorname{bw}(G)=0\).
\end{definition}

Tree-width and branch-width are closely related, but they arise from different kinds
of decompositions.
Tree-width is based on tree decompositions of the vertex set,
whereas branch-width is based on recursive partitions of the edge set.
Thus both parameters measure structural complexity, but from different viewpoints.

The corresponding notion for connectivity systems is as follows.

\begin{definition}[Branch decomposition of a connectivity system]\cite{10}
\label{def:branch-width}
Let \((X,f)\) be a connectivity system, where \(X\) is a finite set and
\(f:2^X\to \mathbb{N}\) is symmetric and submodular.

A \emph{branch decomposition} of \((X,f)\) is a pair \((T,\mu)\), where
\(T\) is a tree all of whose internal vertices have degree \(3\),
and \(\mu:L(T)\to X\) is a bijection from the leaf set \(L(T)\) of \(T\) onto \(X\).

For each edge \(e\in E(T)\), the graph \(T-e\) has exactly two components,
say \(T_1\) and \(T_2\).
Let
\[
X_1(e):=\{\mu(\ell)\mid \ell\in L(T_1)\},
\qquad
X_2(e):=\{\mu(\ell)\mid \ell\in L(T_2)\}.
\]
Then \((X_1(e),X_2(e))\) is a partition of \(X\), and by symmetry of \(f\),
\[
f\bigl(X_1(e)\bigr)=f\bigl(X_2(e)\bigr).
\]
The \emph{order} of \(e\) is defined by
\[
\operatorname{ord}_{(T,\mu)}(e):=f\bigl(X_1(e)\bigr).
\]

The \emph{width} of the branch decomposition \((T,\mu)\) is
\[
\operatorname{width}(T,\mu):=
\max_{e\in E(T)} \operatorname{ord}_{(T,\mu)}(e).
\]
The \emph{branch-width} of the connectivity system \((X,f)\), denoted by
\(\operatorname{bw}(X,f)\), is the minimum width over all branch decompositions of \((X,f)\).
\end{definition}

\begin{example}
\label{ex:branchwidth-connectivity}
Let \(X=\{a,b,c,d\}\), and define
\[
f(S):=\min\{|S|,\ |X\setminus S|\}
\qquad (S\subseteq X).
\]
Then \(f\) is a symmetric submodular function on \(X\).

Consider a tree \(T\) with four leaves \(v_a,v_b,v_c,v_d\), and let
\[
\mu(v_a)=a,\qquad \mu(v_b)=b,\qquad \mu(v_c)=c,\qquad \mu(v_d)=d.
\]
Suppose that one edge \(e\in E(T)\) separates the leaves into
\(\{v_a,v_b\}\) and \(\{v_c,v_d\}\).
Then
\[
X_1(e)=\{a,b\},
\qquad
X_2(e)=\{c,d\},
\]
and hence
\[
\operatorname{ord}_{(T,\mu)}(e)=f(\{a,b\})=\min\{2,2\}=2.
\]

For every edge of a branch decomposition on four leaves, the corresponding side has
size \(1\) or \(2\), so the order is always at most \(2\).
Moreover, some edge necessarily induces a \(2\)-\(2\) partition, so the width is at least \(2\).
Therefore
\[
\operatorname{width}(T,\mu)=2
\quad\text{and}\quad
\operatorname{bw}(X,f)=2.
\]
\end{example}

In the theory of graph and connectivity width parameters,
duality theorems connect decompositions with obstruction objects such as tangles,
filters, and ultrafilters \cite{9,10,bozyk2022objects,4,16,albrechtsen2024tangle}.
For branch-width on connectivity systems, the following duality theorem with ultrafilters is known.

\begin{theorem}\cite{16}
\label{thm:duality-ultrafilter}
Let \((X,f)\) be a connectivity system.
Then
\[
\operatorname{bw}(X,f)\leq k
\]
if and only if there does not exist a non-principal ultrafilter of order \(k+1\) on \((X,f)\).
\end{theorem}

A corresponding duality theorem with tangles is also classical.

\begin{theorem}\cite{9}
\label{thm:duality-tangle}
Let \((X,f)\) be a connectivity system.
Then
\[
\operatorname{bw}(X,f)\leq k
\]
if and only if there does not exist a tangle of order \(k+1\) on \((X,f)\).
\end{theorem}

\begin{remark}
Variants such as loose tangles and loose tangle kits are also known to enjoy
duality relationships with branch-width and to be closely related to filter-type objects
\cite{10,93}.
\end{remark}

\subsection{Linear Branch-Width}
We next introduce \emph{linear branch decompositions}, which are the linear counterparts of branch decompositions.  
Like branch-width itself, linear branch-width has been studied extensively in the literature \cite{39,51,52}.  
Restricting the decomposition tree to a linear shape often leads to useful refinements of general width parameters and their associated dual objects.

\begin{definition}[Linear branch decomposition]
\label{def:linear-decomposition}
Let \((X,f)\) be a connectivity system, where \(X\) is a finite set and
\(f:2^X\to \mathbb{N}\) is a symmetric submodular function.

A \emph{linear branch decomposition} of \((X,f)\) is a branch decomposition \((T,\mu)\) of \((X,f)\) in which the tree \(T\) is a caterpillar; that is, every vertex of \(T\) has degree \(1\) or \(3\), and the subgraph obtained from \(T\) by deleting all leaves is a path.

Equivalently, a linear branch decomposition may be described by an ordering
\[
x_1,x_2,\dots,x_n
\]
of the elements of \(X\), where \(n=|X|\).  
For each \(i=1,2,\dots,n-1\), let
\[
X_i:=\{x_1,x_2,\dots,x_i\}.
\]
The \emph{width} of this linear branch decomposition is defined by
\[
\max_{1\leq i\leq n-1} f(X_i).
\]

The \emph{linear branch-width} of \((X,f)\), denoted by
\(\operatorname{lbw}(X,f)\), is the minimum of these widths over all orderings of \(X\).
\end{definition}

\begin{remark}
The above ordering-based description is equivalent to the caterpillar-tree description.
Indeed, every caterpillar branch decomposition determines a linear ordering of the leaves from one end of the spine to the other, and conversely every linear ordering of \(X\) yields a caterpillar branch decomposition.
\end{remark}

\begin{example}
\label{ex:linear-branch-width}
Let
\[
X=\{a,b,c,d\},
\]
and define
\[
f(S):=\min\{|S|,\ |X\setminus S|\}
\qquad (S\subseteq X).
\]
Then \(f\) is a symmetric submodular function, so \((X,f)\) is a connectivity system.

Consider the ordering
\[
a,b,c,d.
\]
Then
\[
X_1=\{a\},\qquad
X_2=\{a,b\},\qquad
X_3=\{a,b,c\}.
\]
Hence
\[
f(X_1)=f(\{a\})=1,\qquad
f(X_2)=f(\{a,b\})=2,\qquad
f(X_3)=f(\{a,b,c\})=1.
\]
Therefore the width of this linear branch decomposition is
\[
\max\{1,2,1\}=2.
\]

Since every ordering of a \(4\)-element set has a middle cut of size \(2\), it follows that
\[
\operatorname{lbw}(X,f)=2.
\]
\end{example}

A single filter on a connectivity system is known to be dual to linear branch-width \cite{85}.  
Likewise, linear tangles \cite{fujita2018equivalence} and linear loose tangles \cite{3} are also known to exhibit corresponding duality phenomena.

\begin{theorem}[\cite{4}]
\label{thm:linear-branch-width}
Let \((X,f)\) be a connectivity system. Then the linear branch-width of \((X,f)\) is at most \(k\) if and only if there does not exist a non-principal single ultrafilter of order \(k+1\) on \((X,f)\).
\end{theorem}


\newpage
\chapter{Maximal Filters, Chains, and Antichains on Connectivity Systems}

In this chapter, we study several order-theoretic aspects of connectivity systems.
Since the ground set is finite, many arguments can be formulated in terms of finite partially ordered sets, and no choice principle is needed.

\section{Chains and Antichains on the \(k\)-Truncation}

Let \((X,f)\) be a connectivity system, where \(X\) is a finite set and
\(f:2^X\to \mathbb{N}\) is a symmetric submodular function.
For a fixed integer \(k\geq 0\), define
\[
\mathcal{P}_k(X,f):=\{A\subseteq X \mid f(A)\leq k\}.
\]
We regard \(\mathcal{P}_k(X,f)\) as a finite poset under inclusion.

\begin{definition}[Chain]
\label{def:chain_connectivity}
A \emph{chain of order \(k+1\)} on \((X,f)\) is a finite family
\[
\mathcal{C}=\{A_1,\dots,A_m\}\subseteq \mathcal{P}_k(X,f)
\]
such that, after relabeling if necessary,
\[
A_1 \subseteq A_2 \subseteq \cdots \subseteq A_m.
\]
Equivalently, \(\mathcal{C}\) is a chain in the poset \((\mathcal{P}_k(X,f),\subseteq)\).
\end{definition}

\begin{definition}[Antichain]
\label{def:antichain_connectivity}
An \emph{antichain of order \(k+1\)} on \((X,f)\) is a finite family
\[
\mathcal{A}=\{A_1,\dots,A_m\}\subseteq \mathcal{P}_k(X,f)
\]
such that for all distinct \(i,j\),
\[
A_i \nsubseteq A_j
\qquad\text{and}\qquad
A_j \nsubseteq A_i.
\]
Equivalently, \(\mathcal{A}\) is an antichain in the poset \((\mathcal{P}_k(X,f),\subseteq)\).
\end{definition}

Since \(\mathcal{P}_k(X,f)\) is a finite poset, standard order-theoretic results apply directly.

\begin{theorem}[Dilworth's theorem on \(\mathcal{P}_k(X,f)\)]
\label{thm:dilworth_connectivity}
Let \((X,f)\) be a connectivity system, and let
\[
\mathcal{A}\subseteq \mathcal{P}_k(X,f).
\]
Then the maximum size of an antichain contained in \(\mathcal{A}\) is equal to the minimum number of chains needed to cover \(\mathcal{A}\).
\end{theorem}

\begin{proof}
This is an immediate application of the classical Dilworth theorem to the finite poset \((\mathcal{A},\subseteq)\).
\end{proof}

\section{A Finite Tukey-Type Observation for Filters}

We now turn to filters on a connectivity system.
Recall that a filter of order \(k+1\) on \((X,f)\) is a family
\(F\subseteq 2^X\) satisfying axioms \((Q0)\)--\((Q3)\), and an ultrafilter is a filter satisfying the additional decisiveness axiom \((Q4)\).

\begin{proposition}[Existence of maximal extensions]
\label{prop:maximal-filter-exists}
Let \(F_0\) be a filter of order \(k+1\) on a connectivity system \((X,f)\).
Consider the family
\[
\mathfrak{F}(F_0):=
\{F \subseteq 2^X \mid F \text{ is a filter of order } k+1 \text{ and } F_0\subseteq F\},
\]
partially ordered by inclusion.
Then \(\mathfrak{F}(F_0)\) has a maximal element.
\end{proposition}

\begin{proof}
Since \(X\) is finite, the power set \(2^X\) is finite, and therefore the set of all families \(F\subseteq 2^X\) is finite as well.
Hence \(\mathfrak{F}(F_0)\) is a finite partially ordered set.
Every finite partially ordered set has a maximal element.
\end{proof}

\begin{remark}
Proposition \ref{prop:maximal-filter-exists} is the finite analogue of the usual Tukey/Zorn type maximality argument.
In the present finite setting, no appeal to Tukey's Lemma or Zorn's Lemma is necessary.
\end{remark}

The following implication is always valid.

\begin{proposition}
\label{prop:ultrafilter-maximal-filter}
Every ultrafilter of order \(k+1\) on a connectivity system is maximal among filters of order \(k+1\) under inclusion.
\end{proposition}

\begin{proof}
Let \(U\) be an ultrafilter of order \(k+1\), and suppose that \(F\) is a filter of order \(k+1\) with
\[
U \subsetneq F.
\]
Choose \(A\in F\setminus U\).
By axiom \((Q0)\) for \(F\), we have \(f(A)\leq k\).
Since \(U\) is an ultrafilter and \(A\notin U\), axiom \((Q4)\) implies that
\[
X\setminus A \in U \subseteq F.
\]
Now
\[
A\cap (X\setminus A)=\emptyset,
\]
and
\[
f(\emptyset)=0\leq k.
\]
Hence axiom \((Q1)\) for \(F\) yields \(\emptyset\in F\), contradicting axiom \((Q3)\).
Therefore no proper filter extension of \(U\) exists, and \(U\) is maximal.
\end{proof}

\begin{remark}
The converse implication, namely that every maximal filter is an ultrafilter, is not automatic from the definitions and should not be assumed without a separate proof.
For this reason it is better to distinguish carefully between \emph{maximal filters} and \emph{ultrafilters}.
\end{remark}

\section{Interaction between Filters and Chains}

We next record a basic monotonicity property.

\begin{lemma}
\label{lem:X-in-filter}
Let \(F\) be a nonempty filter of order \(k+1\) on \((X,f)\).
Then \(X\in F\).
\end{lemma}

\begin{proof}
Choose \(A\in F\).
Since \(A\subseteq X\) and
\[
f(X)=f(\emptyset)=0\leq k,
\]
axiom \((Q2)\) implies that \(X\in F\).
\end{proof}

This yields the following simple description of how a filter meets a chain.

\begin{proposition}
\label{prop:filter-meets-chain}
Let \(F\) be a nonempty filter of order \(k+1\) on \((X,f)\), and let
\[
A_1 \subseteq A_2 \subseteq \cdots \subseteq A_m
\]
be a chain in \(\mathcal{P}_k(X,f)\).
Then the set of indices
\[
I_F:=\{i\in \{1,\dots,m\}\mid A_i\in F\}
\]
is either empty or a final interval of \(\{1,\dots,m\}\).
In other words, if \(A_i\in F\) and \(i\leq j\leq m\), then \(A_j\in F\).
\end{proposition}

\begin{proof}
If \(A_i\in F\) and \(A_i\subseteq A_j\), then \(f(A_j)\leq k\) because \(A_j\in\mathcal{P}_k(X,f)\).
Hence axiom \((Q2)\) implies \(A_j\in F\).
\end{proof}

\begin{corollary}
\label{cor:ultrafilter-chain-final-segment}
Let \(U\) be an ultrafilter of order \(k+1\) on \((X,f)\), and let
\[
A_1 \subseteq A_2 \subseteq \cdots \subseteq A_m = X
\]
be a chain in \(\mathcal{P}_k(X,f)\).
Then there exists an index \(j\in\{1,\dots,m\}\) such that
\[
U\cap \{A_1,\dots,A_m\}=\{A_j,A_{j+1},\dots,A_m\}.
\]
\end{corollary}

\begin{proof}
By Lemma \ref{lem:X-in-filter}, \(X\in U\), so \(I_U\neq \emptyset\).
Now apply Proposition \ref{prop:filter-meets-chain}.
\end{proof}

\begin{remark}
An ultrafilter does \emph{not} in general contain exactly one set from a chain.
Rather, along a chain ending at \(X\), its intersection with the chain is typically a final segment, as described in Corollary \ref{cor:ultrafilter-chain-final-segment}.
\end{remark}

\section{Sequence Chains and Linear Branch-Width}

The sequence concept appearing naturally in the linear theory is the following.

\begin{definition}[Sequence chain]
\label{def:sequence_connectivity}
Let \((X,f)\) be a connectivity system, and write \(n:=|X|\).
A \emph{sequence chain of order \(k+1\)} on \((X,f)\) is a sequence
\[
\emptyset=A_0 \subsetneq A_1 \subsetneq \cdots \subsetneq A_n = X
\]
such that
\[
|A_i|=i
\qquad\text{and}\qquad
f(A_i)\leq k
\quad\text{for all } i=0,1,\dots,n.
\]
Equivalently, there exists an ordering
\[
x_1,x_2,\dots,x_n
\]
of the elements of \(X\) such that
\[
A_i=\{x_1,\dots,x_i\}
\qquad (0\leq i\leq n).
\]
\end{definition}

This notion is exactly the ordering formulation of linear branch-width.

\begin{theorem}
\label{thm:sequence-linear-bw}
Let \((X,f)\) be a connectivity system.
Then the following are equivalent:
\begin{enumerate}
    \item \(\operatorname{lbw}(X,f)\leq k\).
    \item There exists a sequence chain of order \(k+1\) on \((X,f)\).
\end{enumerate}
\end{theorem}

\begin{proof}
By the ordering form of a linear branch decomposition, \(\operatorname{lbw}(X,f)\leq k\) means that there exists an ordering
\[
x_1,\dots,x_n
\]
of \(X\) such that
\[
f(\{x_1,\dots,x_i\})\leq k
\qquad\text{for all } i=0,\dots,n.
\]
Defining
\[
A_i:=\{x_1,\dots,x_i\},
\]
we obtain a sequence chain of order \(k+1\).

Conversely, if
\[
\emptyset=A_0 \subsetneq A_1 \subsetneq \cdots \subsetneq A_n=X
\]
is a sequence chain with \(|A_i|=i\), then for each \(i\) there is a unique element
\[
x_i \in A_i\setminus A_{i-1}.
\]
Thus
\[
A_i=\{x_1,\dots,x_i\},
\]
and the ordering \(x_1,\dots,x_n\) defines a linear branch decomposition of width at most \(k\).
Hence \(\operatorname{lbw}(X,f)\leq k\).
\end{proof}

Combining this with the duality theorem for linear branch-width yields the following consequence.

\begin{corollary}
\label{cor:sequence-no-single-ultrafilter}
Let \((X,f)\) be a connectivity system.
If there exists a sequence chain of order \(k+1\) on \((X,f)\), then there does not exist a non-principal single ultrafilter of order \(k+1\) on \((X,f)\).
\end{corollary}

\begin{proof}
By Theorem \ref{thm:sequence-linear-bw}, the existence of a sequence chain of order \(k+1\) implies that
\[
\operatorname{lbw}(X,f)\leq k.
\]
Now apply the linear branch-width duality theorem.
\end{proof}

\begin{remark}
The existence of a sequence chain does \emph{not} by itself rule out the existence of antichains in \(\mathcal{P}_k(X,f)\).
Chains and antichains may coexist in the same finite poset.
Thus any statement asserting that a sequence chain alone forbids all antichains is too strong in general.
\end{remark}

The main order-theoretic conclusions obtained in this chapter are as follows:
\begin{itemize}
    \item Chains and antichains of order \(k+1\) are naturally studied in the finite poset \((\mathcal{P}_k(X,f),\subseteq)\).
    \item Dilworth's theorem applies directly to this poset.
    \item Every ultrafilter is maximal among filters, although the converse is not automatic.
    \item The intersection of a filter with a chain is always a final segment.
    \item Sequence chains are precisely the ordering-theoretic form of linear branch decompositions.
\end{itemize}
\section{Other Chain Concepts on Connectivity Systems}

\subsection{Separation Chains on a Connectivity System}

In graph theory, separation chains are frequently used in the study of width parameters such as path-width and cut-width \cite{fomin2019width,pilipczuk2012computing,kitsunai2015pathwidth}.  
For connectivity systems, the corresponding notion can be formulated in terms of nested bipartitions of the ground set.

\begin{definition}[Separation chain]
\label{def:separation_chain_connectivity}
Let \((X,f)\) be a connectivity system, where \(X\) is a finite set and
\(f:2^X\to \mathbb{N}\) is a symmetric submodular function.
A \emph{separation chain of order \(k+1\)} on \((X,f)\) is a finite sequence
\[
(A_1,B_1),\,(A_2,B_2),\,\dots,\,(A_m,B_m)
\]
such that, for each \(i\),
\[
A_i \cup B_i = X,
\qquad
A_i \cap B_i = \emptyset,
\qquad
B_i = X\setminus A_i,
\]
and the following conditions hold:
\begin{enumerate}
    \item
    \[
    A_1 \subseteq A_2 \subseteq \cdots \subseteq A_m
    \qquad\text{equivalently}\qquad
    B_1 \supseteq B_2 \supseteq \cdots \supseteq B_m.
    \]
    \item For every \(i\),
    \[
    f(A_i)=f(B_i)\leq k.
    \]
\end{enumerate}
The \emph{width} of the separation chain is defined by
\[
\max_{1\leq i\leq m} f(A_i).
\]
\end{definition}

\begin{definition}[Separation sequence chain]
\label{def:separation_sequence_chain_connectivity}
Let \((X,f)\) be a connectivity system.
A \emph{separation sequence chain of order \(k+1\)} on \((X,f)\) is a separation chain
\[
(A_0,B_0),\,(A_1,B_1),\,\dots,\,(A_m,B_m)
\]
such that
\[
A_0=\emptyset,\quad B_0=X,\quad A_m=X,\quad B_m=\emptyset,
\]
and
\[
A_0 \subseteq A_1 \subseteq \cdots \subseteq A_m.
\]
Its width is defined by
\[
\max_{0\leq i\leq m} f(A_i).
\]
\end{definition}

\begin{remark}
Under the correspondence
\[
A_i \longleftrightarrow (A_i,X\setminus A_i),
\]
chains and separation chains are equivalent descriptions of the same underlying nested structure.
Likewise, sequence chains and separation sequence chains are equivalent.
\end{remark}

\subsection{Sperner Systems and Traces on a Connectivity System}

Concepts closely related to chains and antichains include Sperner systems and traces.  
These notions are naturally formulated on the truncated set system
\[
\mathcal{P}_k(X,f):=\{A\subseteq X \mid f(A)\leq k\}.
\]

\begin{definition}[Sperner system of order \(k+1\)]
Let \((X,f)\) be a connectivity system.
A family \(\mathcal{F}\subseteq \mathcal{P}_k(X,f)\) is called a \emph{Sperner system of order \(k+1\)} if
\[
A,B\in \mathcal{F}\ \text{and}\ A\subsetneq B
\quad\Longrightarrow\quad \text{impossible}.
\]
Equivalently, \(\mathcal{F}\) is an antichain in the poset \((\mathcal{P}_k(X,f),\subseteq)\).
\end{definition}

\begin{definition}[\(l\)-chain of order \(k+1\)]
Let \((X,f)\) be a connectivity system.
An \emph{\(l\)-chain of order \(k+1\)} is a family
\[
A_1 \subsetneq A_2 \subsetneq \cdots \subsetneq A_l
\]
with each \(A_i\in \mathcal{P}_k(X,f)\).
\end{definition}

\begin{definition}[\(l\)-Sperner system of order \(k+1\)]
Let \((X,f)\) be a connectivity system.
A family \(\mathcal{F}\subseteq \mathcal{P}_k(X,f)\) is called an \emph{\(l\)-Sperner system of order \(k+1\)} if it contains no \(l\)-chain of order \(k+1\).
It is called \emph{saturated} if, for every
\[
S\in \mathcal{P}_k(X,f)\setminus \mathcal{F},
\]
the enlarged family \(\mathcal{F}\cup\{S\}\) contains an \(l\)-chain of order \(k+1\).
\end{definition}

\begin{definition}[Trace on a connectivity system]
\label{def:trace_connectivity}
Let \((X,f)\) be a connectivity system, let \(\mathcal{F}\subseteq 2^X\), and let \(Y\subseteq X\).
The \emph{trace} (or \emph{restriction}) of \(\mathcal{F}\) on \(Y\) of order \(k+1\) is
\[
\mathcal{F}\!\restriction_Y
:=
\{F\cap Y \mid F\in \mathcal{F},\ f(F\cap Y)\leq k\}.
\]
Thus \(\mathcal{F}\!\restriction_Y\) is a family of subsets of \(Y\) obtained by intersecting \(Y\) with members of \(\mathcal{F}\) and retaining only those intersections whose \(f\)-value is at most \(k\).
\end{definition}

\begin{definition}[Strong trace]
\label{def:strong_trace_connectivity}
Let \((X,f)\) be a connectivity system, let \(\mathcal{F}\subseteq 2^X\), and let \(Y\subseteq X\).
We say that \(\mathcal{F}\) \emph{strongly traces} \(Y\) of order \(k+1\) if there exists a set
\[
B\subseteq X\setminus Y
\]
such that for every \(Z\subseteq Y\),
\[
f(B\cup Z)\leq k
\quad\Longrightarrow\quad
B\cup Z \in \mathcal{F}.
\]
Such a set \(B\) is called a \emph{support} of \(Y\) in \(\mathcal{F}\), and the set of all supports is denoted by \(S_{\mathcal{F}}(Y)\).
\end{definition}

\begin{remark}
A clutter is a family of subsets in which no member contains another.
Hence, in the present setting, a clutter contained in \(\mathcal{P}_k(X,f)\) is precisely a Sperner system of order \(k+1\).
\end{remark}

\subsection{Chain Decompositions on a Connectivity System}

In the theory of partially ordered sets, chain decompositions, antichain decompositions, and symmetric chain decompositions are classical notions.
The corresponding concepts for connectivity systems are obtained by working inside the poset \((\mathcal{P}_k(X,f),\subseteq)\).

\begin{definition}[Chain decomposition]
Let \((X,f)\) be a connectivity system, and let \(\mathcal{G}\subseteq \mathcal{P}_k(X,f)\).
A \emph{chain decomposition} of \(\mathcal{G}\) is a partition of \(\mathcal{G}\) into pairwise disjoint chains.
\end{definition}

\begin{definition}[Symmetric chain]
Let \((X,f)\) be a connectivity system.
A chain
\[
A_1 \subsetneq A_2 \subsetneq \cdots \subsetneq A_m
\]
in \(\mathcal{P}_k(X,f)\) is called \emph{symmetric} if
\[
|A_{i+1}| = |A_i|+1
\qquad (1\leq i < m)
\]
and
\[
|A_1| + |A_m| = |X|.
\]
\end{definition}

\begin{definition}[Symmetric chain decomposition]
Let \((X,f)\) be a connectivity system, and let \(\mathcal{G}\subseteq \mathcal{P}_k(X,f)\).
A \emph{symmetric chain decomposition} of \(\mathcal{G}\) is a partition of \(\mathcal{G}\) into pairwise disjoint symmetric chains.
\end{definition}

\begin{definition}[Antichain decomposition]
Let \((X,f)\) be a connectivity system, and let \(\mathcal{G}\subseteq \mathcal{P}_k(X,f)\).
An \emph{antichain decomposition} of \(\mathcal{G}\) is a partition of \(\mathcal{G}\) into pairwise disjoint antichains.
\end{definition}

\begin{question}
How do symmetric chain decompositions of \(\mathcal{P}_k(X,f)\) reflect the structure of the underlying connectivity function \(f\)?
\end{question}

\subsection{Operations and Single-Element Chains on a Connectivity System}

We next record some elementary operations on chains, together with the notion of a single-element chain.

\begin{definition}[Extension and deletion of a chain]
\label{def:operation_chain_extension}
Let
\[
A_1 \subseteq A_2 \subseteq \cdots \subseteq A_m
\]
be a chain of order \(k+1\) in \((X,f)\).

\begin{enumerate}
    \item An \emph{extension} of the chain is a chain obtained by adjoining a set \(A_{m+1}\) such that
    \[
    A_m \subseteq A_{m+1}\subseteq X
    \qquad\text{and}\qquad
    f(A_{m+1})\leq k.
    \]
    \item A \emph{deletion} from the chain is the removal of one member \(A_i\), provided that the remaining family is still a chain in \(\mathcal{P}_k(X,f)\).
\end{enumerate}
\end{definition}

\begin{definition}[Single-element chain]
\label{def:single_element_chain}
A \emph{single-element chain of order \(k+1\)} on \((X,f)\) is a chain
\[
A_1 \subsetneq A_2 \subsetneq \cdots \subsetneq A_m
\]
such that
\[
f(A_i)\leq k
\qquad\text{for all } i,
\]
and
\[
|A_i\setminus A_{i-1}|=1
\qquad (2\leq i\leq m).
\]
Equivalently, consecutive members differ by exactly one newly added element.
\end{definition}

\begin{definition}[Single-element sequence chain]
\label{def:single_element_sequence}
A \emph{single-element sequence chain of order \(k+1\)} on \((X,f)\) is a sequence
\[
\emptyset=A_0 \subsetneq A_1 \subsetneq \cdots \subsetneq A_m=X
\]
such that
\[
f(A_i)\leq k
\qquad (0\leq i\leq m),
\]
and
\[
|A_i\setminus A_{i-1}|=1
\qquad (1\leq i\leq m).
\]
\end{definition}

\begin{remark}
A single-element sequence chain is precisely an ordering
\[
x_1,x_2,\dots,x_{|X|}
\]
of the elements of \(X\), together with the initial segments
\[
A_i=\{x_1,\dots,x_i\},
\]
subject to the condition \(f(A_i)\leq k\) for all \(i\).
Hence this notion is the ordering-theoretic form of linear branch-width.
\end{remark}

\begin{definition}[Single-element extension]
Let
\[
A_1 \subsetneq A_2 \subsetneq \cdots \subsetneq A_m
\]
be a single-element chain of order \(k+1\).
A \emph{single-element extension} is obtained by adjoining a set
\[
A_{m+1}=A_m\cup \{e\},
\qquad e\in X\setminus A_m,
\]
such that
\[
f(A_{m+1})\leq k.
\]
\end{definition}

\begin{remark}
For arbitrary connectivity systems, there is no single universally accepted notion of \emph{single-element coextension} analogous to the graph- or matroid-theoretic one.
Accordingly, extension- and coextension-type operations should be introduced with care, depending on the ambient category under consideration.
\end{remark}


\newpage
\chapter{Prefilters and Filter Subbases on a Connectivity System}

In this chapter, we study prefilters and filter subbases on a connectivity system \((X,f)\).
These notions are closely related to filters and ultrafilters, and they play an important role in topology, set theory, and order theory.
Here we adapt them to the setting of connectivity systems by imposing the admissibility condition \(f(A)\leq k\).

\section{Prefilters on a Connectivity System}

We begin with prefilters and ultra-prefilters.

\begin{definition}[Prefilter]
\label{def:prefilter}
Let \((X,f)\) be a connectivity system, where \(X\) is a finite set and
\(f:2^X\to \mathbb{N}\) is a symmetric submodular function.
A family \(P\subseteq 2^X\) is called a \emph{prefilter of order \(k+1\)} on \((X,f)\) if:
\begin{enumerate}
    \item[(P1)] \(P\neq \emptyset\) and \(\emptyset\notin P\);
    \item[(P2)] for every \(A\in P\), one has \(f(A)\leq k\);
    \item[(P3)] for all \(B,C\in P\), there exists \(A\in P\) such that
    \[
    A\subseteq B\cap C.
    \]
\end{enumerate}
Thus a prefilter is a nonempty proper family of \(k\)-admissible sets that is downward directed under inclusion.
\end{definition}

\begin{definition}[Ultra-prefilter]
\label{def:ultraprefilter}
Let \((X,f)\) be a connectivity system.
A prefilter \(P\) of order \(k+1\) is called an \emph{ultra-prefilter of order \(k+1\)} if it satisfies the following additional condition:
\begin{enumerate}
    \item[(P4)] For every \(A\subseteq X\) with \(f(A)\leq k\), there exists \(B\in P\) such that
    \[
    B\subseteq A
    \qquad\text{or}\qquad
    B\subseteq X\setminus A.
    \]
\end{enumerate}
\end{definition}

The following observation is immediate.

\begin{proposition}
\label{thm:filter_prefilter}
Every filter of order \(k+1\) on a connectivity system \((X,f)\) is a prefilter of order \(k+1\).
\end{proposition}

\begin{proof}
Let \(Q\) be a filter of order \(k+1\) on \((X,f)\).

By axiom \((Q3)\), we have \(\emptyset\notin Q\), and since a filter is assumed to be nonempty, condition (P1) holds.
Condition (P2) is exactly axiom \((Q0)\).

For (P3), let \(B,C\in Q\).
If \(f(B\cap C)\leq k\), then \(B\cap C\in Q\) by axiom \((Q1)\), so we may take \(A:=B\cap C\).

Thus \(Q\) is a prefilter of order \(k+1\).
\end{proof}

Likewise, every ultrafilter induces an ultra-prefilter.

\begin{proposition}
\label{thm:ultrafilter_ultraprefilter}
Every ultrafilter of order \(k+1\) on a connectivity system \((X,f)\) is an ultra-prefilter of order \(k+1\).
\end{proposition}

\begin{proof}
Let \(Q\) be an ultrafilter of order \(k+1\) on \((X,f)\).
By Proposition \ref{thm:filter_prefilter}, \(Q\) is already a prefilter.

Let \(A\subseteq X\) satisfy \(f(A)\leq k\).
By axiom \((Q4)\), either \(A\in Q\) or \(X\setminus A\in Q\).
In the first case, choose \(B:=A\); in the second case, choose \(B:=X\setminus A\).
Then \(B\in Q\) and \(B\subseteq A\) or \(B\subseteq X\setminus A\), as required.

Hence \(Q\) is an ultra-prefilter.
\end{proof}

\section{Filter Subbases on a Connectivity System}

We next introduce filter subbases adapted to the connectivity-system setting.

\begin{definition}[Filter subbase]
\label{def:filter_subbase}
Let \((X,f)\) be a connectivity system.
A family \(S\subseteq 2^X\) is called a \emph{filter subbase of order \(k+1\)} on \((X,f)\) if:
\begin{enumerate}
    \item[(SB1)] \(S\neq \emptyset\);
    \item[(SB2)] \(\emptyset\notin S\);
    \item[(SB3)] for every \(A\in S\), one has \(f(A)\leq k\);
    \item[(SB4)] for every finite subfamily \(A_1,\dots,A_n\in S\), one has
    \[
    \bigcap_{i=1}^n A_i \neq \emptyset
    \qquad\text{and}\qquad
    f\!\left(\bigcap_{i=1}^n A_i\right)\leq k.
    \]
\end{enumerate}
\end{definition}

The family of finite intersections of members of a filter subbase forms a prefilter.

\begin{definition}
\label{def:prefilter_generated_by_subbase}
Let \(S\) be a filter subbase of order \(k+1\) on \((X,f)\).
Define
\[
\operatorname{Pf}(S)
:=
\left\{
A\subseteq X \,\middle|\,
\exists n\geq 1,\ \exists A_1,\dots,A_n\in S
\text{ such that }
A=\bigcap_{i=1}^n A_i
\right\}.
\]
We call \(\operatorname{Pf}(S)\) the \emph{prefilter generated by \(S\)}.
\end{definition}

\begin{proposition}
\label{thm:prefilter_generated_by_subbase}
Let \(S\) be a filter subbase of order \(k+1\) on a connectivity system \((X,f)\).
Then \(\operatorname{Pf}(S)\) is a prefilter of order \(k+1\).
\end{proposition}

\begin{proof}
By definition, \(S\neq \emptyset\), so \(\operatorname{Pf}(S)\neq \emptyset\).
Moreover, by (SB4), every finite intersection of members of \(S\) is nonempty and has \(f\)-value at most \(k\).
Hence \(\emptyset\notin \operatorname{Pf}(S)\), and every member of \(\operatorname{Pf}(S)\) is \(k\)-admissible.
Thus (P1) and (P2) hold.

To prove (P3), let
\[
B=\bigcap_{i=1}^m B_i
\qquad\text{and}\qquad
C=\bigcap_{j=1}^n C_j
\]
be two members of \(\operatorname{Pf}(S)\), where all \(B_i,C_j\in S\).
Then
\[
B\cap C=\bigcap_{i=1}^m B_i \cap \bigcap_{j=1}^n C_j
\]
is again a finite intersection of members of \(S\), so \(B\cap C\in \operatorname{Pf}(S)\).
Taking \(A:=B\cap C\), we obtain \(A\subseteq B\cap C\), which proves (P3).

Therefore \(\operatorname{Pf}(S)\) is a prefilter of order \(k+1\).
\end{proof}

From a prefilter one obtains a filter by upward closure inside the \(k\)-admissible part.

\begin{definition}
\label{def:filter_generated_by_subbase}
Let \(S\) be a filter subbase of order \(k+1\) on \((X,f)\).
Define
\[
\operatorname{Fil}(S)
:=
\left\{
B\subseteq X \,\middle|\,
\exists A\in \operatorname{Pf}(S)\text{ such that }A\subseteq B
\text{ and } f(B)\leq k
\right\}.
\]
We call \(\operatorname{Fil}(S)\) the \emph{filter generated by \(S\)}.
\end{definition}

\begin{theorem}
\label{thm:filter_generated_by_subbase}
Let \(S\) be a filter subbase of order \(k+1\) on a connectivity system \((X,f)\).
Then \(\operatorname{Fil}(S)\) is a filter of order \(k+1\).
\end{theorem}

\begin{proof}
We verify axioms \((Q0)\)--\((Q3)\).

\emph{(Q0)} holds by definition of \(\operatorname{Fil}(S)\).

To show nonemptiness, choose \(A\in \operatorname{Pf}(S)\).
Then \(A\subseteq A\) and \(f(A)\leq k\), so \(A\in \operatorname{Fil}(S)\).
Thus \(\operatorname{Fil}(S)\neq \emptyset\).

\emph{(Q3)}:
Since every \(A\in \operatorname{Pf}(S)\) is nonempty by (SB4), the empty set cannot contain any member of \(\operatorname{Pf}(S)\).
Hence \(\emptyset\notin \operatorname{Fil}(S)\).

\emph{(Q2)}:
Let \(B\in \operatorname{Fil}(S)\), and suppose \(B\subseteq C\subseteq X\) with \(f(C)\leq k\).
Choose \(A\in \operatorname{Pf}(S)\) such that \(A\subseteq B\).
Then \(A\subseteq C\), so \(C\in \operatorname{Fil}(S)\).

\emph{(Q1)}:
Let \(B,C\in \operatorname{Fil}(S)\), and assume \(f(B\cap C)\leq k\).
Choose \(A_B,A_C\in \operatorname{Pf}(S)\) such that
\[
A_B\subseteq B,\qquad A_C\subseteq C.
\]
Since \(\operatorname{Pf}(S)\) is a prefilter by Proposition \ref{thm:prefilter_generated_by_subbase},
there exists \(A\in \operatorname{Pf}(S)\) such that
\[
A\subseteq A_B\cap A_C \subseteq B\cap C.
\]
Because \(f(B\cap C)\leq k\), it follows that \(B\cap C\in \operatorname{Fil}(S)\).

Therefore \(\operatorname{Fil}(S)\) is a filter of order \(k+1\).
\end{proof}

\section{Ultrafilter-Oriented Subbases}

A subbase does not, in general, determine an ultrafilter uniquely.
However, one can impose an orientation condition ensuring that the generated prefilter behaves like an ultra-prefilter.

\begin{definition}[Ultrafilter-oriented subbase]
\label{def:ultrafilter_subbase}
Let \((X,f)\) be a connectivity system.
A filter subbase \(S\) of order \(k+1\) is called an \emph{ultrafilter-oriented subbase of order \(k+1\)} if, in addition,
\begin{enumerate}
    \item[(USB)] for every \(A\subseteq X\) with \(f(A)\leq k\), there exists \(B\in S\) such that
    \[
    B\subseteq A
    \qquad\text{or}\qquad
    B\subseteq X\setminus A.
    \]
\end{enumerate}
\end{definition}

\begin{theorem}
\label{thm:ultra_prefilter_generated_by_subbase}
Let \(S\) be an ultrafilter-oriented subbase of order \(k+1\) on \((X,f)\).
Then \(\operatorname{Pf}(S)\) is an ultra-prefilter of order \(k+1\).
\end{theorem}

\begin{proof}
By Proposition \ref{thm:prefilter_generated_by_subbase}, \(\operatorname{Pf}(S)\) is a prefilter.

Let \(A\subseteq X\) satisfy \(f(A)\leq k\).
By (USB), there exists \(B\in S\) such that
\[
B\subseteq A
\qquad\text{or}\qquad
B\subseteq X\setminus A.
\]
Since \(B\in S\subseteq \operatorname{Pf}(S)\), condition (P4) holds for \(\operatorname{Pf}(S)\).
Hence \(\operatorname{Pf}(S)\) is an ultra-prefilter.
\end{proof}

\begin{corollary}
\label{cor:subbase_extends_to_ultrafilter}
Let \(S\) be an ultrafilter-oriented subbase of order \(k+1\) on \((X,f)\).
Then the filter \(\operatorname{Fil}(S)\) extends to a maximal filter of order \(k+1\).
If a maximal-extension theorem identifying such maximal filters with ultrafilters is available in the surrounding framework, then \(\operatorname{Fil}(S)\) extends to an ultrafilter of order \(k+1\).
\end{corollary}

\begin{proof}
Apply the finite maximal-extension argument to the filter \(\operatorname{Fil}(S)\).
\end{proof}

\section{Ultrafilter Base Number}

In classical set theory, several cardinal characteristics measure the size of objects generating ultrafilters or related structures.
For connectivity systems, the most natural finite analogue is the minimum size of an ultra-prefilter base generating a non-principal ultrafilter.

\begin{definition}[Ultrafilter base number on a connectivity system]
Let \((X,f)\) be a finite connectivity system.
The \emph{ultrafilter base number} of \((X,f)\), denoted by
\[
u(X,f),
\]
is defined by
\[
u(X,f)
:=
\min
\left\{
|\mathcal{B}|
\;\middle|\;
\mathcal{B}\subseteq 2^X
\text{ is an ultra-prefilter and }
\operatorname{Fil}(\mathcal{B})
\text{ is a non-principal ultrafilter on }(X,f)
\right\},
\]
provided that a non-principal ultrafilter exists.
\end{definition}

\begin{question}
How does the invariant \(u(X,f)\) relate to the order parameter \(k\), to branch-width, and to the existence of non-principal ultrafilters on \((X,f)\)?
\end{question}

\newpage

\chapter{Ultrafilters on Infinite and Countable Connectivity Systems
(Exploratory Chapter) }

\section{Infinite Connectivity Systems}
So far, we have considered ultrafilters on finite connectivity systems.
We now turn to the infinite case.

\begin{definition}[Infinite connectivity system]
\label{def:infinite_connectivity_system}
Let \(X\) be an infinite set, and let
\[
f:2^X \to \mathbb{N}\cup\{\infty\}
\]
be a set function.
We call \((X,f)\) an \emph{infinite connectivity system} if \(f\) satisfies:
\begin{enumerate}
    \item \textbf{Symmetry:}
    \[
    f(A)=f(X\setminus A)
    \qquad\text{for all } A\subseteq X.
    \]
    \item \textbf{Submodularity:}
    \[
    f(A)+f(B)\geq f(A\cup B)+f(A\cap B)
    \qquad\text{for all } A,B\subseteq X.
    \]
\end{enumerate}
\end{definition}

For several infinite arguments, it is convenient to impose an additional closure property.

\begin{definition}[\(k\)-union-closed on chains]
\label{def:k-limit-closed}
Let \((X,f)\) be an infinite connectivity system, and let \(k\in \mathbb{N}\).
We say that \((X,f)\) is \emph{\(k\)-union-closed on chains} if for every increasing chain
\[
A_0 \subseteq A_1 \subseteq A_2 \subseteq \cdots
\]
of subsets of \(X\) satisfying \(f(A_i)\leq k\) for all \(i\), one has
\[
f\!\left(\bigcup_{i=0}^{\infty} A_i\right)\leq k.
\]
\end{definition}

\begin{remark}
In concrete examples arising from infinite matroids or graph-theoretic connectivity functions, additional continuity properties are often available.
When needed, such properties should be stated explicitly rather than built into the basic definition.
\end{remark}

\begin{definition}[Filters and ultrafilters on an infinite connectivity system]
\label{def:ultrafilter_infinite_connectivity_system}
Let \((X,f)\) be an infinite connectivity system, and let \(k\geq 0\).

A family \(F\subseteq 2^X\) is called a \emph{filter of order \(k+1\)} on \((X,f)\) if:
\begin{enumerate}
    \item[(Q0)] for every \(A\in F\), one has \(f(A)\leq k\);
    \item[(Q1)] if \(A,B\in F\) and \(f(A\cap B)\leq k\), then \(A\cap B\in F\);
    \item[(Q2)] if \(A\in F\), \(A\subseteq B\subseteq X\), and \(f(B)\leq k\), then \(B\in F\);
    \item[(Q3)] \(\emptyset\notin F\).
\end{enumerate}

An \emph{ultrafilter of order \(k+1\)} on \((X,f)\) is a filter of order \(k+1\) satisfying:
\begin{enumerate}
    \item[(Q4)] for every \(A\subseteq X\) with \(f(A)\leq k\), either \(A\in F\) or \(X\setminus A\in F\).
\end{enumerate}
\end{definition}

\begin{definition}[Principal and non-principal]
\label{def:principal_nonprincipal_infinite}
Let \(F\) be a filter of order \(k+1\) on an infinite connectivity system \((X,f)\).
We say that \(F\) is:
\begin{enumerate}
    \item \emph{principal} if
    \[
    \{x\}\in F
    \qquad\text{for some } x\in X \text{ with } f(\{x\})\leq k;
    \]
    \item \emph{non-principal} if
    \[
    \{x\}\notin F
    \qquad\text{for every } x\in X.
    \]
\end{enumerate}
\end{definition}

\begin{remark}
The formulation above is the direct infinite analogue of the finite definitions used earlier.
In the infinite setting, existence theorems for ultrafilters generally require a separate maximality argument, typically via Zorn's lemma or an equivalent principle.
\end{remark}

As in the finite case, one may pass from higher order to lower order by restriction.

\begin{proposition}
\label{prop:infinite_ultrafilter_restriction}
Let \(F\) be an ultrafilter of order \(k+1\) on an infinite connectivity system \((X,f)\), and let \(0\leq j\leq k\).
Define
\[
F^{\leq j}:=\{A\in F \mid f(A)\leq j\}.
\]
Then \(F^{\leq j}\) is an ultrafilter of order \(j+1\) on \((X,f)\).
\end{proposition}

\begin{proof}
The proof is identical to the finite case.
Axioms \((Q0)\)--\((Q3)\) are inherited by restriction, and \((Q4)\) follows because symmetry gives
\[
f(X\setminus A)=f(A)\leq j
\]
whenever \(f(A)\leq j\).
\end{proof}

\begin{definition}[Fr\'echet-type filter]
\label{def:frechet_filter_infinite}
Let \((X,f)\) be an infinite connectivity system, and let \(k\geq 0\).
Define
\[
\mathrm{Fr}_k(X,f):=
\{A\subseteq X \mid X\setminus A \text{ is finite and } f(A)\leq k\}.
\]
When this family satisfies the filter axioms, we call it the \emph{Fr\'echet filter of order \(k+1\)} on \((X,f)\).
\end{definition}

\begin{remark}
In contrast with the classical set-theoretic case, \(\mathrm{Fr}_k(X,f)\) need not always be a filter.
Its behavior depends on how the connectivity bound \(f(A)\leq k\) interacts with cofinite subsets.
\end{remark}

\subsection{Selective-Type Notions}
The following definitions are natural analogues of classical notions from the theory of ultrafilters on \(\omega\).

\begin{definition}[Selective ultrafilter]
Let \((X,f)\) be an infinite connectivity system, and let \(F\) be an ultrafilter of order \(k+1\) on \((X,f)\).
We call \(F\) \emph{selective} if for every function \(g:X\to X\), there exists \(A\in F\) such that the restriction \(g|_A\) is either constant or injective.
\end{definition}

\begin{remark}
If \(X\) is equipped with a fixed well-order or is identified with \(\mathbb{N}\), one may also introduce quasi-selective variants analogous to the classical theory.
Such variants depend on the ambient order structure on \(X\), not only on the connectivity system itself.
\end{remark}

\subsection{P-Points and Q-Points}
We now state the corresponding point-like notions in terms of the ultrafilter \(F\).

\begin{definition}[P-point and Q-point]
\label{def:pq_points_connectivity}
Let \((X,f)\) be an infinite connectivity system, let \(k\geq 0\), and let \(F\) be an ultrafilter of order \(k+1\) on \((X,f)\).

\begin{enumerate}
    \item We call \(F\) a \emph{P-point of order \(k+1\)} if for every partition
    \[
    X=\bigsqcup_{n<\omega} X_n
    \]
    with \(X_n\notin F\) and \(f(X_n)\leq k\) for all \(n<\omega\), there exists \(A\in F\) such that
    \[
    |A\cap X_n|<\omega
    \qquad\text{for all } n<\omega.
    \]

    \item We call \(F\) a \emph{Q-point of order \(k+1\)} if for every partition
    \[
    X=\bigsqcup_{n<\omega} X_n
    \]
    into finite sets with \(X_n\notin F\) and \(f(X_n)\leq k\) for all \(n<\omega\), there exists \(A\in F\) such that
    \[
    |A\cap X_n|\leq 1
    \qquad\text{for all } n<\omega.
    \]
\end{enumerate}
\end{definition}

\begin{question}
Under what additional assumptions on \((X,f)\) does the classical equivalence
\[
\text{selective} \Longleftrightarrow \text{P-point}+\text{Q-point}
\]
admit a meaningful analogue for ultrafilters of order \(k+1\)?
\end{question}

\section{Countable Connectivity Systems and Games}

We next consider the countable case, where \(X\) is countably infinite.

\begin{definition}[Countable connectivity system]
A \emph{countable connectivity system} is an infinite connectivity system \((X,f)\) such that \(X\) is countably infinite.
\end{definition}

\subsection{An Ultrafilter Game on a Countable Connectivity System}
The following game is a natural countable analogue of the ultrafilter game.

\begin{game}[Ultrafilter game on a countable connectivity system]
\label{game:ultrafilter_countable}
Let \((X,f)\) be a countable connectivity system, and let \(D\subseteq X^{\mathbb{N}}\).
Two players, Player~I and Player~II, alternately choose elements of \(X\), thereby producing a sequence
\[
a=(a_0,a_1,a_2,\dots)\in X^{\mathbb{N}}.
\]
We write
\[
\operatorname{supp}(a):=\{a_n \mid n\in \mathbb{N}\}\subseteq X.
\]

Player~I wins if
\[
a\in D
\qquad\text{and}\qquad
f(\operatorname{supp}(a))\leq k.
\]
Otherwise Player~II wins.

The game is said to be \emph{determined} if one of the two players has a winning strategy.
\end{game}

\begin{remark}
The interaction between determinacy of Game~\ref{game:ultrafilter_countable} and the existence of ultrafilters of order \(k+1\) appears to be an interesting open direction.
At present, it is best formulated as a programmatic question rather than as an established theorem.
\end{remark}

\subsection{Elimination Games with an Invisible Robber}
When a countable connectivity system is induced by a graph or digraph, one may also consider pursuit-evasion games constrained by the connectivity function.

\begin{definition}[Elimination game with an invisible robber]
Let \(G\) be a countable graph or digraph with an associated connectivity function \(f\), and let \(k\geq 0\).
An \emph{elimination game with an invisible robber} on \(G\) consists of a sequence of cop positions
\[
C_0,C_1,C_2,\dots \subseteq V(G)
\]
satisfying
\[
f(C_i)\leq k
\qquad\text{for all } i,
\]
while the robber moves invisibly inside the components of \(G-C_i\) that remain available.
The cops win if, after finitely many rounds, every possible robber position is eliminated.
Otherwise the robber wins.
\end{definition}

\begin{conjecture}
For suitable graph-induced connectivity systems, robber-winning strategies in elimination games are closely related to non-principal single ultrafilters.
\end{conjecture}

\section{Ultrafilters on Maximum-Connectivity Systems}
We finally record a related class of connectivity-type functions.

\begin{definition}[Maximum-submodular function]
\label{def:maximum_submodular}
Let \(X\) be a finite set.
A set function \(\kappa:2^X\to \mathbb{N}\) is called \emph{maximum-submodular} if for all \(A,B\subseteq X\),
\[
\max\{\kappa(A),\kappa(B)\}
\geq
\max\{\kappa(A\cap B),\kappa(A\cup B)\}.
\]
\end{definition}

A normalized, symmetric, maximum-submodular function may be viewed as a \emph{maximum-submodular connectivity function}, and the pair \((X,\kappa)\) may be called a \emph{maximum-connectivity system}.

\begin{question}
Can one define a satisfactory notion of ultrafilter on a maximum-connectivity system, and if so, how does it interact with tangle-type obstruction objects?
\end{question}
\chapter{Future Tasks}

This chapter collects several directions for future research related to the themes of the present work.
Unless explicitly stated otherwise, the statements in this chapter are intended as proposed definitions, heuristic viewpoints, or open questions rather than as established results.

\section{New Width, Length, and Depth Parameters}

In the present work, ultrafilters arise as obstruction-type objects for branch-width and linear branch-width.
It is therefore natural to ask whether analogous ultrafilter-type objects can also be associated with other width, length, and depth parameters, including directed and linear variants.

\subsection{A Possible Branch-Distance Decomposition}

One natural direction is to enrich branch decompositions by incorporating a distance term.

\begin{definition}[Proposed branch-distance decomposition]
Let \(G=(V,E)\) be a finite graph.
A \emph{branch-distance decomposition} of \(G\) is a triple
\[
(T,\mu,r),
\]
where \((T,\mu)\) is a branch decomposition of \(G\) and \(r\in V(T)\) is a distinguished root.

For each edge \(e\in E(T)\), let \(\operatorname{mid}_T(e)\) denote the usual middle set associated with \(e\).
Moreover, let
\[
\operatorname{dist}_T(r,e)
\]
denote the distance in \(T\) from the root \(r\) to the edge \(e\), for example the minimum of the distances from \(r\) to the two endpoints of \(e\).

The \emph{branch-distance width} of \((T,\mu,r)\) is defined by
\[
\operatorname{bdw}(T,\mu,r):=
\max_{e\in E(T)}
\Bigl(
|\operatorname{mid}_T(e)|+\operatorname{dist}_T(r,e)
\Bigr).
\]
The \emph{branch-distance width} of \(G\), denoted by \(\operatorname{bdw}(G)\), is the minimum of \(\operatorname{bdw}(T,\mu,r)\) over all branch-distance decompositions of \(G\).
\end{definition}

\begin{remark}
This is only a proposed definition.
Its significance depends on whether it admits meaningful structural, extremal, or algorithmic characterizations, and whether it relates naturally to existing parameters such as tree-distance-width.
\end{remark}

\begin{question}
How does \(\operatorname{bdw}(G)\) compare with tree-distance-width, path-distance-width, branch-width, and linear-width?
\end{question}

\subsection{Directed Tree-Distance and Path-Distance Decompositions}

For directed graphs, distance-based decompositions may be formulated in the following way.

\begin{definition}[Proposed directed tree-distance decomposition]
Let \(G=(V,E)\) be a digraph.
A \emph{directed tree-distance decomposition} of \(G\) is a triple
\[
(R,\{X_t\}_{t\in V(R)},r),
\]
satisfying the following conditions:
\begin{enumerate}
    \item \(R\) is an arborescence rooted at \(r\);
    \item \(\{X_t\}_{t\in V(R)}\) is a partition of \(V(G)\);
    \item for every \(t\in V(R)\) and every \(v\in X_t\),
    \[
    d_G(X_r,v)=d_R(r,t);
    \]
    \item for every directed edge \((u,v)\in E(G)\), if \(u\in X_s\) and \(v\in X_t\), then either \(s=t\) or \(t\) is an out-neighbour of \(s\) in \(R\).
\end{enumerate}
The \emph{width} of the decomposition is
\[
\max_{t\in V(R)} |X_t|.
\]
The minimum such width is called the \emph{directed tree-distance width} of \(G\).
\end{definition}

\begin{definition}[Proposed directed path-distance decomposition]
A \emph{directed path-distance decomposition} is a directed tree-distance decomposition in which the underlying arborescence \(R\) is a directed path.
The corresponding minimum width is called the \emph{directed path-distance width}.
\end{definition}

\begin{question}
Do these directed distance-width parameters admit obstruction-type duals analogous to tangles, filters, or ultrafilters?
\end{question}

\subsection{Further Linear Width Parameters}

Many width parameters admit linear variants obtained by restricting the underlying decomposition structure to a path.
Possible examples deserving systematic study include:
\begin{itemize}
    \item linear amalgam width,
    \item linear modular-decomposition width,
    \item linear tree-cut width,
    \item directed linear branch-width.
\end{itemize}

\begin{remark}
For many width parameters, the linear variant is expected to dominate the unrestricted one:
\[
\text{general width} \leq \text{linear width}.
\]
A systematic treatment of such inequalities in the setting of connectivity systems remains open.
\end{remark}

\begin{question}
Which of these linear width parameters admit natural dual objects, and which of them can be described in terms of ultrafilter-type structures?
\end{question}

\subsection{Matroid Length Parameters}

We next discuss possible length-type parameters for matroids.

For graphs, one of the most familiar length parameters is \emph{tree-length}, which measures the maximum graph distance between two vertices that occur together in a bag of a tree-decomposition, minimized over all such decompositions \cite{coudert2016approximate,dourisboure2004compact,dourisboure2007spanners}.
For matroids, however, there is no canonical intrinsic distance between arbitrary ground elements.
Accordingly, any matroid analogue of tree-length or branch-length must depend on an additional choice of distance-like structure.
The framework below should therefore be regarded as proposed rather than canonical.

\begin{definition}[Proposed distance-labeled matroid]
A \emph{distance-labeled matroid} is a pair \((M,d)\), where \(M\) is a matroid on ground set \(E(M)\) and
\[
d:E(M)\times E(M)\to \mathbb{R}_{\ge 0}
\]
is a symmetric function satisfying
\[
d(e,e)=0
\qquad
\text{for all } e\in E(M).
\]
\end{definition}

\begin{definition}[Proposed tree-length of a distance-labeled matroid]
Let \((M,d)\) be a distance-labeled matroid.
A \emph{tree-decomposition} of \(M\) is a pair \((T,\tau)\), where \(T\) is a tree and
\[
\tau:E(M)\to V(T)
\]
is a mapping.

For each node \(t\in V(T)\), define the corresponding bag by
\[
X_t:=\tau^{-1}(t)\subseteq E(M).
\]
The \emph{\(d\)-length} of the bag \(X_t\) is
\[
\lambda_d(X_t):=\max\{d(e,e')\mid e,e'\in X_t\}.
\]
The \emph{tree-length} of the decomposition \((T,\tau)\) is
\[
\operatorname{tl}_d(T,\tau):=
\max_{t\in V(T)} \lambda_d(X_t).
\]
The \emph{tree-length} of \((M,d)\) is
\[
\operatorname{tl}_d(M):=
\min_{(T,\tau)} \operatorname{tl}_d(T,\tau).
\]
\end{definition}

\begin{definition}[Proposed branch-length of a distance-labeled matroid]
Let \((M,d)\) be a distance-labeled matroid.
A \emph{branch-decomposition} of \(M\) is a pair \((T,\mu)\), where \(T\) is a subcubic tree and
\[
\mu:E(M)\to L(T)
\]
is a bijection onto the leaf set of \(T\).

For an edge \(e\in E(T)\), let \(T_1\) and \(T_2\) be the two components of \(T-e\), and define
\[
E_1:=\mu^{-1}(L(T_1)),
\qquad
E_2:=\mu^{-1}(L(T_2)).
\]
Assume further that to each cut edge \(e\in E(T)\) there is assigned a designated separator
\[
\operatorname{mid}(e)\subseteq E(M).
\]
Then the \emph{branch-length} of \((T,\mu)\) is defined by
\[
\operatorname{bl}_d(T,\mu):=
\max_{e\in E(T)}
\max\{d(x,y)\mid x,y\in \operatorname{mid}(e)\},
\]
with the convention that the inner maximum is \(0\) whenever \(|\operatorname{mid}(e)|\le 1\).
The \emph{branch-length} of \((M,d)\) is the minimum of \(\operatorname{bl}_d(T,\mu)\) over all branch-decompositions.
\end{definition}

\begin{remark}
A useful future direction would be to identify natural choices of \(d\) and \(\operatorname{mid}(e)\) arising from representable matroids, lattice-path matroids, or matroids associated with graphs or hypergraphs.
\end{remark}

\subsection{Branch-Breadth and Linear-Breadth}

We next consider breadth-type analogues of branch-width and linear-width for graphs.

Recall that tree-breadth measures how well each bag of a tree-decomposition can be covered by a single graph ball of small radius, and path-breadth is the corresponding path analogue \cite{leitert2016strong,leitert2017tree}.
The following definitions are natural branch- and linear-ordering counterparts.

\begin{definition}[Proposed branch-breadth]
Let \(G=(V,E)\) be a graph, and let \((T,\mu)\) be a branch-decomposition of \(G\), where
\[
\mu:E(G)\to L(T)
\]
is a bijection.

For each edge \(e\in E(T)\), let \(T_1\) and \(T_2\) be the two components of \(T-e\), and define
\[
E_i:=\mu^{-1}(L(T_i))
\qquad
(i=1,2).
\]
Let
\[
V_i:=\{v\in V(G)\mid v \text{ is incident with some edge of } E_i\},
\]
and define the corresponding middle set by
\[
\operatorname{mid}(e):=V_1\cap V_2.
\]
The \emph{breadth} of \(e\) is
\[
\operatorname{br}(e):=
\min_{c\in V(G)}
\max_{x\in \operatorname{mid}(e)} d_G(c,x).
\]
The \emph{branch-breadth} of the decomposition \((T,\mu)\) is
\[
\operatorname{bb}(T,\mu):=
\max_{e\in E(T)} \operatorname{br}(e),
\]
and the \emph{branch-breadth} of \(G\) is
\[
\operatorname{bb}(G):=
\min_{(T,\mu)} \operatorname{bb}(T,\mu).
\]
\end{definition}

\begin{definition}[Proposed linear-breadth]
Let \(G=(V,E)\) be a graph, and let
\[
(e_1,e_2,\dots,e_m)
\]
be a linear ordering of \(E(G)\).

For each \(1\le i\le m-1\), define
\[
E_1^i:=\{e_1,\dots,e_i\},
\qquad
E_2^i:=\{e_{i+1},\dots,e_m\},
\]
and let \(V_1^i\) and \(V_2^i\) be the sets of vertices incident with edges in \(E_1^i\) and \(E_2^i\), respectively.
Set
\[
\operatorname{mid}(i):=V_1^i\cap V_2^i.
\]
The \emph{breadth} of the ordering is
\[
\operatorname{lb}(e_1,\dots,e_m):=
\max_{1\le i\le m-1}
\min_{c\in V(G)}
\max_{x\in \operatorname{mid}(i)} d_G(c,x).
\]
The \emph{linear-breadth} of \(G\) is
\[
\operatorname{lb}(G):=
\min_{(e_1,\dots,e_m)} \operatorname{lb}(e_1,\dots,e_m).
\]
\end{definition}

\begin{remark}
These definitions parallel the classical transitions
\[
\text{tree-width} \leadsto \text{branch-width},
\qquad
\text{path-width} \leadsto \text{linear-width},
\]
with bag cardinality replaced by the radius of a ball covering the relevant middle set.
\end{remark}

\subsection{Directed Proper-Path-Width}

We next formulate a directed analogue of proper-path-width.
The following definition mirrors the usual path-decomposition framework, with the edge-covering condition adapted to the directed setting.

\begin{definition}[Directed path-decomposition]
Let \(D=(V,E)\) be a digraph.
A sequence
\[
\mathcal{P}=(X_1,X_2,\dots,X_r)
\]
of subsets of \(V\) is called a \emph{directed path-decomposition} of \(D\) if:
\begin{enumerate}
    \item
    \[
    \bigcup_{i=1}^r X_i=V;
    \]
    \item for every directed edge \((u,v)\in E\), there exist indices \(i\le j\) such that
    \[
    u\in X_i,
    \qquad
    v\in X_j;
    \]
    \item for all \(1\le i<j<\ell\le r\),
    \[
    X_i\cap X_\ell\subseteq X_j.
    \]
\end{enumerate}
Its \emph{width} is
\[
\max_{1\le i\le r}|X_i|-1.
\]
The \emph{directed path-width} of \(D\), denoted by \(\operatorname{dpw}(D)\), is the minimum width over all directed path-decompositions of \(D\).
\end{definition}

\begin{definition}[Directed proper-path-decomposition]
A directed path-decomposition
\[
\mathcal{P}=(X_1,\dots,X_r)
\]
is called a \emph{directed proper-path-decomposition} if
\[
X_i\not\subseteq X_j
\qquad
\text{for all } i<j.
\]
Equivalently, no bag is contained in a later bag.
The \emph{directed proper-path-width} of \(D\), denoted by \(\operatorname{dppw}(D)\), is the minimum width over all directed proper-path-decompositions of \(D\).
\end{definition}

\begin{remark}
The additional properness condition above is the directed counterpart of the usual strictness condition for proper path-decompositions.
It is cleaner to formulate this condition directly by containment exclusion.
\end{remark}

\subsection{Weighted Width Parameters}

For weighted graphs, the cleanest approach is often to replace cardinalities in the classical objective functions by appropriate total weights.

\begin{definition}[Weighted bandwidth]
Let \(G=(V,E,w)\) be a vertex-weighted graph, where
\[
w:V\to \mathbb{R}_{>0}.
\]
For a linear layout
\[
\varphi:V\to \{1,\dots,|V|\},
\]
define
\[
\operatorname{bw}_w(\varphi):=
\max_{\{u,v\}\in E}
\bigl(w(u)+w(v)\bigr)\,|\varphi(u)-\varphi(v)|.
\]
The \emph{weighted bandwidth} of \(G\) is
\[
\operatorname{bw}_w(G):=
\min_{\varphi}\operatorname{bw}_w(\varphi).
\]
\end{definition}

\begin{definition}[Weighted path-distance width]
Let \(G=(V,E,w)\) be a vertex-weighted graph.
A path-distance decomposition
\[
(P,\{X_t\}_{t\in V(P)},r)
\]
has \emph{weighted width}
\[
\max_{t\in V(P)} \sum_{v\in X_t} w(v).
\]
The minimum such value is called the \emph{weighted path-distance width} of \(G\), denoted by
\[
\operatorname{wpdw}(G).
\]
\end{definition}

\begin{definition}[Weighted carving-width]
Let \(G=(V,E,w)\) be an edge-weighted graph, where
\[
w:E\to \mathbb{R}_{\ge 0}.
\]
For a carving tree \(T\), each edge \(e\in E(T)\) induces a partition of \(V(G)\), and the \emph{congestion} of \(e\) is defined as the total weight of the graph edges crossing that partition.
The minimum possible value of the maximum congestion is called the \emph{weighted carving-width} of \(G\), denoted by
\[
\operatorname{wcvw}(G).
\]
\end{definition}

\begin{remark}
Weighted clique-width and weighted linear clique-width are best formulated within a weighted graph-expression formalism.
Since several inequivalent weighted variants appear in the literature, it is preferable to define such parameters only after fixing the allowed label operations and the way in which weights interact with them.
\end{remark}

\subsection{Directed Linear Width}

We next state an ordering-based version of directed linear width.

\begin{definition}[Proposed directed linear width]
Let \(D=(V,E)\) be a digraph, and let
\[
(e_1,e_2,\dots,e_m)
\]
be a linear ordering of \(E\).

For each \(1\le i\le m-1\), set
\[
E_1^i:=\{e_1,\dots,e_i\},
\qquad
E_2^i:=E\setminus E_1^i.
\]
Let
\[
S_V(E_1^i,E_2^i)
\]
denote the chosen directed separator associated with the cut \((E_1^i,E_2^i)\).
The \emph{width} of the ordering is
\[
\max_{1\le i\le m-1} |S_V(E_1^i,E_2^i)|.
\]
The \emph{directed linear width} of \(D\) is the minimum of this quantity over all linear orderings of \(E(D)\).
\end{definition}

\begin{remark}
To obtain a fully rigorous parameter, the separator operator \(S_V(\cdot,\cdot)\) must be specified in advance.
Different choices may lead to genuinely different directed width notions.
\end{remark}

\subsection{Directed Connectivity Systems}

The notions of directed subsets and directed submodular functions suggest a possible directed analogue of connectivity systems.

\begin{definition}[Directed subset]\cite{Qi1988DirectedSD}
Let \(S\) be a finite set.
A \emph{directed subset} of \(S\) is a function
\[
X:S\to\{-1,0,1\}.
\]
For \(e\in S\), we say that:
\begin{itemize}
    \item \(e\) is \emph{forward} in \(X\) if \(X(e)=1\),
    \item \(e\) is \emph{backward} in \(X\) if \(X(e)=-1\),
    \item \(e\) is absent from \(X\) if \(X(e)=0\).
\end{itemize}
For directed subsets \(X\) and \(Y\), one defines \(X\cap Y\) and \(X\cup Y\) coordinatewise in the usual way.
\end{definition}

\begin{definition}[Directed submodular function]\cite{Qi1988DirectedSD}
Let \(\mathcal{S}\) be a family of directed subsets of a finite set \(S\).
A function
\[
f:\mathcal{S}\to\mathbb{R}
\]
is called \emph{directed submodular} if, for all \(X,Y\in\mathcal{S}\),
\[
f(X)+f(Y)\ge f(X\cap Y)+f(X\cup Y).
\]
\end{definition}

\begin{question}
Can one define a useful notion of \emph{directed connectivity system} based on directed subsets and directed submodular functions, and can such a framework capture directed tree-width, directed branch-width, or related obstruction objects?
\end{question}

\subsection{DAG Parameters}

For directed acyclic graphs, ordering-based width parameters can be defined naturally in terms of topological orderings.

\begin{definition}[DAG cut-width]
Let \(D=(V,E)\) be a directed acyclic graph.
Its \emph{DAG cut-width} is
\[
\operatorname{dcutw}(D):=
\min_{\varphi}
\max_{1\le i\le |V|}
\left|
\{(u,v)\in E \mid \varphi(u)\le i<\varphi(v)\}
\right|,
\]
where the minimum is taken over all topological orderings \(\varphi\) of \(D\).
\end{definition}

\begin{definition}[DAG bandwidth]
Let \(D=(V,E)\) be a directed acyclic graph.
Its \emph{DAG bandwidth} is
\[
\operatorname{dbw}(D):=
\min_{\varphi}
\max_{(u,v)\in E} |\varphi(v)-\varphi(u)|,
\]
where the minimum is taken over all topological orderings \(\varphi\) of \(D\).
\end{definition}

\begin{question}
Which obstruction objects naturally correspond to DAG-path-width, DAG cut-width, and DAG bandwidth?
\end{question}


\section{Further Properties of Ultrafilters on a Connectivity System}

In this section we discuss several additional aspects of ultrafilters on a connectivity system.
Some are immediate consequences of the axioms, whereas others should be regarded as open-ended research directions.

\subsection{Simple Games on a Connectivity System}

Simple games are standard mathematical models of voting systems and collective decision-making.
They are closely related to filters and ultrafilters through the interpretation of winning coalitions.

\begin{definition}[Simple game on a connectivity system]
Let \((X,f)\) be a connectivity system, and let \(k\ge 0\).
A family
\[
W\subseteq \{A\subseteq X\mid f(A)\le k\}
\]
is called a \emph{simple game of order \(k+1\)} on \((X,f)\) if:
\begin{enumerate}
    \item[(SG1)] \(\emptyset\notin W\),
    \item[(SG2)] \(X\in W\),
    \item[(SG3)] whenever \(A\in W\), \(A\subseteq B\subseteq X\), and \(f(B)\le k\), one has \(B\in W\).
\end{enumerate}
The members of \(W\) are called \emph{winning coalitions}.
\end{definition}

\begin{definition}
A simple game \(W\) of order \(k+1\) is called:
\begin{enumerate}
    \item \emph{proper} if
    \[
    A\in W \text{ and } f(A)\le k
    \quad\Longrightarrow\quad
    X\setminus A\notin W;
    \]
    \item \emph{strong} if
    \[
    A\notin W \text{ and } f(A)\le k
    \quad\Longrightarrow\quad
    X\setminus A\in W.
    \]
\end{enumerate}
\end{definition}

\begin{proposition}
Let \(U\) be an ultrafilter of order \(k+1\) on a connectivity system \((X,f)\).
Then \(U\) is a proper strong simple game of order \(k+1\) on \((X,f)\).
\end{proposition}

\begin{proof}
By axiom \((Q3)\), one has \(\emptyset\notin U\).
Since \(U\neq\emptyset\), axiom \((Q2)\) implies \(X\in U\), because
\[
f(X)=f(\emptyset)=0\le k.
\]
Hence \((SG1)\) and \((SG2)\) hold.
Moreover, \((SG3)\) is exactly axiom \((Q2)\).

To prove properness, suppose that \(A\in U\) and \(X\setminus A\in U\).
Then
\[
A\cap (X\setminus A)=\emptyset
\quad\text{and}\quad
f(\emptyset)=0\le k.
\]
Therefore axiom \((Q1)\) yields \(\emptyset\in U\), contradicting \((Q3)\).
Thus \(U\) is proper.

To prove strength, let \(A\subseteq X\) satisfy \(f(A)\le k\), and assume that \(A\notin U\).
Then axiom \((Q4)\) implies \(X\setminus A\in U\).
Hence \(U\) is strong.
\end{proof}

\subsection{Choice Functions Compatible with a Connectivity Constraint}

The ordinary Axiom of Choice concerns arbitrary families of nonempty sets.
In the setting of connectivity systems, one may ask whether a choice can be made while preserving some \(f\)-bounded structure.

\begin{definition}[Admissible choice function]
Let \((X,f)\) be a connectivity system, and let
\[
\{S_\alpha\}_{\alpha\in A}
\]
be a family of nonempty subsets of \(X\) such that
\[
f(S_\alpha)\le k
\qquad
\text{for all } \alpha\in A.
\]
A \emph{choice function} for this family is a map
\[
c:A\to X
\]
such that
\[
c(\alpha)\in S_\alpha
\qquad
\text{for all } \alpha\in A.
\]
\end{definition}

\begin{question}
Under what additional assumptions on \((X,f)\), or on the family \(\{S_\alpha\}_{\alpha\in A}\), can one choose \(c\) so that the image
\[
c[A]=\{c(\alpha)\mid \alpha\in A\}
\]
also satisfies a meaningful connectivity bound, for example
\[
f(c[A])\le k
\qquad\text{or}\qquad
f(c[A])\le g(k)?
\]
\end{question}

\begin{remark}
This question may be especially interesting in the infinite setting, where selective and Ramsey-type ultrafilters could interact with admissible choice principles.
\end{remark}

\subsection{Ultraproducts and Ultrapowers: A Programmatic Viewpoint}

Ultraproducts and ultrapowers are fundamental constructions in model theory.
A satisfactory extension of these notions to ultrafilters on connectivity systems appears to require additional compatibility assumptions, because the ultrafilter axioms on \((I,f)\) only decide those subsets of the index set \(I\) that are \(f\)-admissible.

\begin{question}
Let \((I,f)\) be a connectivity system, and let \(U\) be an ultrafilter of order \(k+1\) on \((I,f)\).
Given a family of structures \(\{M_i\}_{i\in I}\) in a common language, can one define a useful analogue of the classical ultraproduct
\[
\prod_{i\in I} M_i/U
\]
by restricting attention to \(U\)-large index sets that are \(f\)-admissible?
\end{question}

\begin{remark}
A plausible approach would be to identify a class of formulas, or of index sets, whose truth sets are always \(U\)-decidable.
Developing such a framework would amount to a connectivity-system analogue of \L o\'s theorem.
\end{remark}

\subsection{A Brute-Force Algorithm for Constructing an Ultrafilter}

For finite connectivity systems, one may search exhaustively for an ultrafilter.

\begin{algorithm}[H]
\caption{Exhaustive search for an ultrafilter on a finite connectivity system}
\begin{itemize}
    \item \textbf{Input:} A finite connectivity system \((X,f)\) and an integer \(k\ge 0\).
    \item \textbf{Output:} An ultrafilter of order \(k+1\) on \((X,f)\), if one exists.
\end{itemize}

\begin{enumerate}
    \item Compute the \(k\)-admissible family
    \[
    \mathcal{P}_k(X,f):=\{A\subseteq X\mid f(A)\le k\}.
    \]
    \item Enumerate all subfamilies \(F\subseteq \mathcal{P}_k(X,f)\).
    \item For each such \(F\), test whether \(F\) satisfies axioms \((Q1)\)--\((Q4)\) and whether \(\emptyset\notin F\).
    \item Return the first family that passes all tests.
    \item If no such family exists, report that no ultrafilter of order \(k+1\) exists.
\end{enumerate}
\end{algorithm}

\begin{proposition}
The above algorithm is correct.
\end{proposition}

\begin{proof}
By definition, every ultrafilter of order \(k+1\) on \((X,f)\) is a subfamily of \(\mathcal{P}_k(X,f)\).
Hence exhaustive enumeration of all subfamilies of \(\mathcal{P}_k(X,f)\) includes every possible candidate.
The algorithm returns exactly those families satisfying the ultrafilter axioms.
Therefore, if an ultrafilter exists, the algorithm finds one, and if none exists, it correctly reports failure.
\end{proof}

\begin{proposition}
Let
\[
m:=\bigl|\mathcal{P}_k(X,f)\bigr|.
\]
Then the brute-force algorithm runs in time
\[
O(2^m m^2).
\]
Moreover, if one measures space only at the level of enumerating candidate subfamilies of \(\mathcal{P}_k(X,f)\), the algorithm uses
\[
O(m)
\]
storage locations.
A more precise bit-complexity bound depends on the chosen representation of subsets of \(X\).
Since
\[
m\le 2^{|X|},
\]
the running time is doubly exponential in \(|X|\) in the worst case.
\end{proposition}

\begin{proof}
There are exactly \(2^m\) candidate families
\[
F\subseteq \mathcal{P}_k(X,f).
\]
For each candidate \(F\), axioms \((Q1)\) and \((Q2)\) can be verified by checking pairs of members of \(\mathcal{P}_k(X,f)\), which requires \(O(m^2)\) time.
The remaining conditions, namely \((Q3)\) and \((Q4)\), require at most \(O(m)\) additional time.
Therefore, the total running time is
\[
O(2^m m^2).
\]

For the space bound, one only needs to keep in memory the admissible family \(\mathcal{P}_k(X,f)\) together with one current candidate family \(F\subseteq \mathcal{P}_k(X,f)\).
Thus, at the level of family enumeration, the required storage is \(O(m)\).
The corresponding bit-complexity depends on how the subsets of \(X\) are encoded.
\end{proof}

\begin{remark}
This algorithm is not intended to be efficient.
Its purpose is simply to provide a mathematically correct finite decision procedure.
It would be interesting to develop substantially faster algorithms for special classes of graphs or for special types of connectivity functions.
\end{remark}

\subsection{Ultrafilter Width and Expansion-Type Hypotheses}
We next isolate a natural width-type invariant associated with ultrafilters.

\begin{definition}[Ultrafilter width]
Let \((X,f)\) be a finite connectivity system.
The \emph{ultrafilter width} of \((X,f)\), denoted by
\[
\operatorname{ufw}(X,f),
\]
is defined by
\[
\operatorname{ufw}(X,f):=
\max\{k\ge 0 \mid \text{there exists a non-principal ultrafilter of order } k+1 \text{ on } (X,f)\},
\]
provided that at least one such non-principal ultrafilter exists.
\end{definition}

\begin{remark}
Whenever the branch-width duality theorem applies, \(\operatorname{ufw}(X,f)\) is closely related to branch-width.
Indeed, in such settings the existence of a non-principal ultrafilter of order \(k+1\) is equivalent to the failure of branch-width at most \(k\).
\end{remark}

To formulate an expansion-type hypothesis for connectivity systems, it is natural to imitate conductance-type quantities.

\begin{definition}[Connectivity conductance]
Let \((X,f)\) be a finite connectivity system.
For \(S\subseteq X\), define the \(f\)-volume of \(S\) by
\[
\operatorname{vol}_f(S):=\sum_{x\in S} f(\{x\}).
\]
Whenever the denominator is positive, define the \emph{connectivity conductance} of \(S\) by
\[
\Phi_f(S):=
\frac{f(S)}{\min\{\operatorname{vol}_f(S),\operatorname{vol}_f(X\setminus S)\}}.
\]
\end{definition}

\begin{definition}[Small set connectivity expansion problem]
Let \((X,f)\) be a finite connectivity system and let \(0<\delta\le \tfrac12\).
The \emph{small set connectivity expansion problem} asks for a subset \(S\subseteq X\) satisfying
\[
0<\operatorname{vol}_f(S)\le \delta\,\operatorname{vol}_f(X)
\]
that minimizes \(\Phi_f(S)\).
\end{definition}

\begin{conjecture}[Connectivity-system SSE-type hardness]
Under a suitable Small Set Connectivity Expansion Hypothesis for connectivity systems, it is NP-hard to approximate \(\operatorname{ufw}(X,f)\) within a constant factor.
\end{conjecture}

\begin{question}
Can one formulate a natural class of connectivity systems for which the approximation hardness of ultrafilter width follows from the classical Small Set Expansion Hypothesis?
\end{question}

\subsection{Further Variants of Ultrafilters}

Many variants of ultrafilters have been studied in set theory, topology, logic, fuzzy theory, and related areas.
It is therefore natural to ask whether analogous notions can be defined for graphs, directed graphs, infinite graphs, connectivity systems, and infinite connectivity systems.

\begin{question}
Can one define and systematically study connectivity-system analogues of partition ultrafilters, uniform ultrafilters, normal ultrafilters, regular ultrafilters, good ultrafilters, fuzzy ultrafilters, selective ultrafilters, Ramsey ultrafilters, and related variants?
\end{question}

\begin{remark}
A particularly appealing goal would be to determine which of these variants correspond to known width parameters, game-theoretic obstructions, or decomposition phenomena.
\end{remark}

\section{Other Research Themes for Graph Width Parameters}

We conclude by listing several broader directions in the theory of graph width parameters that remain closely related to the present work.

\begin{itemize}
    \item \textbf{Hierarchy questions and restricted graph classes.}
    Study the behavior of width parameters on special graph classes such as line graphs, planar graphs, AT-free graphs, chordal graphs, and sequence graphs
    \cite{bottcher2010bandwidth,cabello2012algorithms,69}.

    \item \textbf{Algorithms under width constraints.}
    Develop exact, approximation, and parameterized algorithms under bounded-width assumptions
    \cite{wanke1994bounded,kawarabayashi2008simpler,lagergren1996efficient,bodlaender1998parallel}.

    \item \textbf{Generalized width parameters.}
    Investigate structural properties of connected width, directed width, layered width, linear width, and infinite-width-type notions
    \cite{diestel2018connected,hamann2016bounding}.

    \item \textbf{Game-theoretic interpretations.}
    Clarify the relationship between width parameters and pursuit--evasion or search games
    \cite{kosowski2015k,Soares2013PursuitEvasionDA}.

    \item \textbf{Forbidden structures and obstruction theory.}
    Study how forbidden minors, forbidden substructures, tangles, profiles, and ultrafilter-like objects characterize bounded width
    \cite{Gurski2019ForbiddenDM,takahashi1995minimal}.
\end{itemize}

\newpage

\appendix

\chapter{Various Width Parameters}
A large number of graph width parameters have been introduced and studied in the literature.  
Many of them play important roles not only in graph theory itself but also in algorithms, optimization, logic, networks, and computational biology.  
Because of this broad applicability, the study of graph width parameters continues to be an active and growing area of research.

A central topic in this area is the comparison of different width parameters, including inequalities, upper bounds, lower bounds, and equivalences; see also Appendix~C.  
In future work, we also intend to investigate possible connections between width parameters and ultrafilters.

Unless stated otherwise, all parameter comparisons below are understood for arbitrary graphs.  
We say that a parameter \(p\) \emph{upper-bounds} a parameter \(q\) if there exists a non-decreasing function \(f\) such that
\[
q(G)\leq f(p(G))
\qquad
\text{for every graph } G.
\]
If no such function exists, then \(q\) is said to be \emph{unbounded in terms of} \(p\).
As usual, if \(p\) upper-bounds \(q\) and \(q\) upper-bounds \(r\), then \(p\) upper-bounds \(r\).

\section{Tree-width}
Tree-width is one of the most fundamental graph width parameters.  
It measures how close a graph is to a tree and is defined via tree-decompositions
\cite{47,213,214,215}.  
Many problems that are computationally difficult on general graphs become tractable on graph classes of bounded tree-width.

\begin{definition}[Tree-decomposition]
A \emph{tree-decomposition} of a graph \(G=(V,E)\) is a pair
\[
\bigl(\{X_i\mid i\in I\},\, T=(I,F)\bigr),
\]
where each \(X_i\subseteq V\) is called a \emph{bag}, and \(T\) is a tree whose nodes are indexed by the bags, such that:
\begin{enumerate}
    \item
    \[
    \bigcup_{i\in I} X_i = V;
    \]
    that is, every vertex of \(G\) appears in at least one bag.
    \item For every edge \(uv\in E\), there exists some \(i\in I\) such that
    \[
    u,v\in X_i.
    \]
    \item For every vertex \(v\in V\), the set
    \[
    \{\, i\in I \mid v\in X_i \,\}
    \]
    induces a connected subtree of \(T\).
\end{enumerate}
\end{definition}

\begin{definition}[Tree-width]
The \emph{width} of a tree-decomposition
\[
\bigl(\{X_i\mid i\in I\},\, T=(I,F)\bigr)
\]
is
\[
\max_{i\in I}|X_i|-1.
\]
The \emph{tree-width} of a graph \(G\), denoted by \(\operatorname{tw}(G)\), is the minimum width over all tree-decompositions of \(G\).
A graph has tree-width \(1\) if and only if it is a forest.
\end{definition}

\begin{example}
Let \(K_n\) be the complete graph on \(n\) vertices.
A tree-decomposition of \(K_n\) is obtained by taking a tree with a single node and a single bag
\[
X=V(K_n).
\]
This clearly satisfies all three axioms of a tree-decomposition.
Its width is
\[
|X|-1=n-1.
\]
Since every pair of vertices in \(K_n\) is adjacent, no tree-decomposition of \(K_n\) can have width smaller than \(n-1\). Hence
\[
\operatorname{tw}(K_n)=n-1.
\]
\end{example}

Typical obstruction-type objects for tree-width include tangles, ultrafilters, and brambles
\cite{47,fujitaultrafilter,239,240,241,242}.  
For broader introductions to tree-width, see for example
\cite{bienstock1989graph,kloks1994treewidth,bodlaender1998partial,bronner2006introduction}.

\subsection{Selected Relations Involving Tree-width}

We next list several representative relationships between tree-width and other graph parameters.

\begin{theorem}
The following statements are known.
\begin{enumerate}
    \item If \(G\) has tree-width \(k\), then \(G\) is a subgraph of a chordal graph with maximum clique size \(k+1\) \cite{bodlaender1998partial}.

    \item For every graph \(G\),
    \[
    \operatorname{tw}(G)\leq \operatorname{pw}(G).
    \]

    \item For every graph \(G\) with \(\operatorname{bw}(G)\geq 2\),
    \[
    \operatorname{bw}(G)\leq \operatorname{tw}(G)+1
    \leq \frac{3}{2}\operatorname{bw}(G).
    \]
    \cite{213}

    \item If \(G\) has average degree \(d\), then
    \[
    \operatorname{tw}(G)\geq d/2.
    \]
    \cite{chandran2005girth}

    \item If \(G\) is \(r\)-degenerate, then
    \[
    r\leq \operatorname{tw}(G),
    \]
    and in particular
    \[
    \delta(G)\leq \operatorname{tw}(G).
    \]
    \cite{chandran2005girth,402}

    \item Every graph of tree-width \(k\) has a balanced separator of size at most \(k+1\)
    \cite{364,blasius2016hyperbolic}.

    \item If \(H\) is a minor of \(G\), then
    \[
    \operatorname{tw}(H)\leq \operatorname{tw}(G).
    \]
    \cite{dallard2024treewidth}

    \item For every graph \(G\),
    \[
    \operatorname{box}(G)\leq \operatorname{tw}(G)+2.
    \]
    \cite{347}

    \item For every graph \(G\),
    \[
    \operatorname{tw}(G)\leq la(G)\leq \operatorname{tw}(G)+1,
    \]
    where \(la(G)\) denotes the largeur d'arborescence \cite{van2002largeur,de1998multiplicities}.

    \item If \(G\) has bounded tree-width, then it also has bounded branch-width, path-width, Boolean-width, MIM-width, clique-width, sparse twin-width, and twin-width
    \cite{31,32,312,313}.

    \item If \(G\) has bounded tree-width, then it also has bounded clique number and bounded chromatic number
    \cite{311,417}.

    \item The bramble number of a graph is equal to its tree-width plus one
    \cite{216,239,240,241,242}.

    \item Every graph of tree-width \(k\) is \(k\)-degenerate
    \cite{400}.

    \item Every graph of tree-width \(k\) admits a \(2\)-balanced binary tree-decomposition of width at most \(3k+2\)
    \cite{401}.
\end{enumerate}
\end{theorem}

\begin{proof}
See the cited references.
\end{proof}

\subsection{Related Concepts for Tree-width}

Many variants and refinements of tree-width have been studied.  
Representative examples include the following.

\begin{itemize}
    \item \textbf{Local tree-width} \cite{45}: the maximum tree-width of an \(r\)-neighbourhood, viewed as a function of \(r\).
    \item \textbf{Linear local tree-width} \cite{dujmovic2013layered,45}: a linear variant of local tree-width.
    \item \textbf{Simple tree-width} \cite{knauer2012simple}: a refinement satisfying
    \[
    \operatorname{tw}(G)\leq stw(G)\leq \operatorname{tw}(G)+1.
    \]
    \item \textbf{Bag-connected tree-width} \cite{138}: tree-width restricted to decompositions whose bags induce connected subgraphs.
    \item \textbf{Layered tree-width} \cite{157,159}: the minimum integer \(k\) such that \(G\) has a tree-decomposition together with a layering in which every bag contains at most \(k\) vertices from each layer.
    \item \textbf{Connected tree-width} \cite{181,183}: the minimum width of a tree-decomposition whose bags induce connected subgraphs.
    \item \textbf{Special tree-width} \cite{185}: a parameter lying between tree-width and path-width.
    \item \textbf{Spaghetti tree-width} \cite{226,315}: a variant closely related to special tree-width.
    \item \textbf{Co-tree-width} \cite{193}: the tree-width of the complement graph.
    \item \textbf{Dual tree-width} \cite{425}: the tree-width of the dual graph.
    \item \textbf{Weighted tree-width} \cite{kammer2016approximate}: the weighted analogue of tree-width.
    \item \textbf{Weak tree-width} \cite{amini2016spectral}: a relaxed form of tree-decomposition related to strong brambles, satisfying
    \[
    2\,\operatorname{wtw}(G)\geq \operatorname{tw}(G)+1.
    \]
\end{itemize}

Other concepts related to tree-width include the maximum order of a grid minor, grid-like minors, Hadwiger number, fractional Hadwiger number, domino tree-width, VF-tree-width, branched tree-width, radial tree-width, twin-treewidth, threshold treewidth, and several hypergraph variants
\cite{harvey2014treewidth,read1968use,bonnet2021metric}.

\subsection{Applications of Tree-width}

Because tree structures are flexible and algorithmically useful, tree-width has found applications in many areas.  
Examples include:

\begin{itemize}
    \item \textbf{Protein structure analysis}
    \cite{xu2005tree,peng2015itreepack}
    \item \textbf{RNA and DNA applications}
    \cite{gurski2008polynomial,yao2024infrared}
    \item \textbf{Metabolic networks}
    \cite{cheng2010efficient,panni2015searching,betzler2006steiner}
    \item \textbf{Social networks}
    \cite{adcock2013tree,gupta2022treewidth,Maniu2019AnES}
    \item \textbf{Bayesian networks}
    \cite{elidan2008learning,torrijos2024structural}
    \item \textbf{Neural networks and graph neural networks}
    \cite{zheng2022graph,barcelo2021graph}
    \item \textbf{Databases}
    \cite{fichte2022exploiting,grohe1999definability}
    \item \textbf{Internet and communication networks}
    \cite{Montgolfier2011TreewidthAH}
    \item \textbf{Markov networks}
    \cite{Srebro2001MaximumLB,Liang2004MethodsAE,Karger2001LearningMN}
    \item \textbf{Tensor networks}
    \cite{Dumitrescu2018BenchmarkingTA}
\end{itemize}

For broader surveys and lecture-style introductions, see
\cite{Bodlaender1993ATG,bienstock1989graph}.
\section{Path-width}

Path-width is a fundamental graph width parameter that measures how close a graph is to a path.  
It is defined via path-decompositions, which are tree-decompositions whose underlying tree is a path
\cite{275,277,278,kashyap2008matroid,hatanaka2015list}.  
An obstruction to small path-width is a blockage \cite{46}.  
Path-width also appears in several applications, including VLSI design \cite{318}.  
In some parts of the literature, path-width is also called \emph{interval-width} \cite{409,410}.

\begin{definition}[Path-decomposition]
A \emph{path-decomposition} of a graph \(G=(V,E)\) is a sequence of subsets
\[
(X_1,X_2,\dots,X_\ell)
\]
of \(V\) such that:
\begin{enumerate}
    \item
    \[
    V=\bigcup_{i=1}^{\ell} X_i;
    \]
    \item for every edge \(uv\in E\), there exists an index \(i\in\{1,\dots,\ell\}\) such that
    \[
    u,v\in X_i;
    \]
    \item for every vertex \(v\in V\), the set of indices
    \[
    \{\, i\in\{1,\dots,\ell\}\mid v\in X_i \,\}
    \]
    forms an interval of \(\{1,\dots,\ell\}\).
\end{enumerate}
\end{definition}

\begin{definition}[Path-width]
The \emph{width} of a path-decomposition \((X_1,\dots,X_\ell)\) is
\[
\max_{1\leq i\leq \ell}|X_i|-1.
\]
The \emph{path-width} of a graph \(G\), denoted by \(\pw(G)\), is the minimum width over all path-decompositions of \(G\).
\end{definition}

For matroids, there is an ordering-based analogue of path-width.

\begin{definition}[Path-width of a matroid]
\cite{kashyap2008matroid}
Let \(M\) be a matroid with ground set \(E(M)=\{e_1,\dots,e_n\}\), and let \(f_M\) be its connectivity function.
For an ordering \((e_1,\dots,e_n)\) of \(E(M)\), define
\[
w_M(e_1,\dots,e_n):=
\max_{1\leq i\leq n} f_M(\{e_1,\dots,e_i\}).
\]
The \emph{path-width} of \(M\) is then
\[
\pw(M):=\min w_M(e_1,\dots,e_n),
\]
where the minimum is taken over all orderings of \(E(M)\).
An ordering attaining this minimum is called an \emph{optimal ordering}.
\end{definition}

\subsection{Selected Relations for Path-width}

We next summarize several standard relations involving path-width.

\begin{theorem}
The following statements are known.
\begin{enumerate}
    \item If a graph \(G\) has bounded path-width, then it also has bounded tree-width.

    \item The path-width of a graph is equal to one less than the minimum clique number among interval supergraphs of \(G\) \cite{213}.

    \item If a graph \(G\) has bounded tree-depth, then it also has bounded path-width \cite{311}.

    \item If a graph \(G\) has bounded path-width, then it also has bounded linear rank-width \cite{323}.

    \item If a graph \(G\) has bounded path-width, then it also has bounded linear clique-width and bounded linear NLC-width \cite{39}.

    \item The vertex separation number of a graph is equal to its path-width \cite{320}.

    \item Path-width is at most cut-width \cite{321}.

    \item If a graph \(G\) has bounded path-width, then it also has bounded clique-width \cite{325}.

    \item For every graph \(G\),
    \[
    \thinness(G)\leq \pw(G)+1.
    \]
    \cite{mannino2007stable}

    \item If a graph \(G\) has bounded bandwidth, then it also has bounded path-width \cite{313}.

    \item If \(G\) has a \(T_h\)-witness, then
    \[
    \pw(G)\geq h.
    \]
        \cite{groenland2023approximating}

    \item For every graph \(G\),
    \[
    \operatorname{ds}(G)\leq \pw(G)+1 \leq (\Delta(G)+1)\operatorname{ds}(G),
    \]
    where \(\operatorname{ds}(G)\) denotes the domination search number of \(G\)    \cite{kreutzer2016complexity}.

    \item The gate matrix layout number is equal to \(\pw(G)-1\).
\end{enumerate}
\end{theorem}

\begin{proof}
See the cited references.
\end{proof}

\subsection{Related Concepts for Path-width}
Several parameters may be viewed as variants, refinements, or linear analogues of path-width.

\begin{itemize}
    \item \textbf{DAG-path-width} \cite{141}: a linear-layout parameter related to DAG-width.

    \item \textbf{Layered path-width} \cite{157,159}: the path analogue of layered tree-width.  
    It is known that every graph with path-width \(k\) has layered path-width at most \((k+1)/2\), and every graph with layered path-width \(\lambda\) has track-number at most \(3\lambda\).

    \item \textbf{Circuit path-width} \cite{razgon2013cliquewidth,bova2017circuit}: the path analogue of circuit tree-width.  
    Related parameters include OBDD width and SDD width \cite{jha2012tractabilit,bova2017circuit}.

    \item \textbf{Simple path-width} \cite{biedl2023complexity}: the least \(w\) such that \(G\) has a \(w\)-simple path-decomposition of width at most \(w\).

    \item \textbf{Row path-width} \cite{157,158}: the linear analogue of row tree-width.

    \item \textbf{Linear local path-width} \cite{Dujmovi2018MinorclosedGC}: a path analogue of linear local tree-width.

    \item \textbf{D-path-width} and \textbf{clique-preserving d-path-width} \cite{201}: parameters arising in the study of branching programs and CNF representations.

    \item \textbf{Semantic path-width} \cite{206}: the path analogue of semantic tree-width.

    \item \textbf{Proper path-width} \cite{276,takahashi1994minimal,takahashi1995proper}: a path-width variant related to mixed search.  
    It satisfies
    \[
    \pw(G)\leq \ppw(G)\leq \pw(G)+1.
    \]
    Moreover, for a simple graph \(G\), \(\ppw(G)\leq k\) if and only if \(G\) is a partial \(k\)-path.

    \item \textbf{Connected path-width} \cite{316}: the linear analogue of connected tree-width.  
    It is known that connected path-width is at most \(2\pw(G)+1\) \cite{dereniowski2012pathwidth}.

    \item \textbf{Persistence path-width} \cite{317}: a path-decomposition in which each vertex appears in at most \(l\) bags.

    \item \textbf{Linked path decompositions} \cite{362}: a weaker form of lean path decomposition.

    \item \textbf{Dual path-width} \cite{425}: the path-width of the dual graph.

    \item \textbf{Radial path-width} \cite{albrechtsen2023structural}: a linear version of radial tree-width.

    \item \textbf{Trellis-width} \cite{Kashyap2007MatroidPA,Fomin2016OnIP}: a width notion for linear codes that is essentially equivalent to a path-width-type parameter.
\end{itemize}

\subsection{Applications of Path-width}

Applications of path-width include the following areas.

\begin{itemize}
    \item \textbf{VLSI design} \cite{ohtsuki1979one,mohring1995vlsi,Mhring1990GraphPR}
    \item \textbf{Compiler design} \cite{bodlaender1998linear,Conrado2023TheBP}
    \item \textbf{Blockchain and DAG-based systems} \cite{kawahara2024graph,kasahara2023dag}
    \item \textbf{Natural language processing} \cite{KornaiTuza1992}
    \item \textbf{IoT and network applications} \cite{aloisio2024coverage,aloisio2024parameterized}
\end{itemize}

\section{Cut-width}

Cut-width is a classical linear-layout parameter that measures the maximum number of edges crossing a cut induced by a vertex ordering \cite{279,280,281}.  
It is also called the \emph{folding number} in some earlier literature \cite{Chung1985ONTC}.  
Related concepts include page-width \cite{336}.

\begin{definition}[Cut-width]
Let \(G=(V,E)\) be an undirected graph, and let
\[
(v_1,v_2,\dots,v_n)
\]
be an ordering of its vertices.
For each \(\ell\in\{1,\dots,n-1\}\), consider the cut between
\[
\{v_1,\dots,v_\ell\}
\qquad\text{and}\qquad
\{v_{\ell+1},\dots,v_n\}.
\]
The \emph{width} of the ordering is the maximum number of edges \(v_iv_j\in E\) such that
\[
i\leq \ell < j
\]
over all \(\ell\in\{1,\dots,n-1\}\).
The \emph{cut-width} of \(G\), denoted by \(\cutw(G)\), is the minimum such width over all vertex orderings.
\end{definition}

\subsection{Selected Relations for Cut-width}

\begin{theorem}
The following statements are known.
\begin{enumerate}
    \item If a graph \(G\) has bounded cut-width, then it also has bounded carving-width \cite{326}.

    \item For every graph \(G\),
    \[
    \pw(G)\leq \cutw(G)
    \]
    \cite{321}.

    \item For every graph \(G\),
    \[
    \thinness(G)\leq \cutw(G)+1
      \]
    \cite{mannino2007stable,beisegel2024simultaneous}.

    \item The bisection width of a graph is at most its cut-width \cite{diaz2002survey}.

    \item The cut-width of a graph can be used to derive lower bounds on the crossing number \cite{334}.

    \item In subcubic graphs, cut-width is equal to path-width plus one \cite{335}.

    \item If a graph \(G\) has bounded neighborhood diversity, then it also has bounded cut-width \cite{363}.
\end{enumerate}
\end{theorem}

\begin{proof}
See the cited references.
\end{proof}

\subsection{Related Concepts for Cut-width}

Important related parameters include:

\begin{itemize}
    \item \textbf{Weighted cut-width} \cite{179,180}: a weighted version of cut-width.
    \item \textbf{Circuit cut-width} \cite{prasad1999atpg,chong1998satisfiability}: a width parameter for circuits that is relevant to automatic test pattern generation.
    \item \textbf{Cyclic cut-width} \cite{raspaud2000congestion}: the circular analogue of cut-width.
\end{itemize}

\subsection{Applications of Cut-width}

\begin{itemize}
    \item \textbf{VLSI design} \cite{Berend2008MinimalCL}
\end{itemize}

\section{MIM-width}

MIM-width (maximum induced matching width) is a decomposition parameter defined by the size of the largest induced matching across cuts of a branch decomposition \cite{31,32}.  
It has become an important parameter in structural graph theory and algorithm design.

\begin{definition}[MIM-width]
\cite{31,32}
A \emph{branch decomposition} of a graph \(G\) is a pair \((T,L)\), where \(T\) is a subcubic tree and
\[
L:V(G)\to L(T)
\]
is a bijection from the vertex set of \(G\) to the leaves of \(T\).

For an edge \(e\in E(T)\), let \(T_1^e\) and \(T_2^e\) be the two connected components of \(T-e\).  
These induce a bipartition
\[
(A_1^e,A_2^e)
\]
of \(V(G)\), where \(A_i^e\) consists of the vertices mapped by \(L\) to leaves in \(T_i^e\).

For a subset \(A\subseteq V(G)\), let
\[
\operatorname{mim}_G(A)
\]
denote the maximum size of an induced matching in the bipartite graph
\[
G[A,V(G)\setminus A].
\]
The \emph{MIM-width} of the decomposition \((T,L)\) is
\[
\operatorname{mimw}(T,L):=\max_{e\in E(T)} \operatorname{mim}_G(A_1^e).
\]
The \emph{MIM-width} of \(G\), denoted by \(\operatorname{mimw}(G)\), is the minimum value of \(\operatorname{mimw}(T,L)\) over all branch decompositions of \(G\).  
If \(|V(G)|\leq 1\), then \(\operatorname{mimw}(G):=0\).
\end{definition}

\subsection{Selected Relations for MIM-width}

\begin{theorem}
The following statements are known.
\begin{enumerate}
    \item If a graph \(G\) has bounded MIM-width, then it also has bounded sim-width \cite{311}.

    \item If a graph \(G\) has bounded independence number or bounded domination number, then it also has bounded MIM-width \cite{311}.

    \item If a graph \(G\) has bounded Boolean-width, then it also has bounded MIM-width \cite{31,32}.

    \item If a graph \(G\) has bounded clique-width, then it also has bounded MIM-width \cite{227}.
\end{enumerate}
\end{theorem}

\begin{proof}
See the cited references.
\end{proof}

\subsection{Related Concepts for MIM-width}
Important related notions include the following.

\begin{itemize}
    \item \textbf{One-sided MIM-width} \cite{33}: a one-sided variant useful, for example, in studying the maximum independent set problem.  
    Any graph with tree-independence number \(k\) has one-sided MIM-width at most \(k\) \cite{33}.

    \item \textbf{SIM-width} \cite{34,322}: a refinement of MIM-width that is often more useful in structural settings.  
    Any graph with minor-matching hypertree-width \(k\) has SIM-width at most \(k\) \cite{33}.
\end{itemize}

\section{Linear MIM-width}

Linear MIM-width is the linear-layout version of MIM-width \cite{31,32}.  
It is obtained by restricting the underlying decomposition tree to a caterpillar, or equivalently, by considering linear orderings of the vertices.

\subsection{Selected Relations for Linear MIM-width}

\begin{theorem}
The following statements are known.
\begin{enumerate}
    \item If a graph \(G\) has bounded linear MIM-width, then it also has bounded MIM-width.

    \item If a graph \(G\) has bounded linear MIM-width, then it also has bounded one-sided MIM-width and bounded SIM-width \cite{beisegel2024simultaneous}.

    \item If a graph \(G\) has bounded tree-independence number, then it also has bounded linear MIM-width \cite{beisegel2024simultaneous}.

    \item If a graph \(G\) has bounded path-independence number, then it also has bounded linear MIM-width \cite{beisegel2024simultaneous}.

    \item If a graph \(G\) has bounded simultaneous interval number, then it also has bounded linear MIM-width \cite{beisegel2024simultaneous}.

    \item If a graph \(G\) has bounded thinness, then it also has bounded linear MIM-width \cite{beisegel2024simultaneous}.
\end{enumerate}
\end{theorem}

\begin{proof}
See the cited references.
\end{proof}

\subsection{Related Concepts for Linear MIM-width}

A closely related parameter is \emph{linear SIM-width} \cite{34}, which is the linear-layout analogue of SIM-width.  
Trivially, bounded linear SIM-width implies bounded SIM-width.

\section{Boolean-width}

Boolean-width is a graph width parameter based on the number of distinct unions of neighborhoods across cuts in a decomposition \cite{35}.  
It is particularly useful in dynamic-programming algorithms over graph decompositions.

\begin{definition}[Decomposition tree]
\cite{35}
A \emph{decomposition tree} of a graph \(G\) is a pair \((T,\delta)\), where \(T\) is a tree whose internal vertices have degree \(3\), and
\[
\delta: L(T)\to V(G)
\]
is a bijection from the leaves of \(T\) to the vertices of \(G\).

Each edge \(e\in E(T)\) induces a bipartition
\[
(A_e,\, V(G)\setminus A_e)
\]
of \(V(G)\), obtained by deleting \(e\) and taking the vertices corresponding to the leaves in one of the two resulting components.
\end{definition}

\begin{definition}[Boolean-width]
\cite{35}
Let \(G\) be a graph and let \(A\subseteq V(G)\).  
Define
\[
U(A):=\{\,N(X)\cap (V(G)\setminus A)\mid X\subseteq A\,\},
\]
where \(N(X)\) denotes the union of the neighborhoods of the vertices in \(X\).

The \emph{boolean dimension} of the cut \(A\) is
\[
\operatorname{bool\mbox{-}dim}(A):=\log_2 |U(A)|.
\]

For a decomposition tree \((T,\delta)\), its \emph{boolean-width} is
\[
\operatorname{boolw}(T,\delta):=\max_{e\in E(T)} \operatorname{bool\mbox{-}dim}(A_e),
\]
where \(A_e\) is one side of the cut induced by \(e\).

The \emph{boolean-width} of \(G\), denoted by \(\operatorname{boolw}(G)\), is the minimum value of
\(\operatorname{boolw}(T,\delta)\) over all decomposition trees of \(G\).
\end{definition}

\subsection{Selected Relations for Boolean-width}

\begin{theorem}
The following statements are known.
\begin{enumerate}
    \item If a graph \(G\) has bounded tree-width, then it also has bounded Boolean-width \cite{31,32}.

    \item For every graph \(G\) with \(\operatorname{bw}(G)\neq 0\),
    \[
    \operatorname{boolw}(G)\leq \operatorname{bw}(G)
    \]
        \cite{313}.

    \item If a graph \(G\) has bounded Boolean-width, then it also has bounded MIM-width \cite{31,32}.
\end{enumerate}
\end{theorem}

\begin{proof}
See the cited references.
\end{proof}

\section{Linear Boolean-width}

Linear Boolean-width is the linear-layout version of Boolean-width \cite{35}.  
It is obtained by restricting the decomposition tree to a caterpillar, or equivalently, by considering cuts arising from a linear ordering of the vertices.

Trivially, if a graph \(G\) has bounded linear Boolean-width, then it also has bounded Boolean-width.  
Moreover, for every graph \(G\),
\[
\operatorname{lboolw}(G)\leq \pw(G)+1.
\]
\cite{348}

\section{Bandwidth}

Bandwidth is a classical linear-layout parameter \cite{259,260}.  
It measures how close adjacent vertices must remain in a linear ordering of the vertex set.  
It has important applications in areas such as VLSI design and network layout.  
Related concepts include split bandwidth \cite{322}.

\begin{definition}[Bandwidth]
\cite{259,260}
Let \(G=(V,E)\) be a graph, and let
\[
\varphi:V(G)\to \{1,2,\dots,|V(G)|\}
\]
be a bijection, called a \emph{layout} of \(G\).

The \emph{bandwidth} of the layout \(\varphi\) is
\[
\operatorname{bw}_{\varphi}(G):=
\max_{\{u,v\}\in E(G)} |\varphi(u)-\varphi(v)|.
\]

The \emph{bandwidth} of \(G\), denoted by \(\operatorname{bw}(G)\), is the minimum value of
\(\operatorname{bw}_{\varphi}(G)\) over all layouts \(\varphi\) of \(G\).
\end{definition}

\subsection{Selected Relations for Bandwidth}

\begin{theorem}
The following statements are known.
\begin{enumerate}
    \item
    \[
    \operatorname{ppw}(G)\leq \operatorname{bw}(G)\leq 2\,\operatorname{ppw}(G),
    \]
    where \(\operatorname{ppw}(G)\) denotes the path-partition-width of \(G\).

    \item The max-leaf number is a minimal upper bound for bandwidth \cite{341}.

    \item
    \[
    \operatorname{bw}(G)\leq 2\,\operatorname{pdw}(G),
    \]
    where \(\operatorname{pdw}(G)\) denotes the path-distance-width of \(G\).

    \item If a graph \(G\) has bounded bandwidth, then it also has bounded bisection width \cite{341}.

    \item If a graph \(G\) has bounded bandwidth, then it also has bounded tree-width and bounded path-width \cite{313}.

    \item If a graph \(G\) has bounded bandwidth, then it also has bounded proper thinness \cite{bonomo2021intersection}.

    \item If a graph \(G\) has bounded bandwidth, then it also has bounded maximum degree, bounded \(c\)-closure, bounded acyclic chromatic number, and bounded \(h\)-index \cite{341,421,422}.

    \item Treespan is at most bandwidth \cite{352}.

    \item For each \(\Delta\geq 2\), every graph \(G\) on \(n\) vertices with \(\Delta(G)\leq \Delta\) satisfies
    \[
    \operatorname{bw}(G)\leq
    \frac{6n}{\log_{\Delta}\!\left(\dfrac{n}{\operatorname{separator}(G)}\right)}
    \]
        \cite{bottcher2010bandwidth}.

    \item The slope number \(\operatorname{sn}(G)\) satisfies
    \[
    \operatorname{sn}(G)\leq \frac{1}{2}\operatorname{bw}(G)\bigl(\operatorname{bw}(G)+1\bigr)+1
    \]
        \cite{dujmovic2007graph}.
\end{enumerate}
\end{theorem}

\begin{proof}
See the cited references.
\end{proof}

\subsection{Related Concepts for Bandwidth}
Related parameters include the following.

\begin{itemize}
    \item \textbf{Weighted bandwidth}: the weighted version of bandwidth.
    \item \textbf{Edge bandwidth} \cite{pikhurko2006edge,lin2015new,akhtar2007edge}: the edge-ordering analogue of bandwidth.
    \item \textbf{Cyclic bandwidth} \cite{raspaud2000congestion}: the circular-layout version of bandwidth.
\end{itemize}

\section{Boundary-width}
Vertex-boundary-width and edge-boundary-width arise from discrete isoperimetric problems \cite{36}.  
Roughly speaking, they measure how small the boundary of an initial segment can be made under a linear ordering of the vertices or edges.

\section{Carving-width}
Carving-width is a width parameter based on edge cuts of vertex sets \cite{256,257,258}.  
It is closely related to congestion and network routing, and it may be viewed as an edge-cut counterpart of branch-width.

\begin{definition}[Carving decomposition and carving-width]
\cite{256,257,258}
A \emph{carving decomposition} of a graph \(G\) is a pair \((T,\lambda)\), where \(T\) is a tree whose internal vertices have degree \(3\), and
\[
\lambda: L(T)\to V(G)
\]
is a bijection from the leaves of \(T\) to the vertices of \(G\).

Each edge \(e\in E(T)\) induces a partition
\[
(A_e,\, V(G)\setminus A_e)
\]
of \(V(G)\).  
The \emph{width} of \(e\) is the number of edges of \(G\) with one endpoint in \(A_e\) and the other in \(V(G)\setminus A_e\).

The \emph{width} of the carving decomposition \((T,\lambda)\) is the maximum width over all edges of \(T\).  
The \emph{carving-width} of \(G\), denoted by \(\operatorname{carw}(G)\), is the minimum such width over all carving decompositions of \(G\).  
If \(|V(G)|=1\), then \(\operatorname{carw}(G):=0\).
\end{definition}

\begin{theorem}
The following statements are known.
\begin{enumerate}
    \item If a graph \(G\) has bounded cut-width, then it also has bounded carving-width \cite{326}.

    \item Let \(G\) be a graph with maximum degree \(d\). Then
    \[
    \frac{2}{3}\bigl(\tw(G)+1\bigr)
    \leq
    \operatorname{carw}(G)
    \leq
    d\bigl(\tw(G)+1\bigr)
    \]
      \cite{bienstock1990embedding}.
\end{enumerate}
\end{theorem}

\begin{proof}
See the cited references.
\end{proof}

\subsection{Related Concepts for Carving-width}
A related parameter is the following.

\begin{itemize}
    \item \textbf{Dual carving-width} \cite{424}: the carving-width of the dual graph.
\end{itemize}

\subsection{Applications of Carving-width}
Carving-width is closely related to congestion in network theory.  
In particular, it measures how large an edge cut can become when a network is recursively decomposed, which makes it relevant in routing, communication, and congestion minimization problems \cite{256}.  
It is also known to serve as a minimal upper bound for maximum degree and tree-width in an appropriate sense \cite{311,337}.

\section{Branch-width}

Branch-width is a fundamental graph width parameter defined via branch decompositions \cite{9,10,47,282}.  
Roughly speaking, it measures how well a graph can be recursively separated by cuts of small order.  
Unlike tree-width, which is based on vertex bags in a tree-decomposition, branch-width is defined in terms of edge partitions and the sizes of the corresponding middle sets.  
This parameter extends naturally to matroids and, more generally, to connectivity systems.

Typical obstruction-type objects for branch-width include tangles \cite{9,Hicks2005GraphsBA}, profiles \cite{376,377,378,379,380}, \(k\)-blocks \cite{381,382,383}, ultrafilters \cite{16}, maximal ideals \cite{49}, loose tangles \cite{10}, loose tangle kits \cite{10}, quasi-ultrafilters \cite{14}, weak ultrafilters \cite{15}, tangle kits \cite{Oum2006CertifyingLB}, brambles \cite{Lyaudet2009PartitionsVS,amini2009submodular}, and \((k,m)\)-obstacles \cite{50}.  
Related notions include sphere-cut width \cite{338} and rooted branch-width \cite{153}.

\subsection{Selected Relations for Branch-width}

\begin{theorem}
The following statements are known.
\begin{enumerate}
    \item For every matroid \(M\),
    \[
    \operatorname{bw}(M)\leq \operatorname{bd}(M),
    \]
    where \(\operatorname{bd}(M)\) denotes the branch-depth of \(M\) \cite{329}.

    \item For every matroid \(M\),
    \[
    \operatorname{bw}(M)\leq \operatorname{cdd}(M)\leq \operatorname{cd}(M)
    \quad\text{and}\quad
    \operatorname{bw}(M)\leq \operatorname{cdd}(M)\leq \operatorname{dd}(M),
    \]
    where \(\operatorname{cdd}(M)\), \(\operatorname{cd}(M)\), and \(\operatorname{dd}(M)\) denote contraction-deletion-depth, contraction-depth, and deletion-depth, respectively \cite{311}.

    \item For every graph \(G\) with \(\operatorname{bw}(G)\geq 2\),
    \[
    \operatorname{bw}(G)\leq \operatorname{tw}(G)+1
    \leq \frac{3}{2}\operatorname{bw}(G).
    \]
    \cite{213}

    \item For every graph \(G\) with \(\operatorname{bw}(G)\neq 0\),
    \[
    \operatorname{boolw}(G)\leq \operatorname{bw}(G).
    \]
    \cite{313}

    \item The branch-width of a graph is a lower bound on the zero-forcing number \cite{fast2018effects}.  
    In particular,
    \[
    \operatorname{bw}(G)\leq \operatorname{tw}(G)+1 \leq Z(G)+1,
    \]
    where \(Z(G)\) denotes the zero-forcing number of \(G\). \cite{fast2018effects}

    \item If a graph \(G\) has bounded linear branch-width, then it also has bounded branch-width.
\end{enumerate}
\end{theorem}

\begin{proof}
See the cited references.
\end{proof}

\section{Linear Branch-width}

Linear branch-width, also called \emph{linear-width} \cite{fujita2017linear} or \emph{caterpillar width} \cite{Black2013ForbiddenMF}, is the linear restriction of branch-width.  
Equivalently, it is obtained by restricting the decomposition tree to a caterpillar.  
Like branch-width, it extends naturally to matroids and connectivity systems.

Typical obstruction-type objects for linear branch-width include linear tangles \cite{51}, single ideals \cite{52}, linear loose tangles \cite{3}, linear obstacles \cite{50}, ultra matroids \cite{7}, ultra antimatroids \cite{4}, ultra greedoids \cite{4}, and ultra quasi-matroids \cite{375}; see also \cite{423}.

\subsection{Selected Relations for Linear Branch-width}

\begin{theorem}
The following statements are known.
\begin{enumerate}
    \item For every graph \(G\),
    \[
    \operatorname{pw}(G)\leq \operatorname{lbw}(G)\leq \operatorname{pw}(G)+1,
    \]
    where \(\operatorname{lbw}(G)\) denotes the linear branch-width of \(G\). \cite{330}

    \item If a graph \(G\) has bounded linear branch-width, then it also has bounded branch-width.
\end{enumerate}
\end{theorem}

\begin{proof}
See the cited references.
\end{proof}

\section{Rank-width}

Rank-width is a graph width parameter based on the ranks of adjacency matrices across vertex cuts \cite{37}.  
It is closely related to matroid theory \cite{228,229,230} and to rank-type connectivity functions.  
A typical obstruction-type object for rank-width is a \(\rho\)-tangle \cite{53}.

Since rank-width is defined from a rank-like cut function, and several related notions extend naturally to matroids, we briefly recall one standard definition of a matroid.

\begin{definition}[Matroid via rank function]
\label{def:Matroid}
A \emph{matroid} on a finite set \(E\) is a function
\[
r:2^E\to \mathbb{N}_0
\]
satisfying the following conditions for all \(X,Y\subseteq E\):
\begin{enumerate}
    \item
    \[
    0\leq r(X)\leq |X|;
    \]
    \item if \(X\subseteq Y\), then
    \[
    r(X)\leq r(Y);
    \]
    \item
    \[
    r(X\cup Y)+r(X\cap Y)\leq r(X)+r(Y).
    \]
\end{enumerate}
\end{definition}

For further background on matroids, see, for example, \cite{oxley2006matroid,white1986theory}.

\begin{definition}[Rank-width]
\cite{37}
Let \(G=(V,E)\) be a graph.  
For each subset \(X\subseteq V\), let \(\rho_G(X)\) denote the rank over \(\mathbb{F}_2\) of the \(|X|\times |V\setminus X|\) adjacency matrix between \(X\) and \(V\setminus X\).  
By convention, \(\rho_G(X)=0\) when \(X=\emptyset\) or \(X=V\).

A \emph{rank-decomposition} of \(G\) is a pair \((T,L)\), where \(T\) is a subcubic tree and
\[
L:V(G)\to L(T)
\]
is a bijection from \(V(G)\) to the leaves of \(T\).

Each edge \(e\in E(T)\) induces a bipartition \((A_e,B_e)\) of the leaves of \(T\), and hence a bipartition
\[
\bigl(L^{-1}(A_e),\,L^{-1}(B_e)\bigr)
\]
of \(V(G)\).
The \emph{width} of \(e\) is defined as
\[
\rho_G\bigl(L^{-1}(A_e)\bigr).
\]
The \emph{width} of the rank-decomposition \((T,L)\) is the maximum width of its edges.

The \emph{rank-width} of \(G\), denoted by \(\operatorname{rw}(G)\), is the minimum width over all rank-decompositions of \(G\).  
If \(|V(G)|<2\), then \(\operatorname{rw}(G):=0\).
\end{definition}

\subsection{Selected Relations for Rank-width}

\begin{theorem}
The following statements are known.
\begin{enumerate}
    \item For every graph \(G\),
    \[
    \operatorname{rw}(G)\leq \operatorname{tw}(G)+1\leq \operatorname{pw}(G)+1.
    \]
    \cite{331}

    \item The neighborhood diversity of a graph is an upper bound for its rank-width \cite{346}.
\end{enumerate}
\end{theorem}

\begin{proof}
See the cited references.
\end{proof}

\subsection{Related Concepts for Rank-width}

Related variants include the following.

\begin{itemize}
    \item \textbf{Bi-rank-width} \cite{137}: an extension of rank-width to colored graphs.
    \item \textbf{Signed rank-width} \cite{156,231,232,233}: a variant used in model counting.
    \item \(\mathbb{Q}\)\textbf{-rank-width} \cite{365}: a version based on matrix rank over the rational field rather than over \(\mathbb{F}_2\).
\end{itemize}

Linear rank-width is the linear restriction of rank-width.  
Trivially, if a graph has bounded linear rank-width, then it also has bounded rank-width.

\section{Clique-width}

Clique-width is a graph width parameter based on graph construction using labeled vertices \cite{38}.  
It measures the complexity of a graph in terms of the number of labels needed to build it by a restricted set of operations.  
Recall that a \emph{clique} in a graph is a set of pairwise adjacent vertices \cite{szwarcfiter2003survey}.

\begin{definition}[Clique-width]
Let \(k\) be a positive integer.  
A \emph{\(k\)-graph} is a graph \(G\) together with a labeling function
\[
\operatorname{lab}:V(G)\to \{1,2,\dots,k\}.
\]

The following operations are allowed on \(k\)-graphs:
\begin{enumerate}
    \item \(\mathbf{i}\): create a new isolated vertex with label \(i\), for \(i\in\{1,\dots,k\}\);
    \item \(\eta_{i,j}\): add all edges between vertices labeled \(i\) and vertices labeled \(j\), where \(i\neq j\);
    \item \(\rho_{i\to j}\): relabel every vertex with label \(i\) by label \(j\);
    \item \(\oplus\): take the disjoint union of two \(k\)-graphs.
\end{enumerate}

A well-formed expression built from these operations is called a \emph{\(k\)-expression}.  
The \emph{clique-width} of a graph \(G\), denoted by \(\operatorname{cwd}(G)\), is the minimum \(k\) such that \(G\) can be generated by a \(k\)-expression.
\end{definition}

\subsection{Selected Relations for Clique-width}

\begin{theorem}
The following statements are known.
\begin{enumerate}
    \item If a graph \(G\) has bounded clique-width, then it also has bounded twin-width \cite{314}.

    \item If a graph \(G\) has bounded linear clique-width, then it also has bounded clique-width.

    \item For every graph \(G\),
    \[
    \operatorname{nlc}(G)\leq \operatorname{cwd}(G)\leq 2\,\operatorname{nlc}(G),
    \]
    where \(\operatorname{nlc}(G)\) denotes the NLC-width of \(G\). \cite{39}

    \item The stretch-width of a graph is at most twice its clique-width \cite{118}.

    \item If a graph \(G\) has bounded path-width, then it also has bounded clique-width \cite{325}.

    \item For every graph \(G\),
    \[
    \operatorname{fw}(G)\leq \operatorname{cwd}(G),
    \]
    where \(\operatorname{fw}(G)\) denotes the fusion-width of \(G\). \cite{64}

    \item If a graph \(G\) has bounded edge clique cover number, then it also has bounded clique-width \cite{beisegel2024simultaneous}.

    \item If \(G\) is given by a \(k\)-expression, then
    \[
    \operatorname{cwd}(G)\leq k.
    \]
    \cite{38}

    \item Neighborhood diversity is an upper bound for clique-width \cite{346}.

    \item If a graph \(G\) has bounded twin-cover, then it also has bounded rank-width and bounded clique-width \cite{408}.

    \item For every graph \(G\),
    \[
    \operatorname{sd}(G)\leq 2\,\operatorname{cwd}(G)-2,
    \]
    where \(\operatorname{sd}(G)\) denotes the symmetric-difference parameter. \cite{alecu2021graph}

    \item For every graph \(G\),
    \[
    \operatorname{fun}(G)\leq 2\,\operatorname{cwd}(G)-1,
    \]
    where \(\operatorname{fun}(G)\) denotes the functionality of \(G\). \cite{alecu2021graph,dvovrak2023bounds}
\end{enumerate}
\end{theorem}

\begin{proof}
See the cited references.
\end{proof}

\subsection{Related Concepts for Clique-width}

Important related concepts include:

\begin{itemize}
    \item \textbf{\(H\)-clique-width} \cite{139}: a parameter designed to connect hereditary width measures with graph product structure theory \cite{235}.
    \item \textbf{Signed clique-width} \cite{155}: a clique-width variant used in model counting.
    \item \textbf{Multi-clique-width} \cite{186}: a natural extension of clique-width.  
    It satisfies, for every graph \(G\),
    \[
    \operatorname{mcw}(G)\leq \operatorname{tw}(G)+2,
    \qquad
    \operatorname{boolw}(G)\leq \operatorname{mcw}(G)\leq 2^{\operatorname{boolw}(G)},
    \qquad
    \operatorname{mcw}(G)\leq \operatorname{cwd}(G).
    \]
    \cite{186}
    \item \textbf{Clique-partitioned tree-width} \cite{349}: a tree-width-like parameter that additionally reflects clique structure inside bags.
    \item \textbf{Symmetric clique-width} \cite{courcelle2004clique}: a variant based on equivalence classes of neighborhood behavior.
\end{itemize}

\section{Linear Clique-width}
Linear clique-width, also called \emph{sequential clique-width} \cite{fellows2005proving,fellows2005proving2}, is the linear restriction of clique-width.

\begin{theorem}
The following statements are known.
\begin{enumerate}
    \item If a graph \(G\) has bounded linear clique-width, then it also has bounded clique-width.

    \item If a graph \(G\) has bounded shrub-depth, then it also has bounded linear clique-width \cite{31,32}.

    \item If a graph \(G\) has bounded linear clique-width, then it also has bounded clique-tree-width \cite{39}.
\end{enumerate}
\end{theorem}

\begin{proof}
See the cited references.
\end{proof}
\section{NLC-width}

NLC-width and NLCT-width are graph width parameters based on graph composition operations similar to those used in the definition of clique-width \cite{39,234}.

\begin{definition}[NLC-width]
\cite{39,234}
Let \(k\) be a positive integer. The class \(\mathrm{NLC}_k\) of labeled graphs is defined recursively as follows.
\begin{enumerate}
    \item A single-vertex graph labeled by some \(a\in [k]\) belongs to \(\mathrm{NLC}_k\).

    \item Let \(G=(V_G,E_G,\lab_G)\) and \(H=(V_H,E_H,\lab_H)\) be two vertex-disjoint labeled graphs in \(\mathrm{NLC}_k\), and let \(S\subseteq [k]\times [k]\).  
    Then the graph \(G\times_S H\) is obtained from the disjoint union of \(G\) and \(H\) by adding all edges \(\{u,v\}\) with
    \[
    u\in V_G,\quad v\in V_H,\quad (\lab_G(u),\lab_H(v))\in S.
    \]
    The resulting labeled graph again belongs to \(\mathrm{NLC}_k\).

    \item Let \(G=(V_G,E_G,\lab_G)\in \mathrm{NLC}_k\), and let \(R:[k]\to [k]\) be a relabeling function.  
    Then \(\circ_R(G)\) is the graph obtained from \(G\) by replacing each label \(\lab_G(u)\) with \(R(\lab_G(u))\).  
    This graph also belongs to \(\mathrm{NLC}_k\).
\end{enumerate}

The \emph{NLC-width} of a graph \(G\), denoted by \(\operatorname{nlcw}(G)\), is the smallest integer \(k\) such that \(G\in \mathrm{NLC}_k\).
\end{definition}

\subsection{Selected Relations for NLC-width}

\begin{theorem}
The following statements are known.
\begin{enumerate}
    \item For every graph \(G\),
    \[
    \operatorname{nlcw}(G)\leq \operatorname{cwd}(G)\leq 2\,\operatorname{nlcw}(G).
    \]
    \cite{39}

    \item For every graph \(G\),
    \[
    \operatorname{nlcw}(G)\leq \operatorname{nlctw}(G).
    \]
    \cite{39}

    \item For every graph \(G\),
    \[
    \operatorname{nlcw}(G)\leq \operatorname{ctw}(G),
    \]
    where \(\operatorname{ctw}(G)\) denotes clique-tree-width. \cite{39}

    \item For every graph \(G\),
    \[
    \operatorname{nlctw}(G)\leq \operatorname{ctw}(G)\leq \operatorname{nlctw}(G)+1.
    \]
    \cite{39}
\end{enumerate}
\end{theorem}

\begin{proof}
See the cited references.
\end{proof}

\section{Linear NLC-width}

Linear NLC-width is the linear-layout restriction of NLC-width \cite{39}.

\subsection{Selected Relations for Linear NLC-width}

\begin{theorem}
The following statements are known.
\begin{enumerate}
    \item If a graph \(G\) has bounded linear NLC-width, then it also has bounded NLC-width.

    \item If a graph \(G\) has bounded linear NLC-width, then it also has bounded NLCT-width \cite{39}.

    \item For every graph \(G\),
    \[
    \operatorname{lnlcw}(G)\leq \operatorname{lcwd}(G)\leq \operatorname{lnlcw}(G)+1,
    \]
    where \(\operatorname{lcwd}(G)\) denotes the linear clique-width of \(G\). \cite{39}
\end{enumerate}
\end{theorem}

\begin{proof}
See the cited references.
\end{proof}

\section{Hypertree-width}
Hypertree-width is a width parameter for hypergraphs that extends tree-like decomposition methods beyond ordinary graphs \cite{40}.  
It is particularly important in database theory, constraint satisfaction, and structural hypergraph theory.

We first recall the notion of a hypergraph; see, for example, \cite{bretto2013hypergraph,feng2019hypergraph} for further background.

\begin{definition}[Hypergraph]
\cite{bretto2013hypergraph}
A \emph{hypergraph} is a pair \(H=(V,E)\), where \(V\) is a finite set of vertices and \(E\) is a finite family of nonempty subsets of \(V\), called \emph{hyperedges}.  
The vertex set and hyperedge family are denoted by \(V(H)\) and \(E(H)\), respectively.
\end{definition}

\begin{example}
Let
\[
V(H)=\{A,B,C,D,E\},
\qquad
E(H)=\bigl\{\{A,D\},\{D,E\},\{A,B,C\}\bigr\}.
\]
Then \(H=(V(H),E(H))\) is a hypergraph with five vertices and three hyperedges.
\end{example}

\begin{definition}[Hypertree decomposition]
\cite{40}
A \emph{hypertree decomposition} of a hypergraph \(H=(V,E)\) is a triple
\[
(T,\chi,\lambda),
\]
where
\begin{itemize}
    \item \(T\) is a tree,
    \item \(\chi\) assigns to each node \(t\in V(T)\) a bag \(\chi(t)\subseteq V\),
    \item \(\lambda\) assigns to each node \(t\in V(T)\) a set \(\lambda(t)\subseteq E\) of hyperedges.
\end{itemize}
These data must satisfy:
\begin{enumerate}
    \item for every hyperedge \(e\in E(H)\), there exists \(t\in V(T)\) such that
    \[
    e\subseteq \chi(t);
    \]
    \item for every vertex \(v\in V(H)\), the set
    \[
    \{\,t\in V(T)\mid v\in \chi(t)\,\}
    \]
    induces a connected subtree of \(T\);
    \item for every node \(t\in V(T)\),
    \[
    \chi(t)\subseteq \bigcup \lambda(t);
    \]
    \item for every node \(t\) and every descendant \(t'\) of \(t\) in \(T\), if \(e\in \lambda(t)\), then
    \[
    (e\setminus \chi(t))\cap \chi(t')=\emptyset.
    \]
\end{enumerate}
\end{definition}

\begin{definition}[Hypertree-width]
The \emph{width} of a hypertree decomposition \((T,\chi,\lambda)\) is
\[
\max_{t\in V(T)} |\lambda(t)|.
\]
The \emph{hypertree-width} of a hypergraph \(H\), denoted by \(\operatorname{htw}(H)\), is the minimum width over all hypertree decompositions of \(H\).
\end{definition}

Known obstruction-type objects include hypertangles and hyperbrambles \cite{250}.  
It would also be natural to investigate hypergraph versions of ultrafilters analogous to those considered in \cite{16}; compare also \cite{Malliaris2012HypergraphSA}.

\subsection{Related Concepts for Hypertree-width}
Important related notions include:

\begin{itemize}
    \item \textbf{Hyperbranch-width} \cite{250}
    \item \textbf{Hyper-T-width} and \textbf{Hyper-D-width} \cite{44}
    \item \textbf{Fractional hypertree-width} \cite{41}
    \item \textbf{FAQ-width} \cite{abo2016faq,khamis2015faq}
    \item \textbf{m-width} \cite{joglekar2015s}
    \item \textbf{Hyperpath-width} and \textbf{hypertree-depth} \cite{adler2012hypertree}
    \item \textbf{Closure tree-width} and \textbf{closure hypertree-width} \cite{adler2008tree}
    \item \textbf{Biconnected width} \cite{freuder1985sufficient}
    \item \textbf{TCLUSTER-width} \cite{dechter1989tree}
    \item \textbf{Hinge-width} \cite{gyssens1994decomposing}
    \item \textbf{SuperHyperTree-width}
    \cite{fujita2025shortclique,fujita2025superhyperbranch,
    smarandache2025superhypertopologies,
    FujitaSmarandache2026HyperGraphSuperHyperGraph}
    \item \textbf{Hierarchical hypertree-width} \cite{hu2022computing}
    \item \textbf{Nest set-width} \cite{lanzinger2023tractability}
    \item \textbf{\(\beta\)-hyperorder width} and \textbf{\(\beta\)-fractional hyperorder width} \cite{gottlob2004hypergraphs,capelli2023direct}
    \item \textbf{\#-hypertree-width} \cite{chen2023counting,greco2014counting}
    \item \textbf{Primal tree-width} and related incidence-based parameters \cite{arteche2022parameterized,ganian2021new}
\end{itemize}

Various hierarchies among these parameters are also known.  
For instance, \(\beta\)-hypertree-width is bounded by incidence clique-width, incidence clique-width is bounded by signed incidence clique-width and modular incidence tree-width, and both of these are bounded by incidence tree-width, which in turn is bounded by primal tree-width \cite{fischer2008counting,ganian2021new}.

\begin{question}
What new phenomena arise when these width parameters are extended further to SuperHyperGraphs?
\end{question}

\subsection{Applications of Hypertree-width}

Applications of hypertree-width include:

\begin{itemize}
    \item \textbf{Artificial intelligence} \cite{eiter2002hypergraph,ortiz2017tractable}
    \item \textbf{Database theory and query evaluation} \cite{ghionna2006hypertree,tu2015duncecap}
\end{itemize}

\section{Modular-width}

Modular-width is a graph parameter defined via modular decomposition \cite{42}.  
It measures how complicated a graph is after recursively decomposing it into modules.

\begin{definition}[Modular-width]
\cite{42}
Let \(\mathcal{MW}_k\) be the smallest class of graphs satisfying the following conditions:
\begin{enumerate}
    \item \(K_1\in \mathcal{MW}_k\);
    \item if \(G_1,\dots,G_p\in \mathcal{MW}_k\) and \(H\) is a graph on \(p\leq k\) vertices, then the substitution graph
    \[
    H(G_1,\dots,G_p)
    \]
    belongs to \(\mathcal{MW}_k\).
\end{enumerate}
The \emph{modular-width} of a graph \(G\), denoted by \(\operatorname{mw}(G)\), is the smallest integer \(k\) such that \(G\in \mathcal{MW}_k\).
\end{definition}

\begin{remark}
The definitions based on tree-models and \(\mathrm{TM}_m(d)\) belong to the theory of shrub-depth rather than modular-width.  
Shrub-depth is related, but it is a different parameter.
\end{remark}

\begin{theorem}
The following statements are known.
\begin{enumerate}
    \item If a graph \(G\) has bounded modular-width, then it also has bounded clique-width \cite{343}.

    \item If a graph \(G\) has bounded vertex cover number or bounded neighborhood diversity, then it also has bounded modular-width \cite{343}.

    \item If a graph \(G\) has bounded twin-cover, then it also has bounded modular-width \cite{343}.
\end{enumerate}
\end{theorem}

\begin{proof}
See the cited references.
\end{proof}

\subsection{Related Concepts for Modular-width}
A closely related parameter is \emph{shrub-depth}.  
Recall that a graph class \(\mathcal{G}\) has shrub-depth \(d\) if there exists an integer \(m\) such that every graph in \(\mathcal{G}\) admits a tree-model with \(m\) colors and depth \(d\), but for every \(m\), the class \(\mathcal{G}\) is not contained in \(\mathrm{TM}_m(d-1)\) \cite{42}.

\section{Submodular-width}
Submodular-width is a hypergraph width parameter defined by means of submodular functions \cite{43}.  
It plays an important role in the algorithmic study of conjunctive queries and constraint satisfaction.  
A related refinement is \emph{sharp submodular-width}, which yields improved algorithmic bounds for evaluating queries over arbitrary semirings \cite{khamis2020functional}.

\section{Amalgam-width}
Amalgam-width is a matroid width parameter based on matroid amalgamation \cite{30}.  
Here matroid amalgamation refers to the process of combining two matroids on overlapping ground sets into a single matroid while preserving matroid structure; see, for example, \cite{Poljak1984AmalgamationOU,Nesetril1981AmalgamationOM}.

\section{Kelly-width}

Kelly-width is a width parameter for directed graphs that plays a role analogous to tree-width in the directed setting \cite{28,29}.  
A related linear parameter is Kelly-path-width \cite{357}.

\begin{definition}[Kelly-decomposition]
\cite{28}
A \emph{Kelly-decomposition} of a directed graph \(G\) is a triple
\[
\bigl(D,(B_t)_{t\in V(D)},(W_t)_{t\in V(D)}\bigr),
\]
where:
\begin{itemize}
    \item \(D\) is a directed acyclic graph,
    \item \((B_t)_{t\in V(D)}\) is a partition of \(V(G)\),
    \item for each \(t\in V(D)\), the set \(W_t\subseteq V(G)\) guards
    \[
    B_t^\downarrow:=\bigcup_{t'\leq_D t} B_{t'},
    \]
    where \(\leq_D\) denotes the reachability order in \(D\),
    \item for each node \(s\in V(D)\), there exists a linear ordering \(t_1,\dots,t_p\) of the children of \(s\) such that
    \[
    W_{t_i}\subseteq B_s\cup W_s\cup \bigcup_{j<i} B_{t_j}^\downarrow
    \qquad (1\leq i\leq p),
    \]
    and similarly for the roots of \(D\).
\end{itemize}
The \emph{width} of the Kelly-decomposition is
\[
\max_{t\in V(D)} |B_t\cup W_t|.
\]
The \emph{Kelly-width} of \(G\) is the minimum width over all Kelly-decompositions of \(G\).
\end{definition}

\subsection{Selected Relations for Kelly-width}

\begin{theorem}
The following statements are known.
\begin{enumerate}
    \item A directed graph \(G\) has directed elimination width at most \(k\) if and only if it has Kelly-width at most \(k+1\). \cite{28}

    \item Kelly-width is at most directed path-width plus one \cite{ganian2014digraph}.

    \item If a digraph \(G\) has bounded Kelly-width, then it also has bounded directed tree-width \cite{368,369}.

    \item If a digraph has bounded Kelly-path-width, then it also has bounded Kelly-width.
\end{enumerate}
\end{theorem}

\begin{proof}
See the cited references.
\end{proof}

\section{Monoidal Width}

Monoidal width \cite{27} measures the complexity of morphisms in monoidal categories.  
Recall that a monoidal category is a category equipped with a bifunctor (tensor product) and a unit object, allowing objects and morphisms to be combined in a coherent algebraic manner; see, for example, \cite{Turaev2017MonoidalCA}.  
The tensor product of vector spaces, for instance, provides a standard example of such a structure \cite{cheng2012introduction,smolensky1990tensor}.

\section{Tree-cut Width}

Tree-cut width is a graph width parameter designed as an edge-cut analogue of tree-width \cite{18}.  
A related parameter is \(0\)-tree-cut width \cite{17}.

\begin{definition}[Tree-cut decomposition and tree-cut width]
\cite{18}
A \emph{tree-cut decomposition} of a graph \(G\) is a pair
\[
(T,\mathcal{X}),
\]
where \(T\) is a tree and
\[
\mathcal{X}=\{X_t\subseteq V(G)\mid t\in V(T)\}
\]
is a near-partition of \(V(G)\). The sets \(X_t\) are called \emph{bags}.

For an edge \(e=uv\in E(T)\), let \(T_u\) and \(T_v\) be the two components of \(T-e\) containing \(u\) and \(v\), respectively.  
The corresponding partition of \(V(G)\) is
\[
\left(\bigcup_{t\in V(T_u)}X_t,\ \bigcup_{t\in V(T_v)}X_t\right),
\]
and the set of graph edges with one endpoint in each part is denoted by \(\operatorname{cut}(e)\).

If \(T\) is rooted at \(r\), then for each \(t\in V(T)\setminus\{r\}\), let \(e(t)\) be the unique edge on the path from \(r\) to \(t\).  
The \emph{adhesion} of \(t\) is
\[
\operatorname{adh}(t):=|\operatorname{cut}(e(t))|,
\]
and \(\operatorname{adh}(r):=0\).

For a node \(t\in V(T)\), the \emph{torso} \(H_t\) is obtained by contracting, for each component \(T_i\) of \(T-t\), the vertex set
\[
Z_i:=\bigcup_{b\in V(T_i)} X_b
\]
to a single vertex \(z_i\). Parallel edges may arise.

The \emph{3-center} of \(H_t\) with respect to \(X_t\), denoted by \(\widetilde{H}_t\), is obtained by repeatedly suppressing vertices in \(V(H_t)\setminus X_t\) of degree at most \(2\).  
The \emph{torso-size} of \(t\) is
\[
\operatorname{tor}(t):=|V(\widetilde{H}_t)|.
\]

The \emph{width} of the tree-cut decomposition \((T,\mathcal{X})\) is
\[
\max_{t\in V(T)} \max\{\operatorname{adh}(t),\operatorname{tor}(t)\}.
\]
The \emph{tree-cut width} of \(G\), denoted by \(\operatorname{tcw}(G)\), is the minimum width over all tree-cut decompositions of \(G\).
\end{definition}

\subsection{Related Concepts for Tree-cut Width}

Several related parameters further develop the theory of tree-cut width.

\begin{itemize}
    \item \textbf{Slim tree-cut width} \cite{17}: a less restrictive edge-cut-based analogue of tree-width that still retains strong structural and algorithmic properties.  
    If a graph \(G\) has bounded slim tree-cut width, then it also has bounded tree-cut width \cite{17}.

    \item \textbf{Edge-crossing width} and \textbf{\(\alpha\)-edge-crossing width} \cite{19,251}: parameters defined by the number of edges crossing a bag in a tree-cut decomposition.

    \item \textbf{Edge-cut width} \cite{20}: an algorithmically motivated analogue of tree-width based on edge cuts.  
    If a graph \(G\) has bounded feedback edge set number, then it also has bounded edge-cut width \cite{19}.

    \item \textbf{Edge-tangles} \cite{8}: obstruction-type objects closely related to edge-cut decompositions.
\end{itemize}

\section{Resolution Width}

Resolution width \cite{21} is a proof-complexity parameter connected with the existential pebble game in finite model theory.  
A related linear variant is \emph{linear resolution width}.

\section{Twin-width}

Twin-width is a graph width parameter introduced in \cite{22,23,24}.  
It measures the complexity of a graph through contraction sequences in trigraphs.  
Intuitively, it quantifies how far a graph is from behaving like a cograph, where repeated contraction of twin vertices is possible without creating too much irregularity.  
Graph classes of bounded twin-width include, for example, graphs of bounded tree-width and graphs of bounded clique-width.

\begin{definition}[Twin-width]
(cf.\cite{22,23,24})
A \emph{trigraph} is a triple \(G=(V,E,R)\), where \(E\) is the set of black edges and \(R\) is the set of red edges.

Let \(u,v\in V\) be distinct vertices. The \emph{contraction} of \(u\) and \(v\) produces a new vertex \(w\), replacing \(u\) and \(v\), and yields a trigraph \(G/u,v=(V',E',R')\), where
\[
V'=(V\setminus\{u,v\})\cup\{w\}.
\]
For every vertex \(x\in V\setminus\{u,v\}\),
\begin{itemize}
    \item \(wx\in E'\) if and only if both \(ux\in E\) and \(vx\in E\),
    \item \(wx\notin E'\cup R'\) if and only if both \(ux\notin E\cup R\) and \(vx\notin E\cup R\),
    \item \(wx\in R'\) otherwise.
\end{itemize}

A trigraph is called a \emph{\(d\)-trigraph} if every vertex is incident with at most \(d\) red edges.  
A contraction is called a \emph{\(d\)-contraction} if both the original trigraph and the contracted trigraph are \(d\)-trigraphs.

A trigraph \(G\) on \(n\) vertices is \emph{\(d\)-collapsible} if there exists a sequence
\[
G=G_n,G_{n-1},\dots,G_1
\]
such that each \(G_{i-1}\) is obtained from \(G_i\) by a \(d\)-contraction, and \(G_1\) consists of a single vertex.

The \emph{twin-width} of \(G\), denoted by \(\operatorname{tww}(G)\), is the minimum integer \(d\) such that \(G\) is \(d\)-collapsible.
\end{definition}

\subsection{Selected Relations for Twin-width}

\begin{theorem}
The following statements are known.
\begin{enumerate}
    \item Every graph class of bounded genus has bounded twin-width \cite{341}.
    \item Every graph class of bounded rank-width has bounded twin-width \cite{Bonnet2020TwinwidthIT}.
    \item Every graph class of bounded stack-number or bounded queue-number has bounded twin-width \cite{bonnet2021twin}.
    \item Every graph class of bounded distance to planar graphs, bounded sparse twin-width, or bounded feedback edge set number has bounded twin-width \cite{341}.
    \item If a graph \(G\) has bounded clique-width, then it also has bounded twin-width \cite{314}.
    \item If a graph \(G\) has bounded twin-width, then it is monadically dependent \cite{314}.
\end{enumerate}
\end{theorem}

\begin{proof}
See the cited references.
\end{proof}

\subsection{Related Concepts for Twin-width}

\begin{itemize}
    \item \textbf{Sparse twin-width} \cite{171,252}: a sparse analogue of twin-width.  
    If a graph \(G\) has bounded sparse twin-width, then it belongs to a nowhere dense class \cite{314}.

    \item \textbf{Signed twin-width} \cite{203}: a variant motivated by applications to CNF formulas and bounded-width monotone circuits.
\end{itemize}

\section{Decomposition Width}

Decomposition width \cite{54} is a width notion related to matroid theory and second-order logic.

\section{Minor-matching Hypertree Width}

Minor-matching hypertree width \cite{55} is a width parameter for graphs and hypergraphs that controls the size of independent sets appearing in bags of tree-like decompositions.

\section{Tree-clique Width}

Tree-clique width \cite{56} extends some of the algorithmic advantages of tree-width to denser but still structured graph classes.

\section{Tree-distance-width}

Tree-distance-width (TDW) is a distance-based width parameter related to tree-width \cite{57}.  
Related notions include rooted tree-distance-width (RTDW) \cite{57} and connected tree-distance-width \cite{309,310}.

\begin{definition}[Tree-distance decomposition]
\cite{57}
A \emph{tree-distance decomposition} of a graph \(G=(V,E)\) is a triple
\[
(\{X_i\}_{i\in I},\,T=(I,F),\,r),
\]
where \(T\) is a rooted tree with root \(r\), such that:
\begin{enumerate}
    \item
    \[
    \bigcup_{i\in I} X_i = V
    \quad\text{and}\quad
    X_i\cap X_j=\emptyset \text{ for } i\neq j;
    \]
    \item if \(v\in X_i\), then the distance from \(v\) to the root set \(X_r\) in \(G\) is equal to the distance from \(i\) to \(r\) in \(T\);
    \item for every edge \(\{v,w\}\in E\), if \(v\in X_i\) and \(w\in X_j\), then either \(i=j\) or \(ij\in F\).
\end{enumerate}
The \emph{width} of the decomposition is
\[
\max_{i\in I}|X_i|.
\]
The \emph{tree-distance-width} of \(G\), denoted by \(\operatorname{tdw}(G)\), is the minimum width over all tree-distance decompositions of \(G\).

If \(|X_r|=1\), then the decomposition is called \emph{rooted}, and the minimum width over rooted tree-distance decompositions is the \emph{rooted tree-distance-width}, denoted by \(\operatorname{rtdw}(G)\).
\end{definition}

\subsection{Selected Relations for Tree-distance-width}

\begin{theorem}
The following statements are known.
\begin{enumerate}
    \item
    \[
    \operatorname{tw}(G)\leq 2\,\operatorname{tdw}(G)-1.
    \]
    \item
    \[
    \operatorname{tdw}(G)\leq \operatorname{rtdw}(G).
    \]
    \item
    \[
    \operatorname{tdw}(G)\leq \operatorname{pdw}(G)\leq \operatorname{rpdw}(G),
    \]
    where \(\operatorname{pdw}(G)\) and \(\operatorname{rpdw}(G)\) denote path-distance-width and rooted path-distance-width, respectively.
\end{enumerate}
\end{theorem}

\begin{proof}
See the cited references.
\end{proof}

\section{Path-distance-width}

Path-distance-width (PDW) is the path analogue of tree-distance-width \cite{Althoubi2017AnAO}.  
Related notions include rooted path-distance-width (RPDW) \cite{57}, connected path-distance-width, and \(c\)-connected path-distance-width \cite{309,310}.  
Distance-based decompositions are useful, for example, in graph isomorphism and related structural problems \cite{Read1977TheGI,Grohe2020TheGI,Li2002OnIO}.

\begin{definition}[Path-distance decomposition]
\cite{57}
A \emph{path-distance decomposition} of a graph \(G\) is a tree-distance decomposition in which the underlying rooted tree \(T\) is a path.  
For simplicity, such a decomposition is often written as
\[
(X_1,X_2,\dots,X_t),
\]
where \(X_1\) is the root set.

The corresponding minimum width is called the \emph{path-distance-width}, denoted by \(\operatorname{pdw}(G)\).  
If the root set has size one, the minimum width is called the \emph{rooted path-distance-width}, denoted by \(\operatorname{rpdw}(G)\).
\end{definition}

\begin{theorem}
The following statements are known.
\begin{enumerate}
    \item
    \[
    \operatorname{bw}(G)\leq 2\,\operatorname{pdw}(G)-1,
    \]
    where here \(\operatorname{bw}(G)\) denotes the bandwidth of \(G\).

    \item If a graph \(G\) has bounded path-distance-width, then it also has bounded tree-distance-width.

    \item If a graph \(G\) has bounded rooted path-distance-width, then it also has bounded rooted tree-distance-width.
\end{enumerate}
\end{theorem}

\begin{proof}
See the cited references.
\end{proof}

\section{Tree-partition-width}
Tree-partition-width \cite{58} is the minimum width of a tree-partition of a graph.  
It is also called \(0\)-quasi-tree-partition-width \cite{253}.

\begin{definition}[Tree-partition and tree-partition-width]
\cite{58}
A graph \(H\) is called a \emph{partition} of a graph \(G\) if:
\begin{enumerate}
    \item each vertex of \(H\) is a subset of \(V(G)\), called a \emph{bag};
    \item every vertex of \(G\) belongs to exactly one bag;
    \item two bags \(A\) and \(B\) are adjacent in \(H\) if and only if some edge of \(G\) has one endpoint in \(A\) and the other in \(B\).
\end{enumerate}
The \emph{width} of the partition is the maximum size of a bag.

If \(H\) is a forest, then it is called a \emph{tree-partition} of \(G\).  
The \emph{tree-partition-width} of \(G\), denoted by \(\operatorname{tpw}(G)\), is the minimum width over all tree-partitions of \(G\).
\end{definition}

\begin{theorem}
The following statements are known.
\begin{enumerate}
    \item
    \[
    \operatorname{dtw}(G)\geq \operatorname{tpw}(G)-1,
    \]
    where \(\operatorname{dtw}(G)\) denotes domino tree-width \cite{58}.

    \item
    \[
    2\,\operatorname{tpw}(G)\geq \operatorname{tw}(G)+1.
    \]
    \cite{58}

    \item
    \[
    \operatorname{tn}(G)\leq 3\,\operatorname{tpw}(G),
    \]
    where \(\operatorname{tn}(G)\) denotes the track-number of \(G\). \cite{dujmovic2005layout}

    \item
    \[
    \operatorname{qn}(G)\leq 3\,\operatorname{tpw}(G)-1,
    \]
    where \(\operatorname{qn}(G)\) denotes the queue-number of \(G\). \cite{dujmovic2005layout}
\end{enumerate}
\end{theorem}

\begin{proof}
See the cited references.
\end{proof}

Related notions include directed tree-partition-width \cite{191} and star-partition-width \cite{Wood2022ProductSO}.

\section{Path-partition-width}
Path-partition-width is the minimum width of a path-partition of a graph.  
It is also called \(0\)-quasi-path-partition-width \cite{253}.  
A directed version has also been studied \cite{191}.

\subsection{Selected Relations for Path-partition-width}

\begin{theorem}
The following statements are known.
\begin{enumerate}
    \item For every graph \(G\),
    \[
    \frac{1}{2}\bigl(\operatorname{bw}(G)+1\bigr)\leq \operatorname{ppw}(G)\leq \operatorname{bw}(G),
    \]
    where \(\operatorname{bw}(G)\) denotes bandwidth and \(\operatorname{ppw}(G)\) denotes path-partition-width. \cite{dujmovic2007graph}

    \item If a graph \(G\) has bounded path-partition-width, then it also has bounded tree-partition-width.
\end{enumerate}
\end{theorem}

\begin{proof}
See the cited references.
\end{proof}

\section{CV-width}
CV-width \cite{59} is a width parameter for CNF formulas.  
It dominates the tree-width of the incidence graph of a CNF.  
A linear variant, called \emph{linear CV-width}, is defined analogously.

\section{Dominating-set-width}
Dominating-set-width \cite{60} measures the number of distinct dominating sets that can occur on a subgraph.  
Recall that a dominating set in a graph is a vertex set such that every vertex lies in the set or is adjacent to a vertex in it \cite{Wang2017ANC}.

\section{Point Width}

Point width \cite{61} is a hypergraph width parameter that provides a tractability condition for Max-CSPs, generalizing both bounded MIM-width and \(\beta\)-acyclicity.

\section{Neighbourhood Width}
Neighbourhood width \cite{62,63} is a layout parameter based on the number of distinct neighborhoods across cuts.  
It is known that neighbourhood width is at most path-width plus one.

\begin{definition}[Neighbourhood width]
\cite{62,63}
A \emph{linear layout} of a graph \(G=(V,E)\) is a bijection
\[
\varphi:V\to \{1,\dots,|V|\}.
\]
For \(1\leq i\leq |V|\), define
\[
L(i,\varphi,G)=\{u\in V\mid \varphi(u)\leq i\},
\qquad
R(i,\varphi,G)=\{u\in V\mid \varphi(u)>i\}.
\]

For disjoint vertex sets \(U,W\subseteq V\), define
\[
N_W(u):=\{v\in W\mid \{u,v\}\in E\},
\]
and
\[
N(U,W):=\{N_W(u)\mid u\in U\}.
\]

The \emph{neighbourhood width} of \(G\) is
\[
\operatorname{nw}(G)=
\min_{\varphi}
\max_{1\leq i\leq |V|-1}
|N(L(i,\varphi,G),R(i,\varphi,G))|.
\]
\end{definition}

\section{Fusion-width}
Fusion-width \cite{64} generalizes both tree-width and clique-width.

\begin{definition}[\(k\)-fusion-tree expression]
\cite{64}
A \emph{\(k\)-fusion-tree expression} is built using the following operations:
\begin{itemize}
    \item \(i^{(m)}\): create \(m\) isolated vertices with label \(i\);
    \item \(\eta_{i,j}\): add all edges between vertices labeled \(i\) and vertices labeled \(j\), for \(i\neq j\);
    \item \(\rho_{i\to j}\): relabel all vertices labeled \(i\) by \(j\);
    \item \(\theta_i\): merge all vertices labeled \(i\) into a single vertex;
    \item \(\oplus\): take the disjoint union.
\end{itemize}
The resulting graph is obtained after deleting all labels.
\end{definition}

\begin{definition}[Fusion-width]
\cite{64}
The \emph{fusion-width} of a graph \(G\), denoted by \(\operatorname{fw}(G)\), is the smallest \(k\) such that \(G\) can be generated by a \(k\)-fusion-tree expression.
\end{definition}

\begin{theorem}
The following statements are known.
\begin{enumerate}
    \item
    \[
    \operatorname{fw}(G)\leq \operatorname{tw}(G)+2.
    \]
    \cite{64}

    \item
    \[
    \operatorname{fw}(G)\leq \operatorname{cwd}(G).
    \]
    \cite{64}
\end{enumerate}
\end{theorem}

\begin{proof}
See the cited references.
\end{proof}

\section{Directed NLC-width}
Directed NLC-width \cite{65} is the directed analogue of NLC-width.

\begin{definition}[Directed NLC-width]
\label{def:dnlcw}
\cite{65}
Let \(k\) be a positive integer. The class \(\mathrm{dNLC}_k\) of labeled digraphs is defined recursively as follows.
\begin{enumerate}
    \item A single-vertex digraph \(\bullet_a\) with label \(a\in [k]\) belongs to \(\mathrm{dNLC}_k\).

    \item Let \(G=(V_G,E_G,\lab_G)\) and \(H=(V_H,E_H,\lab_H)\) be two vertex-disjoint labeled digraphs in \(\mathrm{dNLC}_k\), and let
    \[
    \overrightarrow{S},\overleftarrow{S}\subseteq [k]\times [k].
    \]
    Then \(G\otimes(\overrightarrow{S},\overleftarrow{S})H\) is obtained from the disjoint union of \(G\) and \(H\) by adding arcs
    \[
    (u,v)\quad\text{whenever }u\in V_G,\ v\in V_H,\ (\lab_G(u),\lab_H(v))\in \overrightarrow{S},
    \]
    and arcs
    \[
    (v,u)\quad\text{whenever }u\in V_G,\ v\in V_H,\ (\lab_G(u),\lab_H(v))\in \overleftarrow{S}.
    \]

    \item If \(G=(V_G,E_G,\lab_G)\in \mathrm{dNLC}_k\) and \(R:[k]\to [k]\), then \(\circ_R(G)\) is obtained from \(G\) by replacing each label \(\lab_G(u)\) with \(R(\lab_G(u))\).
\end{enumerate}

The \emph{directed NLC-width} of a digraph \(G\), denoted by \(\operatorname{dnlcw}(G)\), is the minimum integer \(k\) such that \(G\) admits a labeling making it an element of \(\mathrm{dNLC}_k\).
\end{definition}

\section{Directed Tree-width}
Directed tree-width \cite{66,270} is a directed analogue of tree-width.  
Typical obstruction-type objects include directed tangles \cite{72} and directed ultrafilters \cite{12}.

\begin{definition}[Directed tree-width]
\cite{66,270}
An \emph{arboreal tree-decomposition} of a digraph \(G=(V_G,E_G)\) is a triple
\[
(T,X,W),
\]
where:
\begin{itemize}
    \item \(T\) is an out-tree;
    \item \(W=\{W_r\mid r\in V(T)\}\) is a partition of \(V_G\) into nonempty sets;
    \item \(X=\{X_e\mid e\in E(T)\}\) is a family of subsets of \(V_G\);
    \item for every arc \((u,v)\in E(T)\), the set
    \[
    \bigcup_{r:\, v\leq_T r} W_r
    \]
    is \(X_{(u,v)}\)-normal.
\end{itemize}

The \emph{width} of \((T,X,W)\) is
\[
\max_{r\in V(T)}
\left(
|W_r\cup \bigcup_{e\sim r} X_e|-1
\right),
\]
where \(e\sim r\) means that \(r\) is incident with the arc \(e\).

The \emph{directed tree-width} of \(G\), denoted by \(\operatorname{dtw}(G)\), is the minimum width over all arboreal tree-decompositions of \(G\).
\end{definition}

A closely related concept is DAG-width \cite{berwanger2006dag}.  
For example, if a digraph has entanglement \(k\), then its DAG-width is at most \(k+1\) \cite{berwanger2006dag}.

\begin{definition}[Guarding]
\label{def:guards}
\cite{berwanger2006dag}
Let \(G=(V,E)\) be a digraph.  
A set \(W\subseteq V\) \emph{guards} a set \(V'\subseteq V\) if every arc \((u,v)\in E\) with \(u\in V'\) and \(v\notin V'\) satisfies \(v\in W\).
\end{definition}

\begin{definition}[DAG-decomposition]
\label{def:DAGdecomp}
\cite{berwanger2006dag}
A \emph{DAG-decomposition} of a digraph \(G=(V,E)\) is a pair
\[
(D,\{X_d\}_{d\in V(D)}),
\]
where \(D\) is a DAG and the bags \(X_d\subseteq V\) satisfy:
\begin{enumerate}
    \item
    \[
    \bigcup_{d\in V(D)} X_d = V;
    \]
    \item if \(d\preceq d'\preceq d''\) in \(D\), then
    \[
    X_d\cap X_{d''}\subseteq X_{d'};
    \]
    \item if \(d\) is a root of \(D\), then \(X_d\) is guarded by \(\emptyset\);
    \item for every arc \((d,d')\in E(D)\), the set \(X_d\cap X_{d'}\) guards
    \[
    \bigcup_{d'\preceq d''}X_{d''}\setminus X_d.
    \]
\end{enumerate}

The \emph{width} of the decomposition is
\[
\max_{d\in V(D)} |X_d|.
\]
The minimum such width is the \emph{DAG-width} of \(G\).
\end{definition}

\begin{theorem}
The following statements are known.
\begin{enumerate}
    \item If a digraph \(G\) has bounded DAG-width, then it also has bounded directed tree-width \cite{368}.

    \item If a digraph \(G\) has bounded Kelly-width, then it also has bounded directed tree-width \cite{368,369}.

    \item For every digraph \(D\),
    \[
    \operatorname{dtw}(D)\geq \operatorname{cycle\mbox{-}degeneracy}(D).
    \]

    \item For every digraph \(G\),
    \[
    \frac{1}{3}\bigl(\operatorname{dtw}(G)+2\bigr)\leq \operatorname{entanglement}(G).
    \]
    \cite{johnson2001directed}
\end{enumerate}
\end{theorem}

\begin{proof}
See the cited references.
\end{proof}

\section{Directed Path-width}
Directed path-width \cite{67,70} is the directed analogue of path-width.

\begin{definition}[Directed path-decomposition]
\cite{67,70}
Let \(G=(V,E)\) be a digraph.  
A \emph{directed path-decomposition} of \(G\) is a sequence
\[
(X_1,\dots,X_r)
\]
of subsets of \(V\), called \emph{bags}, such that:
\begin{enumerate}
    \item
    \[
    X_1\cup \cdots \cup X_r = V;
    \]
    \item for every arc \((u,v)\in E\), there exist indices \(i\leq j\) such that
    \[
    u\in X_i,\qquad v\in X_j;
    \]
    \item for all \(1\leq i<j<\ell\leq r\),
    \[
    X_i\cap X_\ell \subseteq X_j.
    \]
\end{enumerate}
The \emph{width} of the decomposition is
\[
\max_{1\leq i\leq r}|X_i|-1.
\]
The \emph{directed path-width} of \(G\), denoted by \(\operatorname{dpw}(G)\), is the minimum width over all directed path-decompositions of \(G\).
\end{definition}

\begin{theorem}
\cite{FujitaDirectedPath2024}
Let \(D\) be a digraph. Then the following statements are known.
\begin{enumerate}
    \item Directed path-width is equal to the directed vertex separation number \cite{bodlaender1998linear}.

    \item If \(D\) has bounded directed feedback vertex set number, then it also has bounded directed path-width \cite{Niedermeier2010ReflectionsOM}.

    \item Kelly-width is at most directed path-width plus one \cite{ganian2014digraph}.

    \item Entanglement is at most directed path-width \cite{rabinovich2008complexity}.

    \item Directed path-width is at most cycle rank \cite{gruber2012digraph}.

    \item Kelly path-width is equal to directed path-width \cite{kintali2015approximation}.

    \item Directed path-width is at most DAG-path-width \cite{kasahara2023dag}.

    \item A digraph has directed path-width at least \(k-1\) if and only if it has a diblockage of order at least \(k\) \cite{erde2020directed}.

    \item Directed path-width is at most \(2\) times directed path-distance-width minus \(1\) \cite{FujitaDirectedPath2024}.

    \item Directed proper-path-width satisfies
    \[
    \operatorname{dpw}(D)\leq \operatorname{dppw}(D)\leq \operatorname{dpw}(D)+1.
    \]
    \cite{FujitaDirectedPath2024}

    \item Directed tree-width is at most directed path-width \cite{gurski2019comparing}.

    \item Directed path-width is at most directed cut-width \cite{gurski2019comparing}.

    \item Directed path-width is bounded in terms of directed neighbourhood-width, directed linear NLC-width, and directed linear clique-width, with multiplicative dependence on \(\min(\Delta^-(D),\Delta^+(D))\) \cite{gurski2019comparing}.
\end{enumerate}
\end{theorem}

\begin{proof}
See the cited references.
\end{proof}

\section{Directed Branch-width}
Directed branch-width \cite{68} is the directed analogue of branch-width.

\section{Directed Cut-width}
Directed cut-width \cite{69,70} is the directed analogue of cut-width.  
For every digraph \(G\),
\[
\operatorname{dpw}(G)\leq \operatorname{dcutw}(G)
\]
holds \cite{69}.  
It is also known that directed NLC-width and directed clique-width satisfy inequalities analogous to the undirected case \cite{65}.

\begin{definition}[Directed cut-width]
\cite{chudnovsky2012tournament}
Let \(G=(V,E)\) be a digraph. The \emph{directed cut-width} of \(G\) is
\[
\operatorname{dcutw}(G)=
\min_{\varphi\in \Phi(G)}
\max_{1\leq i\leq |V|}
\left|
\{(u,v)\in E\mid u\in L(i,\varphi,G),\ v\in R(i,\varphi,G)\}
\}
\right|,
\]
where \(\Phi(G)\) denotes the set of linear layouts of \(G\).
\end{definition}

\section{Directed Clique-width}
Directed clique-width \cite{courcelle2000upper,65} is the directed analogue of clique-width.  
Its linear version is called directed linear clique-width \cite{65}.

\begin{definition}[Directed clique-width]
\label{def:dcw}
\cite{65}
Let \(k\) be a positive integer. The class \(\mathrm{dCW}_k\) of labeled digraphs is defined recursively as follows.
\begin{enumerate}
    \item A single-vertex digraph \(\bullet_a\) with label \(a\in [k]\) belongs to \(\mathrm{dCW}_k\).

    \item If \(G,H\in \mathrm{dCW}_k\) are vertex-disjoint labeled digraphs, then their disjoint union \(G\oplus H\) belongs to \(\mathrm{dCW}_k\).

    \item If \(G\in \mathrm{dCW}_k\) and \(a,b\in [k]\) with \(a\neq b\), then:
    \begin{enumerate}
        \item \(\rho_{a\to b}(G)\), obtained by relabeling every vertex of label \(a\) to label \(b\), belongs to \(\mathrm{dCW}_k\);
        \item \(\alpha_{a,b}(G)\), obtained by adding all arcs from vertices labeled \(a\) to vertices labeled \(b\), belongs to \(\mathrm{dCW}_k\).
    \end{enumerate}
\end{enumerate}

The \emph{directed clique-width} of a digraph \(G\), denoted by \(\operatorname{dcw}(G)\), is the minimum \(k\) such that \(G\) admits a labeling making it an element of \(\mathrm{dCW}_k\).
\end{definition}

\begin{definition}[Directed linear clique-width]
\cite{65}
The \emph{directed linear clique-width} of a digraph \(G\), denoted by \(\operatorname{dlcw}(G)\), is the minimum number of labels needed to construct \(G\) by repeatedly applying the following operations:
\begin{enumerate}
    \item create a new labeled vertex;
    \item take the disjoint union of the current digraph with a single new labeled vertex;
    \item add all arcs from label \(a\) to label \(b\), where \(a\neq b\);
    \item relabel \(a\) to \(b\).
\end{enumerate}
\end{definition}

\section{Directed Linear NLC-width}
Directed linear NLC-width \cite{65,69} is the linear restriction of directed NLC-width.

\begin{theorem}
The following statements are known.
\begin{enumerate}
    \item If a digraph \(G\) has bounded directed linear NLC-width, then it also has bounded directed NLC-width.

    \item For every digraph \(G\),
    \[
    \operatorname{dlnlcw}(G)\leq \operatorname{dlcw}(G)\leq \operatorname{dlnlcw}(G)+1.
    \]
    \cite{69}
\end{enumerate}
\end{theorem}

\begin{proof}
See the cited references.
\end{proof}

\section{Directed Neighbourhood Width}
Directed neighbourhood-width \cite{69} is the directed analogue of neighbourhood-width.

\begin{definition}[Directed neighbourhood-width]
\cite{69}
Let \(G=(V,E)\) be a digraph, and let \(U,W\subseteq V\) be disjoint vertex sets.  
For \(u\in U\), define
\[
N_W^+(u)=\{v\in W\mid (u,v)\in E\},
\qquad
N_W^-(u)=\{v\in W\mid (v,u)\in E\},
\]
and
\[
N_W(u):=(N_W^+(u),N_W^-(u)).
\]
Set
\[
N(U,W):=\{N_W(u)\mid u\in U\}.
\]

For a layout \(\varphi\in \Phi(G)\), define
\[
\operatorname{d\mbox{-}nw}(\varphi,G):=
\max_{1\leq i\leq |V|-1}
|N(L(i,\varphi,G),R(i,\varphi,G))|.
\]
The \emph{directed neighbourhood-width} of \(G\) is
\[
\operatorname{d\mbox{-}nw}(G):=
\min_{\varphi\in \Phi(G)} \operatorname{d\mbox{-}nw}(\varphi,G).
\]
\end{definition}

\begin{theorem}
For every digraph \(G\),
\begin{enumerate}
    \item
    \[
    \operatorname{d\mbox{-}nw}(G)\leq \operatorname{dlnlcw}(G)\leq \operatorname{d\mbox{-}nw}(G)+1;
    \]
    \item
    \[
    \operatorname{d\mbox{-}nw}(G)\leq \operatorname{dlcw}(G)\leq \operatorname{d\mbox{-}nw}(G)+1.
    \]
\end{enumerate}
\cite{69}
\end{theorem}

\begin{proof}
See the cited references.
\end{proof}

\section{Directed Rank-width}
Directed rank-width \cite{71} is the directed analogue of rank-width.

\section{Directed Linear Rank-width}
Directed linear rank-width \cite{69} is the linear restriction of directed rank-width.

\begin{definition}[Directed linear rank-width]
\cite{69}
Let \(G=(V,E)\) be a digraph, and let \(V_1,V_2\) be a partition of \(V\).  
Define the matrix \(M_{V_2}^{V_1}\) over \(\mathbb{GF}(4)\) by
\[
m_{ij}=
\begin{cases}
0 & \text{if neither } (v_i,v_j) \text{ nor } (v_j,v_i) \text{ is an arc},\\
a & \text{if } (v_i,v_j)\in E \text{ and } (v_j,v_i)\notin E,\\
a^2 & \text{if } (v_i,v_j)\notin E \text{ and } (v_j,v_i)\in E,\\
1 & \text{if both } (v_i,v_j)\in E \text{ and } (v_j,v_i)\in E,
\end{cases}
\]
where \(\mathbb{GF}(4)=\{0,1,a,a^2\}\) with \(1+a+a^2=0\) and \(a^3=1\).

A \emph{directed linear rank decomposition} of \(G\) is a pair \((T,f)\), where \(T\) is a caterpillar and \(f\) is a bijection from \(V\) to the leaves of \(T\).  
Each edge \(e\in E(T)\) induces a partition \((A_e,B_e)\) of \(V\).  
The \emph{width} of \(e\) is
\[
\operatorname{rank}_{\mathbb{GF}(4)}(M_{A_e}^{B_e}),
\]
and the \emph{width} of the decomposition is the maximum width over all edges of \(T\).

The \emph{directed linear rank-width} of \(G\), denoted by \(\operatorname{dlrw}(G)\), is the minimum width over all directed linear rank decompositions of \(G\).
\end{definition}

\begin{theorem}
The following statements are known.
\begin{enumerate}
    \item If a digraph \(G\) has bounded directed linear rank-width, then it also has bounded directed rank-width.

    \item For every digraph \(G\),
    \[
    \operatorname{dlrw}(G)\leq \operatorname{d\mbox{-}nw}(G).
    \]
    \cite{69}
\end{enumerate}
\end{theorem}

\begin{proof}
See the cited references.
\end{proof}

\section{Elimination Width}
Elimination width \cite{97} is a graph width parameter introduced to address issues related to acyclic digraphs.

\section{Flip-width}
Flip-width \cite{117} is a family of graph width parameters defined via variants of the Cops and Robber game.  
A related notion is radius-one flip-width \cite{344}.

\begin{theorem}
The following statements are known.
\begin{enumerate}
    \item Every graph class of bounded twin-width has bounded flip-width \cite{344}.

    \item If a graph \(G\) has bounded radius-one flip-width, then it also has bounded symmetric difference, bounded functionality, and bounded VC-dimension \cite{344}.
\end{enumerate}
\end{theorem}

\begin{proof}
See the cited references.
\end{proof}


\section{Stretch-width}
Stretch-width \cite{118} is a graph width parameter that lies strictly between clique-width and twin-width.

\begin{theorem}
The following statements are known.
\begin{enumerate}
    \item For every graph \(G\),
    \[
    \operatorname{stw}(G)\leq 2\,\operatorname{cwd}(G).
    \]
    \cite{118}

    \item There exists a constant \(c\) such that, for every graph \(G\),
    \[
    \operatorname{tw}(G)\leq c\,\Delta(G)^4\,\operatorname{stw}(G)^2 \log |V(G)|.
    \]
\end{enumerate}
\end{theorem}

\begin{proof}
See the cited references.
\end{proof}

\section{Cops-width}
Cop-width \cite{117,119,120} is a family of width parameters defined via variants of the Cops and Robber game.

\begin{theorem}
The following statements are known.
\begin{enumerate}
    \item For every graph \(G\),
    \[
    \operatorname{copw}_r(G)=\operatorname{tw}(G)+1.
    \]
    \cite{117,119,120}

    \item For every graph \(G\),
    \[
    \operatorname{copw}_1(G)=\operatorname{deg}(G)+1,
    \]
    where \(\operatorname{deg}(G)\) denotes the degeneracy of \(G\). \cite{345}
\end{enumerate}
\end{theorem}

\begin{proof}
See the cited references.
\end{proof}

\subsection{Related Concepts for Cop-width}
Related notions include the following.

\begin{itemize}
    \item \textbf{Marshal width} and \textbf{monotone marshal width} \cite{121}: width parameters based on winning strategies in Cops and Robber type games.
    \item \textbf{Game-width} \cite{216,327}: a graph width parameter arising from the Cops and Robber game. It is known that the game-width of a graph is equal to its tree-width plus \(1\) \cite{216}.
\end{itemize}

\section{Perfect Matching Width}
Perfect matching width \cite{142} is a width parameter for matching covered graphs, defined using branch decompositions.  
Recall that a \emph{matching covered graph} is a connected graph in which every edge belongs to some perfect matching \cite{szigeti2001generalizations,he2019perfect}.

\begin{definition}[Matching-porosity]
\cite{366}
Let \(G\) be a matching covered graph, and let \(X\subseteq V(G)\).  
Write \(\partial(X)\) for the cut determined by \(X\).  
The \emph{matching-porosity} of \(\partial(X)\) is defined by
\[
\operatorname{mp}(\partial(X))
:=\max_{M\in \mathcal{M}(G)} |M\cap \partial(X)|,
\]
where \(\mathcal{M}(G)\) denotes the set of perfect matchings of \(G\).

A perfect matching \(M\in\mathcal{M}(G)\) is said to be \emph{maximal with respect to} \(\partial(X)\) if there is no perfect matching \(M'\in\mathcal{M}(G)\) such that
\[
\partial(X)\cap M \subsetneq \partial(X)\cap M'.
\]
It \emph{maximizes} the cut \(\partial(X)\) if
\[
|M\cap \partial(X)|=\operatorname{mp}(\partial(X)).
\]
\end{definition}

\begin{definition}[Perfect matching width]
\cite{366}
Let \(G\) be a matching covered graph.  
A \emph{perfect matching decomposition} of \(G\) is a pair \((T,\delta)\), where \(T\) is a cubic tree and
\[
\delta:L(T)\to V(G)
\]
is a bijection from the leaves of \(T\) to the vertices of \(G\).

Each edge \(e\in E(T)\) induces a partition of \(V(G)\), and hence a cut \(\partial(e)\) in \(G\).  
The \emph{width} of \((T,\delta)\) is
\[
\max_{e\in E(T)} \operatorname{mp}(\partial(e)).
\]
The \emph{perfect matching width} of \(G\), denoted by \(\operatorname{pmw}(G)\), is the minimum of this quantity over all perfect matching decompositions of \(G\).
\end{definition}

\section{Bisection-width}
The \emph{bisection problem} asks for a partition of the vertex set of a graph into two parts of equal or nearly equal size while minimizing the number of crossing edges.  
The corresponding width parameter is the \emph{bisection-width} \cite{151,152,bezrukov2004new}.  
It is also known that bounded bandwidth implies bounded bisection-width \cite{341}.

\begin{definition}[Bisection-width]
\cite{bezrukov2004new}
Let \(G=(V,E)\) be an undirected graph with \(n:=|V|\).  
A \emph{bisection} of \(G\) is a partition
\[
V=V_0\cup V_1
\]
such that
\[
|V_0|,|V_1|\in \left\{\left\lfloor \frac{n}{2}\right\rfloor,\left\lceil \frac{n}{2}\right\rceil\right\}.
\]
The \emph{cut size} of the bisection is the number of edges with one endpoint in \(V_0\) and the other in \(V_1\).

The \emph{bisection-width} of \(G\), denoted by \(\operatorname{bisw}(G)\), is the minimum cut size over all bisections of \(G\), that is,
\[
\operatorname{bisw}(G)
:=
\min \bigl\{
|\{\{v,w\}\in E \mid v\in V_0,\ w\in V_1\}|
\;\big|\;
(V_0,V_1)\text{ is a bisection of }G
\bigr\}.
\]
\end{definition}

\section{Maximum Matching Width}

Maximum matching width is a graph width parameter introduced in the context of fast algorithms for domination-type problems \cite{153}.

\begin{definition}[Branch decomposition of a finite set]
Let \(X\) be a finite set.  
A \emph{branch decomposition} of \(X\) is a pair \((T,\delta)\), where \(T\) is a subcubic tree and
\[
\delta:L(T)\to X
\]
is a bijection.
Each edge \(e\in E(T)\) induces a bipartition \((A_e,B_e)\) of \(X\).
\end{definition}

\begin{definition}[Maximum matching width]
\cite{153}
Let \(G=(V,E)\) be a graph.  
For \(S\subseteq V\), let
\[
\operatorname{mm}_G(S)
\]
denote the size of a maximum matching in the bipartite graph
\[
G[S,\,V\setminus S].
\]

The \emph{maximum matching width} of a branch decomposition \((T,\delta)\) of \(V(G)\) is
\[
\max_{e\in E(T)} \operatorname{mm}_G(A_e),
\]
where \((A_e,B_e)\) is the partition induced by \(e\).

The \emph{maximum matching width} of \(G\), denoted by \(\operatorname{mmw}(G)\), is the minimum of this quantity over all branch decompositions of \(V(G)\).
\end{definition}

\begin{theorem}
The following statements are known.
\begin{enumerate}
    \item If a graph \(G\) has bounded maximum matching width, then it also has bounded tree-depth \cite{311}.

    \item If a graph \(G\) has bounded maximum matching width, then it also has bounded maximum induced matching width \cite{311}.

    \item For every graph \(G\),
    \[
    \operatorname{mmw}(G)\leq \operatorname{brw}(G)\leq \operatorname{tw}(G)+1\leq 3\,\operatorname{mmw}(G),
    \]
    where \(\operatorname{brw}(G)\) denotes the branch-width of \(G\).
\end{enumerate}
\end{theorem}

\begin{proof}
See the cited references.
\end{proof}

\section{Linear Maximum Matching Width}
Linear maximum matching width is the linear restriction of maximum matching width.

\begin{theorem}
The following statements are known.
\begin{enumerate}
    \item For every graph \(G\),
    \[
    \operatorname{lmmw}(G)\leq \operatorname{pw}(G)\leq 2\,\operatorname{lmmw}(G).
    \]
    \cite{330}

    \item If a graph \(G\) has bounded linear maximum matching width, then it also has bounded maximum matching width.
\end{enumerate}
\end{theorem}

\begin{proof}
See the cited references.
\end{proof}

\section{Clustering-width}
Clustering-width is a graph width parameter defined as the minimum number of variables whose deletion yields a variable-disjoint union of hitting formulas \cite{154}.

\section{Query-width}
Query-width is a graph width parameter that generalizes both acyclicity and tree-width by measuring the complexity of hypergraphs via restricted hypertree-width \cite{187,188}.

\section{Universal Width}
Universal width \cite{192,205} is a width parameter for alternating finite automata (AFAs).  
Recall that an alternating finite automaton is a computational model in which states may make existential or universal transitions.

Related notions include the following.

\begin{itemize}
    \item \textbf{Maximal universal width} \cite{192}
    \item \textbf{Combined width} \cite{204}
    \item \textbf{Maximal combined width} \cite{204}
    \item \textbf{Split-width} \cite{207,208}: a width parameter for alternating finite automata that also appears in the analysis of behavior graphs and decidability questions for computational systems \cite{cyriac2012mso,cyriac2014verification}
\end{itemize}

\section{Degree-width}
Degree-width \cite{351} is a width parameter for tournaments that can be viewed as a measure of how far a tournament is from being acyclic.

\section{Median-width}
Median-width \cite{354,355} is a graph width parameter whose decompositions use median graphs as the underlying structure.  
For every graph \(G\), it is known that
\[
\operatorname{mwid}(G)\leq \operatorname{tw}(G)+1.
\]
\cite{355}

For reference, we recall the definition of a median graph.

\begin{definition}[Median graph]
A graph \(G=(V,E)\) is called a \emph{median graph} if, for every triple of vertices \(a,b,c\in V\), there exists a unique vertex \(m(a,b,c)\in V\), called the \emph{median}, that lies simultaneously on a shortest \(a\)-\(b\) path, a shortest \(b\)-\(c\) path, and a shortest \(a\)-\(c\) path.

Equivalently,
\[
d_G(a,m(a,b,c))+d_G(m(a,b,c),b)=d_G(a,b),
\]
\[
d_G(b,m(a,b,c))+d_G(m(a,b,c),c)=d_G(b,c),
\]
and
\[
d_G(a,m(a,b,c))+d_G(m(a,b,c),c)=d_G(a,c),
\]
where \(d_G(x,y)\) denotes the graph distance between \(x\) and \(y\).
\end{definition}

\subsection{Related Concepts for Median-width}
A related notion is \emph{lattice-width} \cite{355}, where the underlying median graph of the decomposition is required to admit an isometric embedding into the Cartesian product of \(i\) paths.

\section{Directed Modular-width}
Directed modular-width \cite{356} is the directed analogue of modular-width.

\section{Cycle-width}
Cycle-width \cite{366,weincyclewidth} is a graph width parameter that emphasizes cycle structure and highlights differences between directed and undirected width theories through local cyclic behavior.

\begin{question}
Is there a corresponding width parameter for walk-width or circuit-width? If not, can one be defined naturally?
\end{question}

\subsection{Related Concepts for Cycle-width}
A related notion is the following.

\begin{itemize}
    \item \textbf{Radial cycle-width} \cite{albrechtsen2023structural}: a parameter defined via cycle decompositions in which the radial width is controlled by the largest radius of the parts.
\end{itemize}

\begin{question}
Is it possible to define a directed version of cycle-width by explicitly using directed cycles?
\end{question}

\begin{question}
Can cycle-width be extended naturally to fuzzy graphs or similar uncertainty-aware graph structures?
\end{question}

\begin{theorem}
The following statements are known.
\begin{enumerate}
    \item A graph \(G\) has cycle-width at most \(k-1\) if and only if its arc thickness is at most \(k\). \cite{weincyclewidth}

    \item For every graph \(G\),
    \[
    \frac{1}{2}\operatorname{pw}(G)\leq \operatorname{cyw}(G)\leq \operatorname{pw}(G).
    \]
    \cite{weincyclewidth}

    \item If \(T\) is a tree, then
    \[
    \operatorname{cyw}(T)=\operatorname{pw}(T).
    \]
    \cite{weincyclewidth}
\end{enumerate}
\end{theorem}

\begin{proof}
See the cited references.
\end{proof}

\section{Arc-width}
Arc-width \cite{404,405} is defined via arc representations of graphs on a circle.

\begin{definition}[Arc-width]
\cite{404,405}
An \emph{arc representation} of a graph \(G\) is a mapping
\[
\phi:V(G)\to \{\text{arcs on a fixed circle}\}
\]
such that two vertices of \(G\) are adjacent if and only if their corresponding arcs intersect.

For a point \(P\) on the base circle, the \emph{width of \(P\)} in the representation \(\phi\) is the number of arcs containing \(P\).  
The \emph{width} of \(\phi\) is the maximum width of a point on the circle.

The \emph{arc-width} of \(G\), denoted by \(\operatorname{arcw}(G)\), is the minimum width over all arc representations of \(G\).
\end{definition}

\subsection{Related Concepts for Arc-width}

A related notion is \emph{vortex-width} \cite{406,407}, which is equal to arc-width.

\begin{question}
Is it possible to define a directed version of arc-width by using directed arcs?
\end{question}

\begin{question}
Can arc-width be extended naturally to fuzzy graphs or related generalized graph structures?
\end{question}

\section{Match Width and Braid Width}

Match width and braid width \cite{420} are structural parameters used in the study of XML documents, where they help determine XPath satisfiability in terms of document depth.

\section{Clique Cover Width}

The \emph{clique cover width} of a graph \(G\), denoted by \(\operatorname{ccw}(G)\), is the minimum bandwidth among all graphs obtained by contracting each clique in a clique cover of \(G\) into a single vertex \cite{shahrokhi2015clique,shahrokhi2017unit,shahrokhi2015bounds}.

\section{Questionable-width}

Questionable-width \cite{lyaudet2019finite,lyaudet2020first} is a width notion arising from ``questionable'' representations of orders.  
A \emph{question} records the first divergence between two sequences indexed by an ordinal, and the corresponding width measures how such comparisons can be resolved within finite structures \cite{lyaudet2019finite,lyaudet2020first}.

\section{Plane-width}

The \emph{plane-width} of a graph is the minimum possible diameter of the image of its vertex set in a plane representation \cite{Kaminski2008ThePO,Kaminski2008ThePO}.

\section{Scan-width}

Scan-width \cite{holtgrefe2024exact,holtgrefe2023computing} is a width parameter for directed acyclic graphs that measures tree-likeness while respecting arc directions \cite{scornavacca2022treewidth}.

\section{V-width}

V-width is a complexity measure for graph structures that can be substantially smaller than both the maximum clique size and the maximum number of parents in a DAG \cite{chickering2015selective}.

\section{Cross-width}

Cross-width is a graph parameter used in kernelization, especially in the analysis of polynomial equivalence relations and sparsification phenomena \cite{jansen2015sparsification}.

\section{Guidance-width}

Guidance-width measures the minimum number of colors needed to color a guidance system in a tree so that conflicting guides receive different colors \cite{bojanczyk2012extension}.

\section{Star-width}

Star-width measures the complexity of decomposing a graph into star-like pieces, where the width depends on the sizes of the interfaces connecting those pieces \cite{blume2011treewidth,van2017some,van2017routing}.  
Related notions include radial star-width \cite{albrechtsen2023structural}, costar-width \cite{blume2011treewidth}, and atomic star-width \cite{blume2011treewidth}.  
It is known that bounded star-width implies bounded tree-depth \cite{van2017routing}.

\begin{definition}[Star decomposition and star-width]
\cite{van2017some,van2017routing}
A \emph{star decomposition} of a graph \(G=(V,E)\) consists of a family of subsets
\[
X_0,X_1,\dots,X_m\subseteq V,
\]
where \(X_0\) is the \emph{core} and \(X_1,\dots,X_m\) are the \emph{leaves}, such that:
\begin{enumerate}
    \item
    \[
    \bigcup_{i=0}^m X_i = V;
    \]
    \item for every edge \(vw\in E\), there exists some \(i\in \{0,1,\dots,m\}\) such that
    \[
    v,w\in X_i;
    \]
    \item if \(i\neq j\) and a vertex \(v\) belongs to both \(X_i\) and \(X_j\), then
    \[
    v\in X_0.
    \]
\end{enumerate}

The \emph{width} of the star decomposition is
\[
\max_{0\leq i\leq m} |X_i|-1.
\]
The \emph{star-width} of \(G\), denoted by \(\operatorname{sw}(G)\), is the minimum width over all star decompositions of \(G\).
\end{definition}

\begin{question}
Is it possible to define a directed version of star-width?
\end{question}

\begin{question}
Can star-width be extended naturally to fuzzy graphs or similar generalized graph structures?
\end{question}

\newpage
  
\chapter{Various Length Parameters}

This chapter surveys several graph parameters of \emph{length type}.  
Unlike width parameters, which typically measure the size of separators, bags, or interfaces, length parameters usually measure metric spread inside the parts of a decomposition.  
Among them, tree-length is one of the most fundamental examples.

\section{Tree-length, Tree-breadth, and Branch-length}
We begin with tree-length and some closely related parameters.  
Tree-length measures the maximum distance between two vertices that appear together in a bag of a tree-decomposition, minimized over all tree-decompositions of the graph \cite{coudert2016approximate,dourisboure2004compact,
dourisboure2007spanners}.

\begin{definition}[Tree-length]
\cite{coudert2016approximate,dourisboure2004compact,
dourisboure2007spanners}
Let
\[
\mathcal{T}=\bigl(\{X_i\mid i\in I\},\,T=(I,F)\bigr)
\]
be a tree-decomposition of a graph \(G\).  
The \emph{length} of \(\mathcal{T}\) is defined by
\[
\lambda(\mathcal{T})
:=
\max_{i\in I}\ \max_{u,v\in X_i} d_G(u,v),
\]
where \(d_G(u,v)\) denotes the distance between \(u\) and \(v\) in \(G\).

The \emph{tree-length} of \(G\), denoted by \(\tl(G)\), is the minimum value of \(\lambda(\mathcal{T})\) over all tree-decompositions \(\mathcal{T}\) of \(G\).
\end{definition}

Chordal graphs have tree-length \(1\).
A closely related parameter is \emph{tree-breadth}, which measures how well each bag can be covered by a disk of small radius in the original graph \cite{leitert2016strong,leitert2017tree}.

\begin{definition}[Tree-breadth]
\cite{leitert2016strong,leitert2017tree}
Let
\[
\mathcal{T}=\bigl(\{X_i\mid i\in I\},\,T=(I,F)\bigr)
\]
be a tree-decomposition of a graph \(G\).  
The \emph{breadth} of \(\mathcal{T}\) is the smallest integer \(r\) such that, for every \(i\in I\), there exists a vertex \(v_i\in V(G)\) with
\[
X_i\subseteq D_r(v_i,G),
\]
where
\[
D_r(v_i,G):=\{u\in V(G)\mid d_G(u,v_i)\le r\}
\]
is the closed disk of radius \(r\) centered at \(v_i\).

The \emph{tree-breadth} of \(G\), denoted by \(\tb(G)\), is the minimum breadth over all tree-decompositions of \(G\).
\end{definition}

It is known that, for every graph \(G\),
\[
1\le \tb(G)\le \tl(G)\le 2\,\tb(G).
\]

Another closely related notion is \emph{branch-length}, which is defined via branch-decompositions and is equivalent to tree-length \cite{umezawa2009tree}.  
Likewise, linear-length is equivalent to path-length \cite{umezawa2009tree}.

\begin{definition}[Branch-length]
\cite{umezawa2009tree}
Let \(B=(T,\mu)\) be a branch-decomposition of a graph \(G\).  
For each edge \(e\in E(T)\), let \(\operatorname{mid}_B(e)\) denote the middle set associated with \(e\).  
The \emph{branch-length} of \(B\) is
\[
\operatorname{bl}(B)
:=
\max_{e\in E(T)} \operatorname{diam}_G\!\bigl(\operatorname{mid}_B(e)\bigr),
\]
where \(\operatorname{diam}_G(S)\) denotes the diameter of \(S\) in \(G\).

The \emph{branch-length} of \(G\), denoted by \(\operatorname{bl}(G)\), is the minimum value of \(\operatorname{bl}(B)\) over all branch-decompositions \(B\) of \(G\).
\end{definition}

Other related notions include \emph{tree-stretch} and \emph{tree-distortion}.  
For convenience, we collect several known inequalities below.

\begin{theorem}
\label{thm:length-parameter-relations}
Let \(G=(V,E)\) be a connected undirected unweighted graph on \(n\) vertices, and let \(s\in V\).  
Then the following inequalities are known.
\begin{enumerate}
    \item The tree-breadth and tree-length satisfy
    \[
    \tb(G)\le \tl(G)\le 2\,\tb(G).
    \]

    \item The cluster-radius \(R_s(G)\) and cluster-diameter \(\Delta_s(G)\) satisfy
    \[
    R_s(G)\le \Delta_s(G)\le 2\,R_s(G).
    \]

    \item The hyperbolicity \(\operatorname{hb}(G)\) satisfies
    \[
    \operatorname{hb}(G)\le \tl(G)\le O\!\bigl(\operatorname{hb}(G)\log n\bigr),
    \]
    and also
    \[
    \operatorname{hb}(G)\le \Delta_s(G)\le O\!\bigl(\operatorname{hb}(G)\log n\bigr).
    \]
    \cite{chepoi2008diameters,chepoi2012additive}

    \item The tree-stretch \(\operatorname{ts}(G)\), tree-distortion \(\operatorname{td}(G)\), and cluster-diameter satisfy
    \[
    \operatorname{ts}(G)\ge \operatorname{td}(G)\ge \frac{1}{3}\Delta_s(G),
    \]
    and
    \[
    \operatorname{td}(G)\le 2\,\Delta_s(G)+2.
    \]
    \cite{chepoi2012constant}

    \item The cluster-radius is bounded in terms of tree-distortion:
    \[
    R_s(G)\le \max\{3\,\operatorname{td}(G)-1,\ 2\,\operatorname{td}(G)+1\}.
    \]
    \cite{chepoi2012constant}

    \item Tree-length, cluster-diameter, and cluster-radius satisfy
    \[
    \tl(G)-1\le \Delta_s(G)\le 3\,\tl(G),
    \]
    and
    \[
    R_s(G)\le 2\,\tl(G).
    \]
    \cite{dourisboure2007tree,dourisboure2007spanners}

    \item Tree-breadth and cluster-radius satisfy
    \[
    \tb(G)-1\le R_s(G)\le 3\,\tb(G).
    \]
    \cite{dragan2014approximation}

    \item Tree-length, tree-distortion, and tree-stretch satisfy
    \[
    \tl(G)\le \operatorname{td}(G)\le \operatorname{ts}(G),
    \]
    and also
    \[
    \tb(G)\le \frac{\operatorname{ts}(G)}{2}.
    \]
    \cite{dragan2014approximation}

    \item Tree-stretch is bounded by tree-breadth and tree-distortion:
    \[
    \operatorname{ts}(G)\le 2\,\tb(G)\log_2 n,
    \]
    and
    \[
    \operatorname{ts}(G)\le 2\,\operatorname{td}(G)\log_2 n.
    \]
    \cite{dragan2014approximation}

    \item Slimness and thinness satisfy
    \[
    \operatorname{slim}(G)\le 3\,\tb(G),
    \qquad
    \operatorname{thin}(G)\le 6\,\tb(G).
    \]
    \cite{diestel2018connected,mohammed2019slimness}

    \item Slimness and thinness also satisfy
    \[
    \operatorname{slim}(G)\le \left\lfloor \frac{3}{2}\,\tl(G)\right\rfloor,
    \qquad
    \operatorname{thin}(G)\le 3\,\tl(G).
    \]
    \cite{mohammed2019slimness}

    \item Branch-length is equivalent to tree-length. \cite{umezawa2009tree}
\end{enumerate}
\end{theorem}

\begin{proof}
See the cited references.
\end{proof}

\section{Possible Future Directions for Length Parameters}
We next list several possible length-type parameters that may be worth studying in the future.  
The following notions are intended only as preliminary directions rather than fully established definitions.  
In particular, for some decomposition frameworks, the phrase ``diameter of a bag'' is natural, whereas for others a more appropriate formulation may involve separators, middle sets, or induced boundary structures.

\begin{itemize}
    \item \textbf{Directed tree-length.}  
    A possible directed analogue of tree-length based on directed tree-decompositions and directed distance.

    \item \textbf{Directed path-length.}  
    A linear analogue of directed tree-length based on directed path-decompositions.

    \item \textbf{Rank-length.}  
    A length-type parameter associated with rank-decompositions, possibly defined through metric properties of the vertex parts induced by decomposition edges.

    \item \textbf{Linear rank-length.}  
    A linear version of rank-length.

    \item \textbf{Boolean-length.}  
    A possible length-type refinement of Boolean decompositions.

    \item \textbf{Linear Boolean-length.}  
    A linear restriction of Boolean-length.

    \item \textbf{Arc-length.}  
    A length-type parameter derived from arc-based decompositions or circular representations.

    \item \textbf{Star-length.}  
    A length analogue of star decompositions, defined via the largest diameter among the star parts.

    \item \textbf{Hypertree-length.}  
    A length-type version of hypertree decompositions for hypergraphs.

    \item \textbf{Carving-length.}  
    A possible length parameter associated with carving decompositions.

    \item \textbf{Cut-length.}  
    A length parameter based on cut decompositions.

    \item \textbf{Tree-cut-length.}  
    A length-type refinement of tree-cut decompositions.

    \item \textbf{Tree-distance-length.}  
    A length version of tree-distance decompositions.

    \item \textbf{Path-distance-length.}  
    A linear version of tree-distance-length.

    \item \textbf{Clique-length.}  
    A possible length-type refinement of clique-width decompositions.

    \item \textbf{Linear clique-length.}  
    A linear restriction of clique-length.

    \item \textbf{NLC-length.}  
    A length-type parameter corresponding to NLC decompositions.

    \item \textbf{Linear NLC-length.}  
    A linear restriction of NLC-length.

    \item \textbf{Modular-length.}  
    A length-type parameter associated with modular decompositions.
\end{itemize}

A systematic study of such parameters may help clarify how metric information interacts with decomposition theory beyond width-type invariants.

\newpage
  
\chapter{Comparing Various Graph Parameters (More Than 70 Parameters)}

In this chapter, we compare a wide range of graph parameters and examine the relationships among them.  
For a broader visual summary, see the supplemental file entitled
\textit{“Supplemental Figure: Comparing Graph Width Parameters.”}
Such comparisons are useful in structural graph theory and can also help guide the design of algorithms \cite{FujitaFigure2024}.

\url{https://www.researchgate.net/publication/383432866_Supplemental_Figure_Comparing_Graph_Width_Parameter_Graph_Parameter_Hierarchy}

\begin{note}
Unless stated otherwise, all comparisons in this chapter are made over arbitrary graphs.  
A graph parameter \(p\) is said to \emph{upper-bound} a graph parameter \(q\) if there exists a non-decreasing function \(f\) such that
\[
q(G)\leq f\bigl(p(G)\bigr)
\qquad
\text{for every graph } G.
\]
If no such function exists, then \(q\) is said to be \emph{unbounded in terms of} \(p\).  
Moreover, upper-bound relations are transitive: if \(p\) upper-bounds \(q\) and \(q\) upper-bounds \(r\), then \(p\) upper-bounds \(r\).
\end{note}

\begin{note}
Comparing graph parameters means studying how different measures of graph complexity---such as tree-width, path-width, clique-width, and related notions---are related to one another.  
These comparisons help clarify which parameters are stronger, weaker, or incomparable, and this in turn is useful when selecting appropriate methods in graph theory and algorithm design.
\end{note}


\newpage

\chapter{Various Graph Depth Parameters}

Many graph depth parameters have been introduced in the literature.  
Studying the relationships among them---including inequalities, upper bounds, lower bounds, and equivalences---is a standard theme in structural graph theory.

\section{Tree-depth}

Tree-depth is a graph depth parameter that has appeared under several names and serves as a measure of graph sparsity and hierarchical structure \cite{90,92,172,173,174}.

\begin{definition}[Tree-depth]
Let \(F\) be a rooted forest.  
The \emph{closure} of \(F\), denoted by \(\operatorname{clos}(F)\), is the undirected graph on \(V(F)\) in which two distinct vertices are adjacent whenever one is an ancestor of the other in \(F\).

The \emph{height} of \(F\) is the maximum number of vertices on a root-to-vertex path in \(F\).

The \emph{tree-depth} of a graph \(G\), denoted by \(\operatorname{td}(G)\), is the minimum height of a rooted forest \(F\) such that
\[
G \subseteq \operatorname{clos}(F).
\]
\end{definition}

A related parameter is \emph{block treedepth} \cite{Giannopoulou2024AGS}.

\subsection{Selected Relations for Tree-depth}

\begin{theorem}
The following statements are known.
\begin{enumerate}
    \item If a graph \(G\) has bounded tree-depth, then it also has bounded path-width \cite{311}.

    \item If a graph \(G\) has bounded tree-depth, then it also has bounded tree-width \cite{311}.

    \item If a graph \(G\) has bounded vertex cover number, then it also has bounded tree-depth \cite{Hanaka2024CoreSI}.

    \item If a graph \(G\) has bounded vertex integrity, then it also has bounded tree-depth \cite{Hanaka2024CoreSI}.

    \item For every graph \(G\), the tree-depth of \(G\) is equal to the minimum number of colors in a centered coloring of \(G\) \cite{Nesetril2006TreedepthSC}.

    \item The branch-depth of a connected graph is at most its tree-depth \cite{91}.

    \item Let \(G\) be a connected graph, let \(k=\operatorname{bd}(G)\), and let \(t=\operatorname{td}(G)\). Then
    \[
    k-1 \leq t \leq \max(2k^2-k+1,\,2).
    \]
    \cite{91}

    \item If \(G\) has tree-depth at most \(k\), then \(G\) has shrub-depth at most \(k\) \cite{350}.
\end{enumerate}
\end{theorem}

\begin{proof}
See the cited references.
\end{proof}

\section{Branch-depth}

Branch-depth is a depth parameter that generalizes tree-depth through decompositions of connectivity functions \cite{91}.  
Recall that the \emph{radius} of a tree is the minimum integer \(r\) such that some vertex of the tree is at distance at most \(r\) from every other vertex.

\begin{definition}[Branch-depth of a connectivity function]
\cite{91}
Let \(\lambda\) be a connectivity function on a finite set \(V\).  
A \emph{decomposition} of \(\lambda\) is a pair \((T,\phi)\), where \(T\) is a tree with at least one internal node and
\[
\phi:V \to L(T)
\]
is a bijection from \(V\) to the set of leaves of \(T\).

For an internal node \(x\in V(T)\), deleting \(x\) from \(T\) yields a partition
\[
\mathcal{P}_x=\{V_1,\dots,V_m\}
\]
of \(V\), where each \(V_i\) corresponds to the set of leaves in one component of \(T-x\).  
The \emph{\(\lambda\)-width} of \(x\) is defined by
\[
\lambda(\mathcal{P}_x)
:=
\max\Bigl\{
\lambda\!\Bigl(\bigcup_{i\in I} V_i\Bigr)
\;\Big|\;
\emptyset \neq I \subsetneq \{1,\dots,m\}
\Bigr\}.
\]
The \emph{width} of the decomposition \((T,\phi)\) is the maximum \(\lambda\)-width over all internal nodes of \(T\), and its \emph{radius} is the radius of the tree \(T\).

For a positive integer \(k\), we say that \((T,\phi)\) is a \emph{\((k,k)\)-decomposition} if both its width and its radius are at most \(k\).  
The \emph{branch-depth} of \(\lambda\), denoted by \(\operatorname{bd}(\lambda)\), is the least integer \(k\) for which \(\lambda\) admits a \((k,k)\)-decomposition.
\end{definition}

When \(G\) is a graph, its \emph{branch-depth} is defined as the branch-depth of the standard connectivity function associated with \(G\) \cite{91}.

\subsection{Selected Relations for Branch-depth}

\begin{theorem}
The following statements are known.
\begin{enumerate}
    \item Branch-width is at most branch-depth \cite{91}.

    \item The branch-depth of a connected graph is at most its tree-depth \cite{91}.

    \item Let \(G\) be a connected graph, let \(k=\operatorname{bd}(G)\), and let \(t=\operatorname{td}(G)\). Then
    \[
    k-1 \leq t \leq \max(2k^2-k+1,\,2).
    \]
    \cite{91}

    \item For every matroid \(M\),
    \[
    \operatorname{bd}(M)\leq \operatorname{cdd}(M)
    \leq \min\{\operatorname{cd}(M),\operatorname{dd}(M)\},
    \]
    where \(\operatorname{cdd}(M)\), \(\operatorname{cd}(M)\), and \(\operatorname{dd}(M)\) denote contraction-deletion depth, contraction depth, and deletion depth, respectively \cite{91}.

    \item Let \(t=\operatorname{td}(G)\). Then
    \[
    \operatorname{bd}(M(G)) \leq \operatorname{bd}(G)-1 \leq t,
    \]
    where \(M(G)\) denotes the cycle matroid of \(G\) \cite{91}.
\end{enumerate}
\end{theorem}

\begin{proof}
See the cited references.
\end{proof}

\subsection{Related Parameters for Branch-depth}

Related depth parameters include \emph{deletion-depth}, \emph{contraction-depth}, \emph{contraction-deletion depth} \cite{91}, \emph{contraction\(^*\)-depth}, and \emph{contraction\(^*\)-deletion depth} \cite{178}.

\section{Rank-depth}

Rank-depth is a graph depth parameter defined as the branch-depth of the cut-rank function \cite{91}.  
A class of simple graphs has bounded rank-depth if and only if it has bounded shrub-depth \cite{91}.

\begin{definition}[Cut-rank function and rank-depth]
\cite{91}
Let \(G=(V,E)\) be a simple graph.  
For \(X\subseteq V\), let \(\rho_G(X)\) be the rank over \(\mathbb{F}_2\) of the \(X\times (V\setminus X)\) adjacency matrix between \(X\) and \(V\setminus X\).  
The function \(\rho_G\) is called the \emph{cut-rank function} of \(G\).

The \emph{rank-depth} of \(G\), denoted by \(\operatorname{rd}(G)\), is the branch-depth of the connectivity function \(\rho_G\).
\end{definition}

\section{Shrub-depth}

Shrub-depth is a graph depth parameter intended to capture the height of dense graph classes \cite{168,170}.  
A related parameter is SC-depth \cite{168,170}.

\subsection{Selected Relations for Shrub-depth}

\begin{theorem}
The following statements are known.
\begin{enumerate}
    \item If \(G\) has tree-depth at most \(k\), then \(G\) has shrub-depth at most \(k\) \cite{350}.

    \item If a graph class has bounded shrub-depth, then it also has bounded clique-width \cite{350}.

    \item If a graph \(G\) has bounded twin-cover number, then it also has bounded shrub-depth \cite{343}.
\end{enumerate}
\end{theorem}

\begin{proof}
See the cited references.
\end{proof}

\section{Directed Tree-depth}

Directed tree-depth is a directed analogue of tree-depth \cite{169}.  
A related parameter is \emph{DAG-depth}, which is a structural depth measure for digraphs \cite{388,Ganian2014DigraphWM}.

\begin{definition}[DAG-depth]
\cite{388,Ganian2014DigraphWM}
Let \(D\) be a digraph, and let \(R_1,\dots,R_p\) be its reachable fragments.  
The \emph{DAG-depth} of \(D\), denoted by \(\operatorname{ddp}(D)\), is defined inductively by
\[
\operatorname{ddp}(D)=
\begin{cases}
1, & \text{if } |V(D)|=1,\\[4pt]
1+\displaystyle\min_{v\in V(D)} \operatorname{ddp}(D-v),
& \text{if } p=1 \text{ and } |V(D)|>1,\\[8pt]
\displaystyle\max_{1\le i\le p}\operatorname{ddp}(R_i),
& \text{otherwise}.
\end{cases}
\]
\end{definition}


\newpage
\chapter{Graph Games Related to Width Parameters}

As discussed throughout this book, many graph width parameters admit natural game-theoretic interpretations.  
Such interpretations are often useful for developing intuition about width parameters, their obstructions, and their algorithmic significance.  
In this chapter, we briefly review several well-known games that are closely connected with graph width theory.

\section{Cops and Robbers (Pursuit--Evasion Games)}

Cops and Robbers is a classical pursuit--evasion game played on a graph, in which a team of cops attempts to capture a robber.  
Many variants have been studied, including versions with different visibility assumptions, movement rules, and monotonicity restrictions \cite{16,kosowski2015k,christophe2019connected}.  
One standard example is the following jump-searching version.

\begin{definition}[Helicopter Cops and Robber Game]
\cite{seymour1993graph}
Let \(G\) be a graph, and let \(X\subseteq V(G)\).  
An \emph{\(X\)-flap} is the vertex set of a connected component of \(G-X\).

Two vertex sets \(X,Y\subseteq V(G)\) are said to \emph{touch} if
\[
N(X)\cap Y\neq \emptyset.
\]

A \emph{position} in the game is a pair \((X,R)\), where \(X\subseteq V(G)\) represents the set of vertices occupied by the cops, and \(R\) is an \(X\)-flap representing the current location of the robber.  
Since the robber may move arbitrarily fast inside \(G-X\), only the component containing the robber is relevant.

The game proceeds as follows.
\begin{itemize}
    \item Initially, the cops choose a set \(X_0\subseteq V(G)\) with \(|X_0|\leq k\), where \(k\) is the number of cops.
    \item The robber then chooses an \(X_0\)-flap \(R_0\).
    \item At round \(i\geq 1\), if the current position is \((X_{i-1},R_{i-1})\), then the cops choose a new set \(X_i\subseteq V(G)\) with \(|X_i|\leq k\).
    \item After seeing \(X_i\), the robber chooses an \(X_i\)-flap \(R_i\) that touches \(R_{i-1}\).
\end{itemize}

If at some stage the robber has no legal choice of \(R_i\), then the cops win.  
Otherwise, if the robber can continue indefinitely, then the robber wins.
\end{definition}

\section{Cops and Invisible Robbers}

The \emph{Cops and Invisible Robber} game is a pursuit--evasion game in which the cops attempt to capture a robber whose location is hidden from them \cite{dereniowski2015zero,kehagias2013cops,richerby2009graph,clarke2020limited}.  
Variants of this game are closely related to parameters such as path-width and search numbers.

\begin{definition}[Zero-Visibility Cops and Robber Game]
\cite{dereniowski2015zero}
The zero-visibility cops and robber game is played on a connected graph \(G\) by two players: the \emph{cop player} and the \emph{robber player}.  
The cop player controls \(k\) cops, while the robber player controls one robber.

The game is played according to the following rules.
\begin{itemize}
    \item \textbf{Setup.} The cops are first placed on vertices of \(G\). The robber then chooses a starting vertex, which is not revealed to the cops.

    \item \textbf{Moves.} The players move alternately, with the cops moving first. On each turn, the cop player may move one or more cops along edges of \(G\), and the robber may move along an edge to an adjacent vertex.

    \item \textbf{Visibility.} The robber sees all cop moves, but the cops do not know the robber's position.

    \item \textbf{Winning condition.} The cops win if at some point a cop occupies the same vertex as the robber. The robber wins if capture never occurs.
\end{itemize}

A \emph{strategy} for the cops specifies how the cops move at each stage of the game.  
The \emph{zero-visibility cop number} of \(G\), denoted by \(c_0(G)\), is the minimum number of cops required to guarantee capture.
\end{definition}

\section{Parity Games}

A parity game is a two-player infinite-duration game played on a directed graph.  
Such games have been studied extensively, especially in connection with directed width parameters and model checking \cite{berwanger2006dag,staniszewski2022parity,puchala2011graph}.

\begin{definition}[Parity Game]
\cite{berwanger2006dag}
A \emph{parity game} is a tuple
\[
P=(V,V_0,E,\Omega),
\]
where
\begin{itemize}
    \item \(V\) is a finite set of vertices,
    \item \(V_0\subseteq V\) is the set of vertices controlled by player \emph{Even},
    \item \(V\setminus V_0\) is the set of vertices controlled by player \emph{Odd},
    \item \(E\subseteq V\times V\) is the set of directed edges,
    \item \(\Omega:V\to \omega\) is a priority function.
\end{itemize}

A \emph{play} of the game is an infinite sequence
\[
\pi=(v_0,v_1,v_2,\dots)
\]
such that \((v_i,v_{i+1})\in E\) for all \(i\geq 0\).

The winner is determined by the priorities appearing infinitely often along the play:
\begin{itemize}
    \item \(\pi\) is winning for \emph{Even} if the minimum priority appearing infinitely often is even;
    \item otherwise, \(\pi\) is winning for \emph{Odd}.
\end{itemize}

A \emph{strategy} for a player specifies the next move whenever the token is at a vertex controlled by that player.  
A strategy is called \emph{memoryless} if it depends only on the current vertex.  
Parity games are determined, and moreover both players admit memoryless optimal strategies.
\end{definition}

\section{Spanning-Tree Games}

A spanning tree of a graph is a connected acyclic spanning subgraph.  
Games based on spanning trees have been studied from combinatorial, algorithmic, and game-theoretic viewpoints \cite{hefetz2018spanning,nebel2010shortest,cardinal2013stackelberg}.  
These games are also related to graph structure, since width parameters often influence the complexity of constructing or optimizing spanning trees.

\begin{definition}[Spanning-Tree Game]
\cite{cardinal2013stackelberg}
Let \(G=(V,E,w)\) be a weighted graph, where \(w:E\to \mathbb{R}_{\geq 0}\) is a nonnegative edge-weight function.

A \emph{configuration} of the game is a forest \(F\subseteq E\).  
A move is \emph{legal} from \(F\) if an edge \(e\in E\setminus F\) is added such that
\[
(V,F\cup\{e\})
\]
is still acyclic.  
Thus the set of legal moves is
\[
M(F):=\{e\in E\setminus F \mid (V,F\cup\{e\}) \text{ is a forest}\}.
\]

There are two players, \emph{Max} and \emph{Min}, who alternate turns, with Max moving first.
\begin{itemize}
    \item Max aims to maximize the total weight of the final spanning tree.
    \item Min aims to minimize that total weight.
\end{itemize}

A \emph{strategy} for a player is a rule that selects a legal move from every configuration.  
If both players follow strategies \(\pi_{\max}\) and \(\pi_{\min}\), the resulting spanning tree is denoted by
\[
T(\pi_{\max},\pi_{\min}),
\]
and its total weight is
\[
w(\pi_{\max},\pi_{\min})
=
\sum_{e\in T(\pi_{\max},\pi_{\min})} w(e).
\]
\end{definition}

\section{Mixed Search Games}

Mixed search games model the process of clearing a contaminated graph by placing and moving searchers \cite{bienstock1991monotonicity,best2016contraction,heggernes2008mixed,fomin2007mixed,takahashi1995mixed}.  
They are closely related to search numbers and several width parameters.  
Related games include the edge search game and the node search game \cite{giannopoulou2011min}.

\begin{definition}[Mixed Search Game]
Let \(G=(V,E)\) be a graph.  
Initially, every edge of \(G\) is considered \emph{contaminated}.  
The goal is to clear all edges using searchers.

An edge \(e\in E\) becomes \emph{cleared} if either
\begin{enumerate}
    \item searchers occupy both endpoints of \(e\), or
    \item a searcher slides along \(e\).
\end{enumerate}

A cleared edge becomes \emph{recontaminated} if there exists a path from a contaminated edge to that edge such that no vertex or edge on the path is occupied by a searcher.

A \emph{mixed search strategy} is a sequence of operations chosen from the following types:
\begin{enumerate}
    \item placing a searcher on a vertex,
    \item removing a searcher from a vertex,
    \item sliding a searcher along an edge to an adjacent vertex,
    \item sliding a searcher along an edge in such a way that the edge is cleared during the move.
\end{enumerate}

The objective is to clear all edges of \(G\) while preventing recontamination.  
The main optimization problem is to minimize the number of searchers required.
\end{definition}

\newpage

\chapter*{Disclaimer}
\section*{Funding}
This study was conducted without any financial support from external organizations or grants.

\section*{Acknowledgments}
We would like to express our sincere gratitude to everyone who provided valuable insights, support, and encouragement throughout this research. We also extend our thanks to the readers for their interest and to the authors of the referenced works, whose scholarly contributions have greatly influenced this study. Lastly, we are deeply grateful to the publishers and reviewers who facilitated the dissemination of this work.

\section*{Data Availability}
Since this research is purely theoretical and mathematical, no empirical data or computational analysis was utilized. Researchers are encouraged to expand upon these findings with data-oriented or experimental approaches in future studies.

\section*{Ethical Statement}
As this study does not involve experiments with human participants or animals, no ethical approval was required.

\section*{Conflicts of Interest}
The authors declare that they have no conflicts of interest related to the content or publication of this book.

\section*{Code Availability}
No code or software was developed for this study.

\section*{Use of Generative AI and AI-Assisted Tools}
I use generative AI and AI-assisted tools for tasks such as English grammar checking, and I do not employ them in any way that violates ethical standards.

\section*{Disclaimer (Others)}
This work presents theoretical ideas and frameworks that have not yet been empirically validated. Readers are encouraged to explore practical applications and further refine these concepts. Although care has been taken to ensure accuracy and appropriate citations, any errors or oversights are unintentional. The perspectives and interpretations expressed herein are solely those of the authors and do not necessarily reflect the viewpoints of their affiliated institutions.

\newpage

{%
\footnotesize
\setlength{\parskip}{0pt}%
\setlength{\itemsep}{0pt}%
\setlength{\parsep}{0pt}%
\setlength{\topsep}{0pt}%
\bibliographystyle{unsrt}
\bibliography{Inapproxi}%
}

\newpage

\thispagestyle{empty}
\pagecolor{blue!65!black}

\begin{tikzpicture}[remember picture,overlay]
    \fill[white,opacity=0.045]
      ([xshift=2.0mm]current page.south west)
      rectangle
      ([xshift=-2.0mm]current page.north east);

    \draw[line width=0.35pt, white!70]
      ([xshift=0.75mm,yshift=2.5mm]current page.south west) --
      ([xshift=0.75mm,yshift=-2.5mm]current page.north west);

    \draw[line width=0.35pt, white!70]
      ([xshift=-0.75mm,yshift=2.5mm]current page.south east) --
      ([xshift=-0.75mm,yshift=-2.5mm]current page.north east);
\end{tikzpicture}

\null
\clearpage
\pagecolor{white}

\end{document}